\documentclass[a4paper]{amsart}                % American Mathematical Society document class for articles
\usepackage[latin1]{inputenc}                  % Codificacion europea del teclado (acentos, simbolos teclado ...)
\usepackage[T1]{fontenc}                       % Acentos de palabras y division silabica
\usepackage[spanish,english]{babel}            % Idioma principal el ingles
\usepackage[pdftex]{graphicx}                  % Incorporar graficos
\usepackage{amsmath,amssymb,amsthm}            % Paquetes de la American Mathematical Society
\usepackage{latexsym}                          % Simbolos
\usepackage{delarray}                          % Delimiter array: permite integrar delimitadores en las opciones del entorno array
\usepackage{bbm}                               % Fuentes para los numeros reales, racionales, ...
\usepackage{hyperref}                          % Enlaces de "hipertexto"
\usepackage[pdftex,usenames,dvipsnames]{color} % Muy importante para el color
\usepackage{bbding,trfsigns}                   % Simbolos
\usepackage{wasysym}                           % Variar el tama�o de los simbolos
\usepackage{datetime}                          % Imprimir fecha y hora

\DeclareMathAlphabet{\mathpzc}{OT1}{pzc}{m}{it} % Tipo de letra \mathpzc

\newcommand{\B}{\mathbb{B}}
\newcommand{\C}{\mathbb{C}}

\newcommand{\E}{\mathbb{E}}

\newcommand{\G}{\Gamma}

\newcommand{\N}{\mathbb{N}}
\newcommand{\Q}{\mathbb{Q}}
\newcommand{\R}{\mathbb{R}}
\newcommand{\Rn}{{\mathbb{R}^n}}

\newcommand{\eps}{\varepsilon}

\DeclareMathOperator{\supp}{supp}
\DeclareMathOperator{\spann}{span}

\newtheorem{ThA}{Theorem}

\newtheorem{Th}{Theorem}[section]              % Enumera los teoremas de acuerdo con la seccion (Theorem 1.1, Theorem 1.2 , ...)

\newtheorem{Prop}{Proposition}[section]
\newtheorem{Lem}{Lemma}[section]

\title[UMD-valued Bessel square functions in Hardy and BMO spaces]
      {UMD-valued square functions associated with Bessel operators in Hardy and BMO spaces}

\author[J.J. Betancor]{Jorge J. Betancor}
\author[A.J. Castro]{Alejandro J. Castro}
\author[L. Rodr\'iguez-Mesa]{L. Rodr\'iguez-Mesa}

\address{\newline
        Jorge J. Betancor, Alejandro J. Castro and Lourdes Rodr\'iguez-Mesa \newline
        Departamento de An\'alisis Matem\'atico,
        Universidad de La Laguna, \newline
        Campus de Anchieta, Avda. Astrof\'{\i}sico Francisco S\'anchez, s/n, \newline
        38271, La Laguna (Sta. Cruz de Tenerife), Spain}
\email{jbetanco@ull.es, ajcastro@ull.es, lrguez@ull.es}



\thanks{The authors are partially supported by MTM2010/17974.
The second author is also supported by a FPU grant from the Government of Spain.}

\begin{document}

\footnotetext{Date: \today.}

\maketitle                                  % Si no se activa esta opcion no se pone ni el titulo ni los autores en el encabezado de cada pagina

\begin{abstract}
    We consider Banach valued Hardy and BMO spaces in the Bessel setting. Square
    functions associated with Poisson semigroups for Bessel operators are defined by using
    fractional derivatives. If $\B$ is a UMD Banach space we obtain for $\B$-valued
    Hardy and BMO spaces equivalent norms involving $\gamma$-radonifying operators
    and square functions. We also establish characterizations of UMD Banach spaces by
  using Hardy and BMO-boundedness properties of $g$-functions associated to Bessel-Poisson semigroup.
\end{abstract}

%%%%%%%%%%%%%%%%%%%%%%%%%%%%%%%%%%%%%%%%%%%%%%%%%%%%%%%%%%%%%%%%%%%%%%%%%%%%%%%%%%%%%%%%%%%%%%%%%%%%%%%%%%%%%%%%%%%%
\section{Introduction}\label{sec:intro}
%%%%%%%%%%%%%%%%%%%%%%%%%%%%%%%%%%%%%%%%%%%%%%%%%%%%%%%%%%%%%%%%%%%%%%%%%%%%%%%%%%%%%%%%%%%%%%%%%%%%%%%%%%%%%%%%%%%%

If $\{P_t\}_{t>0}$ denotes the classical Poisson semigroup, for every $k \in \N$, the $k$-th square
(also called Littlewood-Paley-Stein) function $g_k(\{P_t\}_{t>0})$ is defined by
$$g_k(\{P_t\}_{t>0})(f)(x)
    = \left( \int_0^\infty \left| t^k \partial_t^k P_t(f)(x) \right|^2 \frac{dt}{t}\right)^{1/2}, \quad x \in \Rn,$$
for every $f \in L^p(\Rn)$. It is well-known that, for every $k \in \N$ and $1<p<\infty$, there exists $C>0$
such that
\begin{equation}\label{I1}
    \frac{1}{C} \|f\|_{L^p(\Rn)}
        \leq \| g_k(\{P_t\}_{t>0})(f)\|_{L^p(\Rn)}
        \leq C \|f\|_{L^p(\Rn)}, \quad f \in L^p(\Rn),
\end{equation}
or, in other words, for every $k \in \N$, the norm $\| \cdot \|_{p,k}$ defined by
$$\|f\|_{p,k}
    = \| g_k(\{P_t\}_{t>0})(f)\|_{L^p(\Rn)}, \quad f \in L^p(\Rn),$$
is equivalent to the usual norm in $L^p(\Rn)$. These equivalent norms $\| \cdot \|_{p,k}$ are more
suitable to establish $L^p$-boundedness properties of certain operators (for instance, Fourier multipliers (\cite[p. 58]{Ste1})).

Square functions have been also defined for other semigroups of operators. In \cite{Ste1} it was developed
the Littlewood-Paley theory for diffusion semigroups. If $\{T_t\}_{t>0}$ is a diffusion semigroup
(in the sense of Stein) on the measure space $(\Omega, \Sigma, \mu)$ where $\mu$ is a $\sigma$-finite measure
defined on the $\sigma$-algebra $\Sigma$ in $\Omega$, for every $k \in \N$, we define
$$g_k(\{T_t\}_{t>0})(f)(x)
    = \left( \int_0^\infty \left| t^k \partial_t^k T_t(f)(x) \right|^2 \frac{dt}{t}\right)^{1/2}, \quad x \in \Omega,$$
for every $f \in L^p(\Omega,\mu)$, $1<p<\infty$.

We have that, for every $k \in \N$ and $1<p<\infty$, there exists a constant $C>0$ such that
\begin{equation}\label{I2}
    \frac{1}{C} \|f-E_0(f)\|_{L^p(\Omega,\mu)}
        \leq \| g_k(\{T_t\}_{t>0})(f)\|_{L^p(\Omega,\mu)}
        \leq C \|f\|_{L^p(\Omega,\mu)},
\end{equation}
for every $f \in L^p(\Omega,\mu)$ (\cite{Hy}).
Here $E_0$ denotes the projector from $L^p(\Omega,\mu)$ to the fixed point space of $\{T_t\}_{t>0}$.
In (\cite[Corollary 3, p. 121]{Ste1}) \eqref{I2} was used to study
$L^p$-boundedness of Laplace transform type multipliers associated to $\{T_t\}_{t>0}$. Spectral
multipliers for a general class of operators were analyzed by Meda (\cite{Me1}) by using $g$-functions .

Also, the square functions have been defined by using convolutions. Suppose that $\psi \in L^2(\Rn)$.
We define the square function $g_\psi$ as follows:
$$g_\psi(f)(x)
    = \left( \int_0^\infty \left| \psi_t * f (x) \right|^2 \frac{dt}{t}\right)^{1/2}, \quad x \in \Rn,$$
where $\psi_t(x)=t^{-n} \psi(x/t)$, $x \in \Rn$ and $t>0$. Conditions on the function $\psi$ can be given in order that the norm $\| \cdot \|_\psi$ defined  by
$$\|f\|_\psi
    = \| g_\psi(f)\|_{L^p(\Rn)}, \quad f \in L^p(\Rn),$$
is equivalent to the usual norm in $L^p(\Rn)$ (see \cite[Remark 2.3]{KaWe}). Note that the classical Poisson semigroup is a convolution
semigroup.

Assume now that $\B$ is a Banach space and $\Omega \subset \Rn$ (for us, usually, $\Omega=\R$ or $\Omega=(0,\infty)$).
For every $1 \leq p < \infty$, we denote by $L^p(\Omega,\B)$ the $p$-th Bochner-Lebesgue $\B$-valued function
space with respect to the Lebesgue measure. By $L^{1,\infty}(\Omega,\B)$ we represent the weak
$L^1$-Bochner-Lebesgue $\B$-valued function space.

Suppose that $T$ is a bounded operator from $L^p(\Rn)$ into itself, for some $1 \leq p <\infty$.
We can define the tensor operator $T \otimes I_\B$ on $L^p(\Rn) \otimes \B$ in the natural way.
Here $I_\B$ denotes the identity operator in $\B$. We cannot ensure that $T \otimes I_\B$ can be
extended to $L^p(\Rn,\B)$ as a bounded operator in $L^p(\Rn,\B)$. However, if $T$ is a positive
operator, that is, $T(f) \geq 0$ when $f \geq 0$, then $T \otimes I_\B$ can be extended from
$L^p(\Rn) \otimes \B$ to $L^p(\Rn,\B)$ as a bounded operator from  $L^p(\Rn,\B)$ into itself.

If $\{T_t\}_{t>0}$ is a diffusion semigroup, $T_t$ is a positive operator in $L^p(\Omega)$, for every
$t>0$ and $1<p<\infty$. Then, for every $t>0$ the operator $T_t \otimes I_\B$ can be extended to
$L^p(\Omega,\B)$ as a bounded operator from $L^p(\Omega,\B)$ into itself, for every $1<p<\infty$. We continue denoting this extension by $T_t$.

In order to define square functions acting on $\B$-valued functions the more natural way is to replace the modulus in the
definitions by the norm $\| \cdot \|_\B$ in $\B$. We consider, for every $k \in \N$ and $1<p<\infty$,
$$g_{k,\B}(\{P_t\}_{t>0})(f)(x)
    = \left( \int_0^\infty \left\| t^k \partial_t^k P_t(f)(x) \right\|_\B^2 \frac{dt}{t}\right)^{1/2}, \quad x \in \Rn,$$
for every $f \in L^p(\Rn,\B)$.

Kwapie\'n \cite{Kw} established that the Banach-valued version of \eqref{I1} characterizes the Hilbert spaces in the following sense:
$\B$ is isomorphic to a Hilbert space if, and only if, for some (equivalently, for every) $1<p<\infty$,
there exits $C>0$ such that
\begin{equation}\label{I3}
    \frac{1}{C} \|f\|_{L^p(\Rn,\B)}
        \leq \| g_{1,\B}(\{P_t\}_{t>0})(f)\|_{L^p(\Rn,\B)}
        \leq C \|f\|_{L^p(\Rn,\B)},
\end{equation}
for every $f \in L^p(\Rn,\B)$. For this type of vector valued $g$-functions the question is to describe
the Banach spaces $\B$ for which one of the two inequalities in \eqref{I3} holds. This problem was considered
in the first time by Xu \cite{Xu} by using square functions defined by the Poisson semigroup for the torus.
In \cite{Xu} Xu  introduced generalized square functions where the exponent $2$ is replaced by $q \in (1,\infty)$
and he characterized the Banach spaces of $q$-martingale type and cotype as those for which some of the inequalities
for the $q$-square function in \eqref{I3} holds. After Xu's results, other authors have investigated this
question for vector valued $q$-square functions associated with other semigroups of operators
(see \cite{AST}, \cite{BFMT}, \cite{MTX} and \cite{TZ}, amongst others).

Hyt\"onen \cite{Hy} and Kaiser and Weis (\cite{Ka} and \cite{KaWe}) have introduced other definitions of square
functions in vector valued settings. They obtained, by using these new square functions, equivalent norms in
$L^p(\Rn,\B)$, $1<p<\infty$, and also in the Hardy space $H^1(\Rn,\B)$ and in the bounded mean
oscillation function space $BMO(\Rn,\B)$, provided that $\B$ is a UMD Banach space.

Hyt\"onen (in \cite{Hy}) defined square functions associated with subordinated diffusion semigroups in terms of
stochastic integrals. Kaiser and Weis (\cite{Ka} and \cite{KaWe}) used $\gamma$-radonifying operators
to get equivalent norms by employing convolution type square functions. Both approaches
(stochastic integrals and $\gamma$-radonifying operators) are connected (see, for instance,
\cite{NeVeWe}). In this paper, we will work with square functions
involving Poisson semigroups associated with Bessel operators and we will use $\gamma$-radonifying operators.

UMD Banach spaces (as in \cite{Hy}, \cite{Ka} and \cite{KaWe}) play an important role in our results. The Hilbert transform
$\mathcal{H}(f)$ of $f$ is defined by
$$\mathcal{H}(f)(x)
    = \lim_{\eps \to 0^+} \frac{1}{\pi} \int_{|x-y|>\eps} \frac{f(y)}{x-y} dy, \quad \text{a.e. } x \in \R,$$
for every $f \in L^p(\R)$, $1 \leq p<\infty$. It is a key result in harmonic analysis that the Hilbert transform
is a bounded operator from $L^p(\R)$ into itself, for every $1 < p<\infty$, and from $L^1(\R)$ into
$L^{1,\infty}(\R)$.

It is clear that the Hilbert transform is not a positive operator in $L^p(\R)$, $1 \leq p < \infty$.
A Banach space $\B$ is said to be a UMD Banach space when the operator $\mathcal{H} \otimes I_\B$ can be extended from
$L^p(\R) \otimes \B$ to $L^p(\R,\B)$ as a bounded operator from $L^p(\R,\B)$ into itself, for some $1<p<\infty$.
This extension property does not depend on $1<p<\infty$ in the following sense: the extension property holds
for some $1<p<\infty$ if, and only if, it is true for every $1<p<\infty$.

The main properties of UMD Banach spaces were established by Bourgain (\cite{Bou})
and Burkholder (\cite{Bu1}). There exist a lot of characterizations for UMD Banach spaces
(see, for instance, \cite{Bou}, \cite{Buk}, \cite{Bu1},  \cite{CL}, \cite{Gue}, \cite{SchSi},  and \cite{Xu}).

We recall now some definitions and properties concerning $\gamma$-radonifying operators. Suppose
that $(\gamma_k)_{k \in \N}$ is a sequence of independent standard Gaussian random variables on a probability
space $(\Omega, \Sigma ,\mathbb{P})$, $\B$ is a Banach space and $H$ is a Hilbert space. If
$T : H \longrightarrow \B$ is a linear operator we define $\|T\|_{\gamma(H,\B)}$ as follows:
$$\|T\|_{\gamma(H,\B)}
    = \sup \left( \E \left\| \sum_k \gamma_k T(h_k)\right\|_\B^2 \right)^{1/2},$$
where the supremum is taken over all the finite ortonormal sets $\{h_k\}$ in H. Here $\E$ denotes the
expectation in $(\Omega, \Sigma ,\mathbb{P})$. The space $\gamma(H,\B)$ of $\gamma$-radonifying operators from
$H$ to $\B$ is defined as the completion, with respect to $\|\cdot\|_{\gamma(H,\B)}$, of the space of finite
rank operators in $L(H,\B)$, the space of bounded linear operators from $H$ into $\B$.

If $H$ is a separable Hilbert space and the Banach space $\B$ does not contain copies of $c_0$ then
$$\|T\|_{\gamma(H,\B)}
    = \left( \E \left\| \sum_{k=1}^\infty \gamma_k T(h_k) \right\|_\B^2 \right)^{1/2}, \quad T \in \gamma(H,\B),$$
where $\{h_k\}_{k=1}^\infty$ is an orthonormal basis of $H$. Note that the quantity in the right hand side
in the last equality does not depend of the orthonormal basis $\{h_k\}_{k=1}^\infty$ of $H$. If $\B$ is a UMD
Banach space, $\B$ does not contain copies of $c_0$.

Suppose now that $H=L^2(A,\G,\mu)$ where $(A,\G,\mu)$ is a $\sigma$-finite measure space with countably generating
$\sigma$-algebra $\G$ and let $\B$ be a Banach space. If $f : A \longrightarrow \B$ is weakly continuous in $H$, that is, for every $S \in \B^*$, the dual
space of $\B$, $S \circ f \in H$, then there exists $T_f \in L(H,\B)$ such that
$$\langle S, T_f(h) \rangle_{\B^*,\B}
    = \int_A \langle S, f(t) \rangle_{\B^*,\B} h(t) d\mu(t), \quad h \in H.$$
It is usual to write $f \in \gamma(A,\mu,\B)$ when $T_f \in \gamma(H,\B)$ and, to simplify, to identify $f$ with
$T_f$. If $\B$ does not contain copies of $c_0$, the space $\{T_f, f \in \gamma(A,\mu,\B)\}$ is dense in $\gamma(H,\B)$.
Throughout this paper we consider $H=L^2((0,\infty),dt/t)$.

As a consequence of \cite[Theorem 4.2]{KaWe} we can deduce the following result.

\begin{ThA}\label{Th:A}
    Let $\B$ be a UMD Banach space and $k \in \N$. We define
    $$G_{k,\B}(f)(t,x)
        = t^k \partial_t^k P_t(f)(x), \quad x \in\Rn \text{ and } t>0,$$
    for every $f \in S(\Rn,\B)$, the $\B$-valued Schwartz function space. Then, there exists $C>0$ such that
    \begin{equation}\label{I3'}
    \frac{1}{C} \|f\|_{E(\Rn,\B)}
        \leq \| G_{k,\B}(f)\|_{E(\Rn,\gamma(H,\B))}
        \leq C \|f\|_{E(\Rn,\B)},
    \end{equation}
    where $E=L^p$, $1<p<\infty$, $E=H^1$ or $E=BMO$.
\end{ThA}

Note that, since $\gamma(H,\C)=H$, \eqref{I3'} can be seen as a Banach valued extension of \eqref{I1}.

Motivated by Theorem~\ref{Th:A}, our objective in this paper is to obtain equivalent norms for Hardy and
BMO spaces defined via Bessel operators by using square functions involving Bessel Poisson semigroups in a Banach
valued setting.

The study of harmonic analysis associated with Bessel operators was started by Muckenhoupt and Stein (\cite{MS}).
We consider the Bessel operator
$$\Delta_\lambda=-x^{-\lambda} \frac{d}{dx} x^{2\lambda} \frac{d}{dx} x^{-\lambda} \text{ on } (0,\infty),$$
where $\lambda >0$. The Hankel transform $h_\lambda$ is defined by
$$h_\lambda(f)(x)
    = \int_0^\infty \sqrt{xy} J_{\lambda-1/2}(xy)f(y)dy, \quad x \in (0,\infty),$$
for every $f \in L^1(0,\infty)$. By $J_\nu$ we denote the Bessel function of the first kind and order $\nu$.
The Hankel transform can be extended from $L^1(0,\infty) \cap L^2(0,\infty)$ to $L^2(0,\infty)$ as an isometry
in $L^2(0,\infty)$.

The space $S_\lambda(0,\infty)$ is constituted by all those smooth functions $\phi$ on $(0,\infty)$ such that,
for every $m,k \in \N$,
$$\eta_{m,k}^\lambda(\phi)
    = \sup_{x \in (0,\infty)} (1+x^2)^m \left| \left( \frac{1}{x} \frac{d}{dx} \right)^k \left(x^{-\lambda} \phi(x) \right) \right| < \infty.$$
$S_\lambda(0,\infty)$ is endowed with the topology generated by the system $\{\eta_{m,k}^\lambda\}_{m,k \in \N}$ of seminorms.
Thus, $S_\lambda(0,\infty)$ is a Fr\'echet space and the Hankel transform $h_\lambda$ is a bounded bijective operator from $S_\lambda(0,\infty)$
into itself and $h_\lambda^{-1}=h_\lambda$ (\cite[Theorem 5.4-1]{Ze}).

The Bessel operator $\Delta_\lambda$ and the Hankel transform $h_\lambda$ are closely connected, as the following
equality shows:
$$h_\lambda(\Delta_\lambda f)(x)
    = x^2 h_\lambda(f)(x), \quad x \in (0,\infty),$$
for every $f \in S_\lambda(0,\infty)$. The Hankel transform plays for the Bessel operator the same role as the
Fourier transformation with respect to the Laplacian operator.

We define the operator $\sqrt{\Delta_\lambda}$ as follows:
$$\sqrt{\Delta_\lambda} f
    = h_\lambda \left( y h_\lambda(f)\right), \quad f \in D(\sqrt{\Delta_\lambda}),$$
where
$D(\sqrt{\Delta_\lambda})= \{f \in L^2(0,\infty) : y h_\lambda(f) \in L^2(0,\infty) \}$.
The Poisson semigroup $\{P_t^\lambda\}_{t>0}$ associated with the Bessel operator $\Delta_\lambda$, that is,
generated by the operator $-\sqrt{\Delta_\lambda}$, is defined by
$$P_t^\lambda(f)(x)
    =\int_0^\infty P_t^\lambda(x,y)f(y)dy, \quad t,x \in (0,\infty),$$
for every $f \in L^p(0,\infty)$, $1 \leq p \leq \infty$, where
$$P_t^\lambda(x,y)
    = \frac{2\lambda t (xy)^\lambda}{\pi} \int_0^\pi \frac{(\sin \theta)^{2\lambda-1}}{\left[ (x-y)^2+t^2+2xy(1-\cos \theta) \right]^{\lambda+1}}d\theta,
    \quad t,x,y \in (0,\infty).$$
Square functions defined by the Bessel Poisson semigroup have been studied in \cite{BCR3}, \cite{BFMT}, \cite{BFS} and \cite{StemPhD}, amongst others.
In \cite{BFMT} it was considered generalized Littlewood-Paley functions associated with $\{P_t^\lambda\}_{t>0}$ in a Banach
valued setting. If $\B$ is a Banach space and $1<q<\infty$ we consider the $q$-square function defined by
$$g_\B^q(\{P_t^\lambda\}_{t>0})(f)(x)
    = \left( \int_0^\infty \left\| t \partial_t P_t^\lambda(f)(x) \right\|_\B^q \frac{dt}{t} \right)^{1/q}, \quad x \in (0,\infty).$$
By using these Littlewood-Paley functions, martingale type and cotype of the Banach space $\B$ can be characterized.
Moreover, as in the classical case \cite{Kw}, according to \cite[Theorems 2.4 and 2.5]{BFMT}, if $\lambda >0$,
the Banach space $\B$ is isomorphic to a Hilbert space if, and only if, for some (equivalently, for every) $1<p<\infty$, there exists $C>0$ such that
\begin{equation}\label{I4}
    \frac{1}{C} \|f\|_{L^p((0,\infty),\B)}
        \leq \| g_{\B}^2(\{P_t^\lambda\}_{t>0})(f)\|_{L^p(0,\infty)}
        \leq C \|f\|_{L^p((0,\infty),\B)}, \quad f \in L^p((0,\infty),\B).
\end{equation}

In order to extend the equivalence \eqref{I4} to Banach spaces that are not Hilbert spaces we define, motivated by
\cite{KaWe}, new square functions involving Bessel Poisson semigroup and $\gamma$-radonifying operators.

In \cite{SW} Segovia and Wheeden introduced fractional derivatives as follows. Suppose that
$F : (0,\infty) \times \R \longrightarrow \C$ is a nice enough function, $\beta>0$ and $m \in \N$ is such that
$m-1 \leq \beta < m$. The $\beta$-th derivative $\partial_t^\beta F$ of $F$ with respect to $t$ is defined by
$$\partial_t^\beta F(t,x)
    = \frac{e^{-i\pi(m-\beta)}}{\G(m-\beta)} \int_0^\infty \partial_t^m F(t+s,x) s^{m-\beta-1} ds, \quad x \in \R \text{ and } t>0.$$
We consider the operator
$$G_\B^{\lambda,\beta}(f)(t,x)
    = t^\beta \partial_t^\beta P_t^\lambda(f)(x), \quad t,x \in (0,\infty),$$
for every $f \in L^p((0,\infty),\B)$, $1 \leq p \leq \infty$.

The following result was established in \cite[Theorem 1.2]{BCR3}.

\begin{ThA}\label{Th:B}
    Let $\lambda>0$, $\beta>0$ and $1<p<\infty$. Suppose that $\B$ is a UMD Banach space. Then, there exists $C>0$
    such that
    $$\frac{1}{C} \|f\|_{L^p((0,\infty),\B)}
        \leq \| G_\B^{\lambda,\beta}(f)\|_{L^p((0,\infty),\gamma(H,\B))}
        \leq C \|f\|_{L^p((0,\infty),\B)}, \quad f \in L^p((0,\infty),\B).$$
\end{ThA}

Note that the Bessel Poisson semigroup $\{P_t^\lambda\}_{t>0}$ is not a diffusion semigroup because it is not
conservative. Then, Theorem~\ref{Th:B} is not included in \cite[Theorem 1.6]{Hy}.

Our objective in this paper is to study the operator $G_\B^{\lambda,\beta}$ in Hardy and BMO spaces adapted to
the Bessel setting.

Fridli \cite{Fri}, in his study about the local Hilbert transform considered the Hardy type space $H^1_{\rm o} (0,\infty )$
which consists on all those functions $f\in L^1(0,\infty)$ such that the odd extension function
$f_{\rm o}$ of $f$ to $\R$ is in the classical Hardy space $H^1(\R)$. In \cite[Theorem 2.1]{Fri}
the space $H^1_{\rm o}(0,\infty )$ was characterized by using the local Hilbert transform.

The maximal operator $P_*^\lambda$ associated with the semigroup $\{P_t^\lambda\}_{t>0}$ is defined by
$$P_*^\lambda(f)
    = \sup_{t>0} |P_t^\lambda(f)|, \quad f \in L^p(0,\infty), \ 1 \leq p \leq \infty.$$
The space $H^1_{\rm o}(0,\infty )$ can be described by $P_*^\lambda$ as follows (see \cite[Theorem 1.10]{BDT}).

\begin{ThA}\label{Th:C}
    Let $\lambda>0$ and $f \in L^1(0,\infty)$. Then, $f \in H_{\rm o}^1(0,\infty )$ if, and only if,
    $P_*^\lambda(f) \in L^1(0,\infty)$.
\end{ThA}

Note that, according to Theorem~\ref{Th:C}, the Hardy space defined by the Bessel Poisson semigroup $\{P_t^\lambda\}_{t>0}$
actually does not depend on $\lambda$.

If $\B$ is a Banach space, the $\B$-valued Hardy space $H_{\rm o}^1((0,\infty ),\B)$ is defined as in the scalar case. By considering
the maximal operator $P_*^\lambda$ on $L^p((0,\infty),\B)$, $1 \leq p < \infty$, as follows
$$P_*^\lambda(f)
    = \sup_{t>0} \|P_t^\lambda(f)\|_\B,$$
we introduce the space $H^1(\Delta _\lambda , \mathbb{B})$ constituted by all those functions $f \in L^1((0,\infty),\B )$ such that $P_*^\lambda(f) \in L^1(0,\infty)$.
As in Theorem~\ref{Th:C}, a function  $f\in L^1((0,\infty),\B )$ is in $H_{\rm o}^1((0,\infty ),\B)$ if, and only if, $f\in H^1(\Delta _\lambda , \mathbb{B})$ (see Proposition \ref{Prop1}).

By $BMO(\R)$ we denote as usual the space of all bounded mean oscillation functions on $\R$.
In \cite{BCFR1} the space $BMO_{\rm o}(0,\infty)$ was introduced. A complex measurable function $f$ on $(0,\infty)$ is said
to be in $BMO_{\rm o}(0,\infty )$ when there exists $C>0$ such that:
\begin{itemize}
    \item[$(Bi)$] $\displaystyle \frac{1}{|I|} \int_I |f(x)| dx \leq C$, for every $I=(0,r)$, $r>0$,\\
    \item[$(Bii)$] $\displaystyle \frac{1}{|I|} \int_I |f(x)-f_I| dx \leq C$, for each $I=(r,s)$, $0<r<s<\infty $.
\end{itemize}
Here, |I| denotes the length of $I$ and $\displaystyle f_I=\frac{1}{|I|} \int_I f(y)dy$.

Condition $(Bii)$ says that $f \in BMO(0,\infty)$. It is not hard to see that $f \in BMO_{\rm o}(0,\infty )$ if,
and only if, the odd extension $f_{\rm o}$ of $f$ to $\R$ is in $BMO(\R)$. As in the classical case, the dual
space of $H^1_{\rm o}(0,\infty )$ can be identified with the space $BMO_{\rm o}(0,\infty )$. In \cite{BCFR2} harmonic analysis operators
(maximal operators, $g$-functions, Riesz transforms, \dots) in the Bessel setting were studied on $BMO_{\rm o}(0,\infty )$
and in \cite{BCFR1}, $BMO_{\rm o}(0,\infty )$ was described by using Carleson measures involving Bessel Poisson semigroups.

If $\B$ is a Banach space, the Banach valued function space $BMO_{\rm o}((0,\infty ),\B)$ is defined as in the scalar case, replacing the modulus by the norm $\| \cdot \|_\B$ in $\B$. The dual space of $H^1_{\rm o}((0,\infty ),\B)$ can be identified
with $BMO_{\rm o}((0,\infty ),\B^*)$, provided that $\B$ has the UMD property.

Our first result is the following.

\begin{Th}\label{Th:1}
    Let $\lambda \geq 1$ and $\beta>0$. Suppose that $\B$ is a UMD Banach space. Then, there exists $C>0$
    such that
    $$\frac{1}{C} \|f\|_{E((0,\infty ),\B)}
        \leq \| G_\B^{\lambda,\beta}(f)\|_{E((0,\infty ),\gamma(H,\B))}
        \leq C \|f\|_{E((0,\infty ),\B)}, \quad f \in E((0,\infty ),\B),$$
    where $E$ denotes $H^1_{\rm o}$ or $BMO_{\rm o}$.
\end{Th}

The Bessel operator $\Delta_\lambda$ can be written as $\Delta_\lambda=D_\lambda^* D_\lambda$, where
$D_\lambda=x^\lambda \frac{d}{dx} x^{-\lambda}$ and $D_\lambda^*=-x^{-\lambda} \frac{d}{dx} x^{\lambda}$
is the adjoint operator of $D_\lambda$. $D_\lambda$ is used to define the Riesz transform in the Bessel setting
(see \cite{BBFMT}). We consider the operator $\mathcal{G}_\B^\lambda$ acting on $\B$-valued functions by
$$\mathcal{G}_\B^\lambda(f)(t,x)
    = t D_\lambda^* P_t^{\lambda+1}(f)(x), \quad t,x \in (0,\infty).$$
This operator will be useful to get characterizations of UMD Banach spaces.

\begin{Th}\label{Th:2}
    Let $\lambda \geq 1$. Assume that $\B$ is a UMD Banach space. Then, the operator $\mathcal{G}_\B^\lambda$
    is bounded from $H^1_{\rm o}((0,\infty ),\B)$ into $H^1_{\rm o}((0,\infty ),\gamma(H,\B))$ and from $BMO_{\rm o}((0,\infty ),\B)$ into
    $BMO_{\rm o}((0,\infty ),\gamma(H,\B))$.
\end{Th}

We now establish new characterizations of the UMD Banach spaces by using our square functions.

\begin{Th}\label{Th:3}
    Let $\B$ be a Banach space and $\lambda \geq 1$. The following assertions are equivalent.
    \begin{itemize}
        \item[$(i)$] $\B$ is UMD.

         \item[$(ii)$] There exists $C>0$ such that, for every $f \in E((0,\infty )) \otimes \B$,
        $$\|f\|_{E((0,\infty ),\B)}
            \leq C\| G_\B^{\lambda,1}(f)\|_{E((0,\infty ),\gamma(H,\B))}$$
        and
        $$\| \mathcal{G}_\B^{\lambda}(f)\|_{E((0,\infty ) ,\gamma(H,\B))}
            \leq C \|f\|_{E((0,\infty ),\B)},$$
        where $E$ denotes $H^1_{\rm o}$ or $BMO_{\rm o}$.

        \item[$(iii)$] For a certain (equivalently, for every) $\beta>0$ there exists $C>0$ such that
        $$\frac{1}{C} \|f\|_{E((0,\infty ),\B)}
            \leq \| G_\B^{\lambda,\delta}(f)\|_{E((0,\infty ),\gamma(H,\B))}
            \leq C \|f\|_{E((0,\infty ),\B)}, \quad f \in E((0,\infty ),\B),$$
        for $\delta=\beta$ and $\delta=\beta+1$, where $E$ represents $H^1_{\rm o}$ or $BMO_{\rm o}$.

    \end{itemize}
\end{Th}

In the following sections we present proofs of Theorems~\ref{Th:1}, \ref{Th:2} and \ref{Th:3}.

Throughout this paper by $c$ and $C$ we represent positive constants not necessarily the same in each
occurrence.

%\newpage
%%%%%%%%%%%%%%%%%%%%%%%%%%%%%%%%%%%%%%%%%%%%%%%%%%%%%%%%%%%%%%%%%%%%%%%%%%%%%%%%%%%%%%%%%%%%%%%%%%%%%%%%%%%%%%%%%%%%
\section{Proof of Theorem~\ref{Th:1}}\label{sec:Proof1}
%%%%%%%%%%%%%%%%%%%%%%%%%%%%%%%%%%%%%%%%%%%%%%%%%%%%%%%%%%%%%%%%%%%%%%%%%%%%%%%%%%%%%%%%%%%%%%%%%%%%%%%%%%%%%%%%%%%%

\subsection{}\label{subsec:2.1}

In this part we prove that the operator $G_\B^{\lambda,\beta}$ is bounded from  $H^1_{\rm o}((0,\infty ),\B)$ into
$H^1_{\rm o}((0,\infty ),\gamma(H,\B))$.

As in \cite{Fri} we say that a strongly measurable function $a : (0,\infty) \longrightarrow \B$ is a $q$-atom, where
$1<q \leq \infty$, when one of the following two conditions is satisfied:
\begin{itemize}
    \item[$(Ai)$] $a=\dfrac{b}{\delta} \chi_{(0,\delta)}$, where $\delta>0$ and $b \in \B$, being $\|b\|_\B=1$,
    \item[$(Aii)$] there exists an interval $I \subset (0,\infty)$ such that
    $\supp a \subset I$,  $\|a\|_{L^q((0,\infty),\B)} \leq |I|^{1/q-1}$ and $\int_I a(x)dx=0$.
\end{itemize}

The arguments presented in the proof of \cite[Proposition~3.7]{BDT} and \cite[Theorem 2.1]{Fri}
allow us to get the following.

\begin{Prop}\label{Prop1}
    Let $X$ be a Banach space, $\lambda>0$ and $1<q<\infty$. Suppose that $f : (0,\infty) \longrightarrow X$ is a
    strongly measurable function. The following assertions are
    equivalent:
    \begin{itemize}
        \item[$(i)$] $f \in H^1(\Delta_\lambda,X)$.
        \item[$(ii)$] $f\in H_{\rm o}^1((0,\infty ),X)$.
        \item[$(iii)$] For every $j \in \N$, there exist a $q$-atom $a_j$ and $\lambda_j \in \C$ such that
        $\sum_{j\in \N} |\lambda_j| < \infty$ and $f=\sum_{j\in \N}\lambda_j a_j$ in $L^1((0,\infty),X)$.
    \end{itemize}
    Moreover, in this case, we have that
    $$\|P_*^\lambda(f)\|_{L^1((0,\infty ),X)}
        \sim \|f\|_{H_{\rm o}^1((0,\infty ),X)}=\|f_o\|_{H^1(\R,X)}
        \sim \inf\sum_{j\in \N} |\lambda_j|,$$
    where the infimum is taken over all the complex sequences $\{\lambda_j\}_{j \in \N}$ such that
    $\sum_{j\in \N} |\lambda_j| < \infty$ and, for a certain sequence of $q$-atoms $\{a_j\}_{j \in \N}$,
    $f=\sum_{j\in \N}\lambda_j a_j$. Here $f_{\rm o}$ denotes the odd extesion to $\R$ of $f$.
\end{Prop}

In the sequel we write $H_{\rm o}^1((0,\infty ),\B)$ to refer us $H^1(\Delta_\lambda,\B)$, $\lambda>0$, and we denote by $H^1((0,\infty),\B)$ the classical $\B$-valued Hardy space on $(0,\infty)$, that is, a function $f \in L^1((0,\infty),\B)$ is in
$H^1((0,\infty),\B)$ if, and only if, $f= \sum_{j\in \N}\lambda_j a_j$, where $\lambda_j \in \C$, $j \in \N$, being
$\sum_{j\in \N}|\lambda_j|<\infty$, and for some $q \in (1,\infty]$, $a_j$ satisfies the condition $(Aii)$, for every $j \in \N$.

Our objective is to see that, for every $f \in H_{\rm o}^1((0,\infty ),\B)$, $G_\B^{\lambda,\beta}(f) \in H_{\rm o}^1((0,\infty ),\gamma(H,\B))$,
and
$$\|G_\B^{\lambda,\beta}(f)\|_{H_{\rm o}^1((0,\infty ),\gamma(H,\B))}
    \leq C \|f\|_{H_{\rm o}^1((0,\infty ),\B)},
$$
where $C>0$ does not depend on $f$.

By taking into account  Proposition \ref{Prop1}, in order to show this, we can see that $G_\B^{\lambda,\beta}$ is bounded from $H^1_{\rm o}((0,\infty ),\B)$
into $L^{1}((0,\infty),\gamma(H,\B))$ (Lemma \ref{Lemap13}) and that $P_*^\lambda \circ G_\B^{\lambda,\beta}$ is bounded from $H^1_{\rm o}((0,\infty ),\B)$
into $L^{1}(0,\infty)$ (Lemma \ref{Lemap14}). We first need to establish the following boundedness property for $G_\B^{\lambda ,\beta}$.

\begin{Lem}\label{Lemap8}
Let $\B$ be a UMD Banach space, $\beta >0$ and $\lambda \geq 1$. The operator $G_\B^{\lambda,\beta}$ is bounded from $L^1((0,\infty),\B)$ into $L^{1,\infty}((0,\infty),\gamma(H,\B))$
and from $H^1((0,\infty),\B)$ into $L^{1}((0,\infty),\gamma(H,\B))$.
\end{Lem}
\begin{proof}
Let us define
$$
K_\lambda^\beta(t;x,y)
    = t^\beta \partial_t^\beta P_t^\lambda(x,y), \quad t,x,y \in (0,\infty),
$$
and consider $m \in \N$ such that $m-1 \leq \beta < m$.

According to \cite[(4.6)]{GLLNU}, for every $k\in \N$, $\theta \in (0,\pi)$  and $t,s,x,y \in (0,\infty)$, we can write
\begin{align}\label{H0}
    & \partial_t^k \left[ \frac{t+s}{[(t+s)^2 + (x-y)^2 + 2xy(1-\cos \theta)]^{\lambda+1}} \right]
        = -\frac{1}{2\lambda} \partial_t^{k+1} \left[ \frac{1}{[(t+s)^2 + (x-y)^2 + 2xy(1-\cos \theta)]^{\lambda}} \right] \nonumber \\
    & \qquad = - \frac{1}{2} \sum_{\ell \in \N, 0\leq \ell \leq (k+1)/2} (-1)^{k+1-\ell} E_{k+1,\ell}(t+s)^{k+1-2\ell}
                                \frac{(\lambda+1)(\lambda+2) \cdots (\lambda+k-\ell)}{[(t+s)^2 + (x-y)^2 + 2xy(1-\cos \theta)]^{\lambda+1+k-\ell}},
\end{align}
where $E_{k+1,\ell}=2^{k+1-2\ell}(k+1)!/(\ell !(k+1-2\ell)!)$, $\ell \in \N$, $0\leq \ell \leq (k+1)/2$.

Then, for each $k\in \N, \theta \in (0,\pi)$  and $t,s,x,y \in (0,\infty)$,
$$
\left| \partial_t^k \left[ \frac{t+s}{[(t+s)^2 + (x-y)^2 + 2xy(1-\cos \theta)]^{\lambda+1}} \right] \right|
    \leq \frac{C}{(t+s+|x-y|+\sqrt{2xy(1-\cos \theta)})^{2\lambda+1+k}}.
$$
It follows that, for $k\in \N$ and $t,s,x,y \in (0,\infty)$,
$$\partial_t^k P_{t+s}^\lambda(x,y)
    = \frac{2 \lambda (xy)^\lambda}{\pi} \int_0^\pi (\sin \theta)^{2\lambda-1}
        \partial_t^k \left[ \frac{t+s}{[(t+s)^2 + (x-y)^2 + 2xy(1-\cos \theta)]^{\lambda+1}} \right] d\theta .
$$
Then, for $k\in \N$, we have that
\begin{equation}\label{H2}
    |\partial_t^k P_{t+s}^\lambda(x,y)|
        \leq C \frac{(xy)^\lambda}{(t+s+|x-y|)^{2\lambda+1+k}}, \quad t,s,x,y \in (0,\infty),
\end{equation}
and also
\begin{align}\label{8prima}
|\partial_t^k P_{t+s}^\lambda(x,y)|&\leq C(xy)^\lambda \left(\int_0^{\pi /2}\frac{\theta ^{2\lambda -1}}{(t+s+|x-y|+\sqrt{xy}\theta)^{2\lambda +1+k}}d\theta +\int _{\pi /2}^\pi
\frac{(\sin \theta )^{2\lambda -1}}{(t+s+|x-y|+\sqrt{xy})^{2\lambda +1+k}}d\theta \right)\nonumber\\
&\leq C \left(\int_0^{\pi \sqrt{xy}/2}\frac{u^{2\lambda -1}}{(t+s+|x-y|+u)^{2\lambda +1+k}}du+
\frac{(xy)^\lambda }{(t+s+|x-y|+\sqrt{xy})^{2\lambda +1+k}} \right)\nonumber\\
&\leq C\frac{1}{(t+s+|x-y|)^{k+1}},\quad t,s,x,y\in (0,\infty ).
\end{align}

Let $g \in S_\lambda(0,\infty)$. We can write
\begin{equation}\label{H1}
    t^\beta \partial_t^\beta P_t^\lambda(g)(x)
        = \int_0^\infty K_\lambda^\beta (t;x,y) g(y) dy, \quad t,x \in (0,\infty).
\end{equation}
Indeed, by \eqref{H2} we can differentiate under the integral sign and write
$$\partial_t^m P_{t+s}^\lambda (g)(x)
    = \int_0^\infty \partial_t^m P_{t+s}^\lambda(x,y) g(y)dy, \quad t,s,x \in (0,\infty),$$
and, again by (\ref{H2}),
\begin{align*}
    & \int_0^\infty s^{m-\beta-1} \int_0^\infty |\partial_t^m P_{t+s}^\lambda(x,y)| \ |g(y)| dy ds
        \leq C x^\lambda \int_0^\infty \frac{s^{m-\beta-1}}{(t+s)^{2\lambda+1+m}}ds
        < \infty, \quad t,x \in (0,\infty).
\end{align*}
Hence,
\begin{align*}
    \partial_t^\beta P_t^\lambda(g)(x)
        = & \frac{e^{-i\pi(m-\beta)}}{\G(m-\beta)} \int_0^\infty \partial_t^m P_{t+s}^\lambda(g)(x) s^{m-\beta-1} ds \nonumber \\
        = & \int_0^\infty g(y) \frac{e^{-i\pi(m-\beta)}}{\G(m-\beta)} \int_0^\infty \partial_t^m P_{t+s}^\lambda(x,y) s^{m-\beta-1} ds dy \nonumber\\
        = & \int_0^\infty g(y) \partial_t^\beta P_t^\lambda(x,y) dy, \quad t,x \in (0,\infty),
\end{align*}
and \eqref{H1} is proved.

On the other hand, by using Minkowski's inequality and \eqref{8prima} we deduce that
\begin{align}\label{H4}
\left\| K_\lambda^\beta( \cdot ;x,y) \right\|_H    &
        \leq C \int_0^\infty s^{m-\beta-1} \left( \int_0^\infty \left| t^\beta \partial_t^m P_{t+s}^\lambda(x,y) \right|^2 \frac{dt}{t} \right)^{1/2} ds \nonumber\\
        &\leq C \int_0^\infty s^{m-\beta-1} \left( \int_0^\infty \frac{t^{2\beta -1}}{(t+s+|x-y|)^{2m+2}}dt \right)^{1/2} ds \nonumber \\
        &=C \int_0^\infty \frac{s^{m-\beta-1}}{(s+|x-y|)^{m-\beta +1}}ds=\frac{C}{|x-y|}, \quad x,y \in (0,\infty), \ x \neq y.
\end{align}
From \eqref{H4} we infer that the integral in \eqref{H1} converges as a $H$-Bochner integral, provided that $x \notin \supp g$. We define
$$F(x)
    = \int_0^\infty K_\lambda^\beta(\cdot;x,y)g(y)dy, \quad x \notin \supp g,$$
where the integral is understood in the $H$-Bochner sense.

Let $h \in H$. According to the properties of the Bochner integrals and Fubini's theorem we get
\begin{align*}
    \int_0^\infty h(t) F(x)(t) \frac{dt}{t}
        = & \int_0^\infty g(y)  \int_0^\infty K_\lambda^\beta(t;x,y) h(t)  \frac{dt}{t}dy \\
        = & \int_0^\infty h(t)  \int_0^\infty K_\lambda^\beta(t;x,y) g(y)  dy\frac{dt}{t}, \quad x \notin \supp g.
\end{align*}
Then, it follows that
$$G_\C^{\lambda,\beta}(g)(\cdot,x)
    = \int_0^\infty K_\lambda^\beta(\cdot;x,y) g(y)  dy, \quad x \notin \supp g,$$
where the integral is understood in the $H$-Bochner sense.

Suppose now that $g \in S_\lambda(0,\infty) \otimes \B$, that is, $g=\sum_{i=1}^n b_i g_i$,
where $b_i \in \B$ and $g_i \in S_\lambda(0,\infty)$, $i=1, \dots, n \in \N$. Since $\gamma(H,\C)=H$, we have that
$K_\lambda^\beta(\cdot;x,y)g(y) \in \gamma(H,\B)$, $x,y \in (0,\infty)$, $x \neq y$, and
\begin{align*}
    \|K_\lambda^\beta(\cdot;x,y)g(y)\|_{\gamma(H,\B)}
        \leq & C \|K_\lambda^\beta(\cdot;x,y)\|_{H}\|g(y)\|_\B
        \leq C \frac{\|g(y)\|_\B}{|x-y|}, \quad x,y \in (0,\infty), \ x \neq y.
\end{align*}
Also,
$$\int_0^\infty K_\lambda^\beta(\cdot;x,y) g(y)  dy
    = \sum_{i=1}^n b_i \int_0^\infty K_\lambda^\beta(\cdot;x,y) g_i(y) dy, \quad x \notin \supp g, $$
where the integral in the left hand side is understood in the $\gamma(H,\B)$-Bochner sense and the ones in the right hand
side are understood in the $H$-Bochner sense. Hence, we obtain
\begin{equation}\label{H5}
    G_\B^{\lambda,\beta}(g)(\cdot,x)
    = \int_0^\infty K_\lambda^\beta(\cdot;x,y) g(y)  dy, \quad x \notin \supp g,
\end{equation}
as elements in $\gamma(H,\B)$, where the integral is understood in the $\gamma(H,\B)$-Bochner sense.

On the other hand, we have that, for every $t,x,y \in (0,\infty)$, and $x \neq y$,
\begin{align*}
    & \partial_x K_\lambda^\beta(t;x,y)
        = \frac{t^\beta e^{-i\pi(m-\beta)}}{\G(m-\beta)} \int_0^\infty \partial_x \partial_t^m P_{t+s}^\lambda(x,y) s^{m-\beta-1} ds \\
    & \qquad = \frac{2\lambda t^\beta e^{-i\pi(m-\beta)}}{\pi\G(m-\beta)} \int_0^\infty s^{m-\beta-1}
                    \int_0^\pi (\sin \theta )^{2\lambda -1}\partial_x \partial_t^m \left[ \frac{(xy)^\lambda(t+s)}{[(t+s)^2+(x-y)^2+2xy(1-\cos \theta)]^{\lambda+1}} \right] d\theta  ds.
\end{align*}
Derivations under the integral signs can be justified as above. By making a derivative with respect to $x$ in \eqref{H0} for $k=m$, we get
\begin{align}\label{I1I2}
    & | \partial_x \partial_t^mP_{t+s}^\lambda (x,y)|
        \leq C \Big\{ (xy)^\lambda \int_0^\pi \frac{(\sin \theta)^{2\lambda-1} (|x-y|+y(1-\cos \theta))}{(t+s+|x-y|+\sqrt{2xy(1-\cos \theta)})^{2\lambda+m+3}} d\theta\nonumber\\
      & \qquad + x^{\lambda-1}y^\lambda \int_0^\pi \frac{(\sin \theta)^{2\lambda-1}}{(t+s+|x-y|+\sqrt{2xy(1-\cos \theta)})^{2\lambda+m+1}} d\theta\Big\}\nonumber \\
      & \quad \leq C \Big\{ (xy)^\lambda \int_0^\pi \frac{(\sin \theta)^{2\lambda-1}}{(t+s+|x-y|+\sqrt{2xy(1-\cos \theta)})^{2\lambda+m+2}} d\theta \nonumber\\
   & \qquad + x^{\lambda-1}y^\lambda \int_0^\pi \frac{(\sin \theta)^{2\lambda-1}}{(t+s+|x-y|+\sqrt{2xy(1-\cos \theta)})^{2\lambda+m+1}} d\theta\Big\} \nonumber\\
    & \quad = C \Big\{ I_1(t,s;x,y) + I_2(t,s;x,y) \Big\}, \quad t,s,x,y \in (0,\infty).
\end{align}

By proceeding as in the proof of \eqref{8prima} for $k=m+1$ we obtain
\begin{equation}\label{10.1}
I_1(t,s;x,y)
    \leq \frac{C}{(t+s+|x-y|)^{m+2}}, \quad t,s,x,y \in (0,\infty).
\end{equation}
Since $\lambda \geq 1$, we also  have that
\begin{align}\label{I2}
    I_2(t,s;x,y)
        \leq & C y\left\{\int_0^{\pi/2} \frac{\theta^{2\lambda-1}(xy)^{\lambda -1}}
                   {(t+s+|x-y|+\sqrt{xy}\theta)^{2\lambda+m+1}}  d\theta + \frac{(xy)^{\lambda-1}}{(t+s+|x-y|+\sqrt{xy})^{2\lambda+m+1}}\right\}\nonumber\\
        \leq & C y \left\{\int_0^{\pi/2} \frac{\theta}{(t+s+|x-y|+\sqrt{xy}\theta)^{m+3}}  d\theta +\frac{1}{(t+s+|x-y|+\sqrt{xy})^{m+3}}\right\}\nonumber\\
        \leq & C \left\{
                    \begin{array}{l}
                        \displaystyle \int_0^{\pi/2} \frac{\theta y}{(t+s+|x-y|+y\theta)^{m+3}}  d\theta+\frac{1}{(t+s+|x-y|)^{m+2}}, \quad 0<y \leq 2x, \ x \neq y,\\
                            \quad \\
                        \displaystyle \frac{y}{(t+s+|x-y|)^{m+3}}, \quad 0<2x \leq y,
                    \end{array}\right.\nonumber\\
        \leq &C\frac{1}{(t+s+|x-y|)^{m+2}},\quad t,s,x,y\in (0,\infty ).
\end{align}
Hence, by Minkowski's inequality we get
\begin{align*}
  \left\|\partial_x K_\lambda^\beta (\cdot; x, y) \right\|_H&\leq C \int_0^\infty s^{m-\beta-1} \left( \int_0^\infty  t^{2\beta -1}|\partial _x\partial_t^m P_{t+s}^\lambda(x,y) |^2 dt \right)^{1/2} ds \\
        &\leq C \int_0^\infty s^{m-\beta-1} \left( \int_0^\infty \frac{t^{2\beta -1}}{(t+s+|x-y|)^{2m+4}}dt \right)^{1/2} ds =\frac{C}{|x-y|^2}, \quad x,y \in (0,\infty), \ x \neq y.
\end{align*}
We conclude that
\begin{equation}\label{H6}
    \left\| \partial_x K_\lambda^\beta (\cdot ; x, y) \right\|_H
        \leq \frac{C}{|x-y|^2}, \quad x,y \in (0,\infty), \ x \neq y.
\end{equation}
By taking into account symmetries we also get
\begin{equation}\label{H7}
    \left\| \partial_y K_\lambda^\beta (\cdot, x, y) \right\|_H
        \leq \frac{C}{|x-y|^2}, \quad x,y \in (0,\infty), \ x \neq y.
\end{equation}

Fix now $x,y \in (0,\infty)$, $x \neq y$. We define the operator $L(x,y)$ by
$$\begin{array}{llll}
    L(x,y): & \B &\longrightarrow & \gamma(H,\B) \\
            & b  &\longmapsto     & L(x,y)b=K_\lambda^\beta (\cdot ; x, y)b.
  \end{array}$$
If $\{h_j\}_{j \in \N}$ is an orthonormal basis in $H$, since $\gamma(H,\C)=H$,
we deduce from \eqref{H4},
\begin{align*}
    \left\| L(x,y)b \right\|_{\gamma(H,\B)}
        = & \left\| K_\lambda^\beta (\cdot, x, y)b \right\|_{\gamma(H,\B)}
        =  \left( \E \left\| \sum_{j=1}^\infty \gamma_j \int_0^\infty K_\lambda^\beta (t; x, y) h_j(t) \frac{dt}{t} b \right\|_\B^2 \right)^{1/2} \\
        = & \|b\|_\B \left( \E \left| \sum_{j=1}^\infty \gamma_j \int_0^\infty K_\lambda^\beta (t; x, y) h_j(t) \frac{dt}{t} \right|^2 \right)^{1/2}
        \leq C \|b\|_\B \left\| K_\lambda^\beta (\cdot; x, y) \right\|_{H} \\
        \leq & C \frac{\|b\|_\B}{|x-y|}, \quad b \in \B.
\end{align*}
Hence,
\begin{equation}\label{H8a}
    \left\| L(x,y) \right\|_{\B \to \gamma(H,\B)}
        \leq \frac{C}{|x-y|}.
\end{equation}
In a similar way, by identifying $\partial_x K_\lambda^\beta (\cdot; x, y)$ and $\partial_y K_\lambda^\beta (\cdot; x, y)$ with the
corresponding operators in $L(\B,\gamma(H,\B))$, the space of linear and bounded operators from $\B$ into $\gamma(H,\B)$, and by using \eqref{H6}
and \eqref{H7},  we can get that
\begin{equation}\label{H8b}
    \left\| \partial_x K_\lambda^\beta (\cdot; x, y) \right\|_{\B \to \gamma(H,\B)}
        \leq \frac{C}{|x-y|^2},
\end{equation}
and
\begin{equation}\label{H8c}
    \left\| \partial_y K_\lambda^\beta (\cdot; x, y) \right\|_{\B \to \gamma(H,\B)}
        \leq \frac{C}{|x-y|^2}.
\end{equation}

By Theorem \ref{Th:B}, $G_\B^{\lambda ,\beta}$ is a bounded operator from $L^p((0,\infty ),\B)$ into $L^p((0,\infty ),\gamma (H,\B))$, for every $1<p<\infty$. Then, from \eqref{H5}, \eqref{H8a}, \eqref{H8b} and \eqref{H8c} we conclude that the operator $G_\B^{\lambda,\beta}$ is a $(\B,\gamma(H,\B))$-Calder\'on-Zygmund
operator. Hence, $G_\B^{\lambda,\beta}$ can be extended from $S_\lambda(0,\infty) \otimes \B$ to $L^1((0,\infty),\B)$ as a bounded operator
from $L^1((0,\infty),\B)$ into $L^{1,\infty }((0,\infty),\gamma(H,\B))$ and to $H^1((0,\infty),\B)$ as a bounded operator from $H^1((0,\infty),\B)$ into $L^1((0,\infty),\gamma (H,\B))$.
 We denote by
$\widetilde{G}_\B^{\lambda,\beta}$ the extension of $G_\B^{\lambda,\beta}$ to $L^1((0,\infty),\B)$ as a bounded operator from
$L^1((0,\infty),\B)$ into $L^{1,\infty}((0,\infty),\gamma (H,\B))$.

We now prove that
$$\widetilde{G}_\B^{\lambda,\beta}(g)(x)
    = G_\B^{\lambda,\beta}(g)(\cdot, x), \quad \text{a.e. } x \in (0,\infty),$$
as elements of $\gamma (H,\mathbb{B})$, for every $g \in L^1((0,\infty),\B)$.

Let $g \in L^1((0,\infty),\B)$. We choose a sequence $(g_k)_{k \in \N} \subset S_\lambda(0,\infty) \otimes \B$ such that
$$g_k \longrightarrow g, \quad \text{as } k \to \infty, \ \text{in } L^1((0,\infty),\B).$$
Then,
$$G_\B^{\lambda,\beta}(g_k) \longrightarrow \widetilde{G}_\B^{\lambda,\beta}(g),
    \quad \text{as } k \to \infty, \ \text{in } L^{1,\infty}((0,\infty),\gamma(H,\B)),$$
and hence, there exist a set $\Omega \subset (0,\infty)$, being $|(0,\infty) \setminus \Omega|=0$, and an increasing sequence $(n_k)_{k \in \N} \subset \N$
such that, for every $x \in \Omega$,
$$G_\B^{\lambda,\beta}(g_{n_k})(\cdot ,x) \longrightarrow \widetilde{G}_\B^{\lambda,\beta}(g)(x),
    \quad \text{as } k \to \infty, \ \text{in } \gamma(H,\B).$$
Let $\eps>0$. By proceeding as in the proof of \eqref{H4} we obtain
\begin{align*}
 \left\| K_\lambda^\beta( \cdot ;x,y) \right\|_{L^2((\eps,\infty),dt/t)}&  \leq C \int_0^\infty s^{m-\beta-1} \left( \int_\eps^\infty \frac{t^{2\beta-1}}{(t+s+|x-y|)^{2m+2}} dt \right)^{1/2}ds  \\
    & = C \int_0^\infty s^{m-\beta-1} \left( \int_0^\infty \frac{(t+\eps)^{2\beta-1}}{(t+\eps+s+|x-y|)^{2m+2}} dt \right)^{1/2} ds  \\
    &\leq  C\int_0^\infty s^{m-\beta-1} \left( \int_0^\infty \frac{t^{2\beta-1}+\eps^{2\beta-1}}{(t+\eps+s+|x-y|)^{2m+2}} dt \right)^{1/2} ds \\
    & \leq C \left( \frac{1}{\eps + |x-y|} + \frac{\eps^{\beta-1/2}}{(\eps + |x-y|)^{\beta+1/2}} \right)\leq \frac{C}{\eps}, \quad x,y \in (0,\infty).
\end{align*}
Then, by using Minkowski's inequality it follows that
\begin{align*}
     \left\| G_\B^{\lambda,\beta}(g_{n_k})(\cdot,x) - G_\B^{\lambda,\beta}(g)(\cdot,x) \right\|_{L^2((\eps,\infty),dt/t,\B)}&\leq C \int_0^\infty \|g(y)-g_{n_k}(y)\|_\B \left\| K_\lambda^\beta( \cdot ;x,y) \right\|_{L^2((\eps,\infty),dt/t)} dy  \\
    &  \leq \frac{C}{\eps}\|g-g_{n_k}\|_{L^1((0,\infty ),\B)}, \quad x \in (0,\infty),
\end{align*}
that is,
$$G_\B^{\lambda,\beta}(g_{n_k})(\cdot,x) \longrightarrow G_\B^{\lambda,\beta}(g)(\cdot,x),
    \quad \text{as } k \to \infty, \ \text{in } L^2((\eps,\infty),dt/t,\B),$$
uniformly in $x \in (0,\infty)$.

By taking into account that $\gamma(H,\B)$ is continuously contained in $L(H,\B)$, the space of bounded linear operators
from $H$ into $\B$, it follows that, for every $S \in \B^*$ and $h \in C_c^\infty(0,\infty)$, the space of smooth functions with compact
support on $(0,\infty)$, we have that
\begin{align*}
    \langle S , \widetilde{G}_\B^{\lambda,\beta}(g)(x)[h] \rangle _{\B^*,\B}
        = & \lim_{k \to \infty} \langle S , G_\B^{\lambda,\beta}(g_{n_k})(\cdot, x)[h] \rangle  _{\B^*,\B}
        =   \lim_{k \to \infty} \int_0^\infty \langle S , G_\B^{\lambda,\beta}(g_{n_k})(t,x) \rangle  _{\B^*,\B}h(t) \frac{dt}{t} \\
        = &\int_0^\infty \langle S , G_\B^{\lambda,\beta}(g)(t,x) \rangle  _{\B^*,\B}h(t) \frac{dt}{t}, \quad x \in \Omega,
\end{align*}
and then,
\begin{align*}
    \left| \int_0^\infty \langle S , G_\B^{\lambda,\beta}(g)(t,x) \rangle  _{\B^*,\B}h(t) \frac{dt}{t} \right|
        \leq \|S\|_{\B^*} \| \widetilde{G}_\B^{\lambda,\beta}(g)(x)[h] \|_\B
        \leq \|S\|_{\B^*} \| \widetilde{G}_\B^{\lambda,\beta}(g)(x) \|_{H \to \B} \|h\|_H, \quad x \in \Omega.
\end{align*}
Hence,
$\langle S, G_\B^{\lambda,\beta}(g)(\cdot,x) \rangle  _{\B^*,\B}\in H$, $x \in \Omega$, and
$$\widetilde{G}_\B^{\lambda,\beta}(g)(x)
    = G_\B^{\lambda,\beta}(g)(\cdot ,x), \quad x \in \Omega.$$
We conclude that $G_\B^{\lambda,\beta}$ is bounded from $L^1((0,\infty),\B)$ into $L^{1,\infty}((0,\infty),\gamma(H,\B))$
and from $H^1((0,\infty),\B)$ into $L^{1}((0,\infty),\gamma(H,\B))$.
\end{proof}

Next, we establish the behavior of $G_\mathbb{B}^{\lambda ,\beta}$ on $H^1_{\rm o}((0,\infty ),\B)$.

\begin{Lem}\label{Lemap13}
Let $\B$ be a UMD Banach space, $\beta >0$ and $\lambda \geq 1$. The operator $G_\mathbb{B}^{\lambda ,\beta}$ is bounded from $H^1_{\rm o}((0,\infty ),\B)$ into $L^1((0,\infty),\gamma(H,\B))$.
\end{Lem}
\begin{proof}
Let $f \in H_{\rm o}^1((0,\infty ),\B)$. By Proposition \ref{Prop1} we write $f= \sum_{j\in \N} \lambda_j a_j$, where $a_j$ is a $2$-atom and
$\lambda_j \in \C$, $j \in \N$, being $\sum_{j\in \N}|\lambda_j|<\infty$. Since the series
$\sum_{j\in \N}\lambda_j a_j$ converges in $L^1((0,\infty),\B)$ and $G_\B^{\lambda,\beta}$ is bounded from
$L^1((0,\infty),\B)$ into $L^{1,\infty}(0,\infty),\gamma(H,\B))$ (Lemma \ref{Lemap8}), we have that
\begin{equation}\label{H11}
    G_\B^{\lambda,\beta} (f)
        = \sum_{j\in \N}\lambda _jG_\B^{\lambda,\beta} (a_j),
\end{equation}
where the series converges in $L^{1,\infty}((0,\infty),\gamma(H,\B))$.

If $a$ is a $2$-atom satisfying $(Aii)$, since, by Lemma \ref{Lemap8}, $G_\B^{\lambda,\beta}$ is a bounded operator from $H^1((0,\infty),\B)$
into $L^{1}((0,\infty),\gamma(H,\B))$, we get
\begin{equation}\label{H12}
    \| G_\B^{\lambda,\beta} (a) \|_{L^{1}((0,\infty),\gamma(H,\B))}
        \leq C,
\end{equation}
being $C>0$ independent of $a$.

Suppose now that $a=b \chi_{(0,\delta)}/\delta$, for some $\delta>0$ and $b \in \B$, $\|b\|_\B=1$.
By taking into account that $G_\B^{\lambda,\beta}$ is bounded from $L^2((0,\infty),\B)$ into
$L^2((0,\infty),\gamma(H,\B))$ (Theorem \ref{Th:B}), we obtain
\begin{align}\label{H13}
     \int_0^{2\delta} \|G_\B^{\lambda,\beta}(a)(\cdot,x) \|_{\gamma(H,\B)} dx
        &\leq (2\delta)^{1/2}\|G_\B^{\lambda,\beta}(a) \|_{L^2((0,\infty),\gamma(H,\B))}
      \leq C \delta ^{1/2}\|a \|_{L^2((0,\infty),\B)}
        \leq C,
\end{align}
where $C>0$ does not depend on $\delta$ or $b$.

According to \eqref{H5}, since $\gamma(H,\C)=H$, we have that
$$\|G_\B^{\lambda,\beta}(a)(\cdot,x) \|_{\gamma(H,\B)}
    \leq \frac{1}{\delta} \int_0^\delta \| K_\lambda^\beta (\cdot; x,y) \|_H dy, \quad x \geq 2 \delta.$$

By proceeding as in \eqref{H4} and taking into account (\ref{H2}) we can write
\begin{align}\label{Kb}
    & \left\| K_\lambda^\beta( \cdot ;x,y) \right\|_H \leq C (xy)^\lambda \int_0^\infty s^{m-\beta-1} \left( \int_0^\infty \frac{t^{2\beta-1}}{(t+s+|x-y|)^{4\lambda +2m+2}} dt \right)^{1/2}ds \nonumber\\
    &\quad \leq C(xy)^\lambda \int_0^\infty \frac{s^{m-\beta -1}}{(s+|x-y|)^{2\lambda +1 +m-\beta}}ds \leq C \frac{(xy)^\lambda}{|x-y|^{2\lambda+1}}, \quad x,y \in (0,\infty), \ x \neq y.
\end{align}
Hence, it follows that
$$\|G_\B^{\lambda,\beta}(a)(\cdot,x) \|_{\gamma(H,\B)}
    \leq \frac{C}{\delta} \int_0^\delta \frac{(xy)^\lambda}{|x-y|^{2\lambda+1}} dy
    \leq \frac{C}{\delta x^{\lambda+1}} \int_0^\delta y^\lambda dy
    \leq C \frac{\delta^\lambda}{x^{\lambda+1}}, \quad x \geq 2 \delta,$$
and we get
\begin{equation}\label{H14}
    \int_{2\delta}^\infty \|G_\B^{\lambda,\beta}(a)(\cdot,x) \|_{\gamma(H,\B)} dx
        \leq C \int_{2\delta}^\infty \frac{\delta^\lambda}{x^{\lambda+1}}dx
        \leq C,
\end{equation}
where $C>0$ does not depend on $\delta$ and $b$.

From \eqref{H13} and \eqref{H14} we deduce
\begin{equation}\label{H15}
    \|G_\B^{\lambda,\beta}(a)\|_{L^1((0,\infty),\gamma(H,\B))}
        \leq C,
\end{equation}
where $C>0$ is independent of $\delta$ and $b$.

By using \eqref{H11}, \eqref{H12} and \eqref{H15} we conclude
$$ \|G_\B^{\lambda,\beta}(f)\|_{L^1((0,\infty),\gamma(H,\B))}
        \leq C \sum_{j\in \N}|\lambda_j|.$$
Hence,
$$ \|G_\B^{\lambda,\beta}(f)\|_{L^1((0,\infty),\gamma(H,\B))}
        \leq C \|f\|_{H^1_{\rm o}(\R, \B)}.$$
\end{proof}

According to \cite[Theorem 2.4]{BS2} the maximal operator $P^\lambda_*$ given by
$$P_*^\lambda(g)(x)
    = \sup_{s>0} \left\| P_s^\lambda(g)(x) \right\|_{\gamma(H,\B)},$$
for every $g \in L^{p}((0,\infty),\gamma(H,\B))$, $1 \leq p \leq \infty$,  is bounded from
$L^p((0,\infty),\gamma(H,\B))$ into $L^p(0,\infty )$, for every $1<p<\infty$, and from $L^1((0,\infty),\gamma(H,\B))$
into $L^{1,\infty}(0,\infty)$. Then, the operator $P_*^\lambda \circ G_\B^{\lambda,\beta}$ is bounded from
$H^1((0,\infty),\B)$ into $L^{1,\infty}(0,\infty)$.

We now show that $P^\lambda_* \circ G_\B^{\lambda,\beta}$ is a bounded operator from $H_{\rm o}^1((0,\infty ), \B)$
into $L^1(0,\infty)$.
\begin{Lem}\label{Lemap14}
Let $\B$ be a UMD Banach space, $\beta >0$ and $\lambda \geq 1$. We have that $P^\lambda _*\circ G_\mathbb{B}^{\lambda ,\beta}$ is a bounded operator from $H^1_{\rm o}((0,\infty ),\B)$ into $L^1(0,\infty)$.
\end{Lem}
\begin{proof}
Note firstly that $P^\lambda_* \circ G_\B^{\lambda,\beta}$ is bounded
from $H_{\rm o}^1((0,\infty ), \B)$ into $L^{1,\infty}(0,\infty)$.

According to \cite[Lemma 3.1]{BCR3}, we have that, for every $\phi \in S_\lambda(0,\infty)$,
$$h_\lambda \left( t^\beta \partial_t^\beta P_t^\lambda(\phi) \right)
    =e^{i\beta \pi} (ty)^\beta e^{-yt} h_\lambda(\phi)(y), \quad t>0.$$
Since $h_\lambda$ is an isometry in $L^2(0,\infty)$, it follows that $t^\beta \partial_t^\beta P_t^\lambda(\phi) \in L^2(0,\infty)$, for every $\phi \in S_\lambda(0,\infty)$
and $t>0$, and by \cite [\S 8.5 (19)]{Erdelyi} , we get
$$h_\lambda \left( P_s^\lambda \left[t^\beta \partial_t^\beta P_t^\lambda(\phi) \right]\right)(x)
    =e^{i\pi\beta } (tx)^\beta e^{-x(t+s)} h_\lambda(\phi)(x),
    \quad \phi \in S_\lambda(0,\infty) \text{ and } t,s,x \in (0,\infty),$$
and then, for every $\phi \in S_\lambda(0,\infty)$ and $t,s,x \in (0,\infty)$,
\begin{align*}
    P_s^\lambda \left[t^\beta \partial_t^\beta P_t^\lambda(\phi) \right](x)
        & = t^\beta \partial_t^\beta P_{t+s}^\lambda (\phi)(x)
         = \int_0^\infty t^\beta \partial_t^\beta P_{t+s}^\lambda(x,y) \phi(y) dy.
\end{align*}
Also, for every $\phi \in S_\lambda(0,\infty) \otimes \B$,
$$P_s^\lambda (G_\B^{\lambda ,\beta}(\phi)(t,\cdot ))(x)=P_s^\lambda \left[t^\beta \partial_t^\beta P_t^\lambda(\phi) \right](x)
     = t^\beta \partial_t^\beta P_{t+s}^\lambda (\phi)(x), \quad t,s,x \in (0,\infty).$$
By defining the function
$$K_\lambda^{\beta}(t,s;x,y)
    = t^\beta \partial_t^\beta P_{t+s}^\lambda(x,y), \quad t,s,x,y \in (0,\infty),$$
we have that
\begin{equation}\label{H16}
    \|K_\lambda^{\beta}(\cdot,\cdot;x,y)\|_{L^\infty((0,\infty),H)}
        \leq \frac{C}{|x-y|}, \quad x,y \in (0,\infty), \ x \neq y.
\end{equation}
Indeed, as in the proof of \eqref{H4}, by using (\ref{8prima}) we get
\begin{align*}
 &\left\| K_\lambda^\beta( \cdot , \cdot;x,y) \right\|_{L^\infty((0,\infty),H)} =\sup_{s>0}\|t^\beta \partial _t^\beta P_{t+s}^\lambda (x,y)\|_H\\
&\quad \leq C \sup_{s>0}\int_0^\infty u^{m-\beta-1}\left( \int_0^\infty \frac{t^{2\beta-1}}{(t+s+u+|x-y|)^{2m+2}} dt \right)^{1/2}du \\
    & \quad \leq C \sup_{s>0}\frac{1}{s+|x-y|} \leq \frac{C}{|x-y|}, \quad x,y \in (0,\infty), \ x \neq y.
\end{align*}
Also, we get, for every $x,y \in (0,\infty), \ x \neq y$,
\begin{equation}\label{H17}
    \|\partial_x K_\lambda^{\beta}(\cdot,\cdot;x,y)\|_{L^\infty((0,\infty),H)} +  \|\partial_y K_\lambda^{\beta}(\cdot,\cdot;x,y)\|_{L^\infty((0,\infty),H)}
        \leq \frac{C}{|x-y|^2} .
\end{equation}

According to \eqref{H16}, for every $\phi \in S_\lambda(0,\infty)$, the integral
$$\int_0^\infty K_\lambda^\beta(\cdot,\cdot;x,y) \phi(y)dy$$
is convergent in the $L^\infty((0,\infty),H)$-Bochner sense provided that $x \notin \supp \phi$.

Let $N \in \N$. We define the operator $Q_{\lambda,N}^\beta$ as follows:
$$Q_{\lambda,N}^\beta(\phi)(x)
    = \int_0^\infty K_\lambda^\beta(\cdot,\cdot;x,y) \phi(y)dy, \quad \phi \in S_\lambda(0,\infty) \text{ and } x \notin \supp \phi,$$
where the integral is understood in the $C([1/N,N],H)$-Bochner sense. Here, by $C([1/N,N],H)$ we denote the space of
$H$-valued continuous functions on $[1/N,N]$.

Let $\phi \in S_\lambda(0,\infty)$. We have that
$$\left[Q_{\lambda,N}^\beta(\phi)(x)\right](s)
    = \int_0^\infty K_\lambda^\beta(\cdot,s;x,y) \phi(y)dy, \quad  x \notin \supp \phi \text{ and } s \in [1/N,N],$$
where the integral is understood in the $H$-Bochner sense and the equality is understood in $H$.

Moreover, according to some properties of Bochner integration and by applying Fubini's theorem, we get
\begin{align*}
    & \int_0^\infty \left[\left[Q_{\lambda,N}^\beta(\phi)(x)\right](s)\right](t) h(t) \frac{dt}{t}
        = \int_0^\infty \phi(y) \int_0^\infty K_\lambda^\beta(t,s;x,y) h(t) \frac{dt}{t} dy,\\
    & \qquad = \int_0^\infty  h(t)  \int_0^\infty K_\lambda^\beta(t,s;x,y) \phi(y)dy\frac{dt}{t},  \quad  x \notin \supp \phi \text{ and } s \in [1/N,N].
\end{align*}
Then,
$$\left[\left[Q_{\lambda,N}^\beta(\phi)(x)\right](s)\right](t)
    = \int_0^\infty K_\lambda^\beta(t,s;x,y) \phi(y)dy ,  \quad  x \notin \supp \phi \text{ and } s \in [1/N,N],$$
as elements of $H$.

We define $Q_{\lambda,N}^\beta$ on $S_\lambda(0,\infty) \otimes \B$ in the natural way. For every $\phi \in S_\lambda(0,\infty)\otimes \B$ we have that
\begin{equation}\label{H18}
    \left[\left[Q_{\lambda,N}^\beta(\phi)(x)\right](s)\right](t)
        = P_s^\lambda \left( G_\B^{\lambda,\beta}(\phi)(t,\cdot) \right)(x) ,  \quad  x \notin \supp \phi \text{ and } s \in [1/N,N],
\end{equation}
in the sense of equality in $\gamma(H,\B)$.

Since $P_*^\lambda$ is bounded from $L^2((0,\infty),\gamma(H,\B))$ into $L^2(0,\infty )$ and $G_\B^{\lambda,\beta}$ is bounded from
$L^2((0,\infty),\B)$ into $L^2((0,\infty),\gamma(H,\B))$ (Theorem \ref{Th:B}), the operator $P_*^\lambda \circ G_\B^{\lambda,\beta}$ is bounded from
$L^2((0,\infty),\B)$ into $L^2(0,\infty)$. Hence, the operator
$$Z_{\lambda,N}^\beta(f)(t,s;x)
    = P_s^\lambda \left( G_\B^{\lambda,\beta}(f)(t,\cdot) \right)(x), \quad t,x \in (0,\infty), \ s \in [1/N,N],$$
is bounded from $L^2((0,\infty),\B)$ into $L^2((0,\infty), C([1/N,N],\gamma(H,\B)))$. Moreover,
$$\sup_{M \in  \N} \| Z_{\lambda,M}^\beta\|_{L^2((0,\infty),\B) \to L^2((0,\infty), C([1/M,M],\gamma(H,\B)))} < \infty.$$

By taking into account \eqref{H16}, \eqref{H17} and \eqref{H18} and by using vector valued Calder\'on-Zygmund theory we conclude that the operator
$Z_{\lambda,N}^\beta$ can be extended to $L^1((0,\infty),\B)$ as a bounded operator from $L^1((0,\infty),\B)$ into $L^{1,\infty}((0,\infty), C([1/N,N],\gamma(H,\B)))$ and
to $H^1((0,\infty),\B)$ as a bounded operator from $H^1((0,\infty),\B)$ into $L^{1}((0,\infty), C([1/N,N],\gamma(H,\B)))$ .
We denote by $\widetilde{Z}_{\lambda,N}^\beta$ to this extension of $Z_{\lambda,N}^\beta$ to $L^1((0,\infty),\B)$. It has that
$$
\sup_{M \in  \N} \| \widetilde{Z}_{\lambda,M}^\beta\|_{L^1((0,\infty),\B) \to L^{1,\infty}((0,\infty), C([1/M,M],\gamma(M,\B)))} < \infty,
$$
and
$$
\sup_{M \in  \N} \| \widetilde{Z}_{\lambda,M}^\beta\|_{H^1((0,\infty),\B) \to L^{1}((0,\infty), C([1/M,M],\gamma(H,\B)))} < \infty.
$$
Our objective now is to show that
$$\widetilde{Z}_{\lambda,N}^\beta (f)
    = P_s^\lambda \left( G_\B^{\lambda,\beta}(f)\right), \quad f \in L^1((0,\infty),\B),$$
as elements in $L^{1,\infty }((0,\infty),C([1/N,N],\gamma(H,\B)))$.

Let $f \in L^1((0,\infty),\B)$. We choose a sequence $(f_k)_{k\in \N}\subset S_\lambda(0,\infty) \otimes \B$ such that
$$f_k \longrightarrow f, \quad \text{as } k \to \infty, \text{ in } L^1((0,\infty),\B).$$
Then,
$$Z_{\lambda,N}^\beta(f_k) \longrightarrow \widetilde{Z}_{\lambda,N}^\beta(f),
    \quad \text{as } k \to \infty, \text{ in } L^{1,\infty}((0,\infty),C([1/N,N],\gamma(H,\B))).$$

It is not hard to see that, for every $t,x \in (0,\infty)$ and $s \in [1/N,N]$,
$$Z_{\lambda,N}^\beta(f)(t,s;x)
    = \int_0^\infty K_\lambda^\beta(t,s;x,y)f(y)dy
    = G_\B^{\lambda,\beta}(P_s^\lambda (f))(t,x).$$
We know that, for every $s \in (0,\infty)$, $P_s^\lambda$ is a bounded operator from $L^1((0,\infty),\B)$ into itself,
and that the operator $G_\B^{\lambda,\beta}$ is bounded from $L^1((0,\infty),\B)$ into $L^{1,\infty}((0,\infty),\gamma(H,\B))$ (Lemma \ref{Lemap8}).
Then, it follows that, for every $s \in [1/N,N]$,
$$Z_{\lambda,N}^\beta(f_k)(\cdot,s;\cdot) \longrightarrow Z_{\lambda,N}^\beta(f)(\cdot,s;\cdot),
    \quad \text{as } k \to \infty, \text{ in } L^{1,\infty}((0,\infty),\gamma(H,\B)).$$
Hence, we can find an increasing sequence $(n_k)_{k \in \N} \subset \N$  and a subset $\Omega$ of $(0,\infty)$
such that $|(0,\infty )\setminus \ \Omega|=0$, and
$$Z_{\lambda,N}^\beta(f_{n_k})(\cdot,s;x) \longrightarrow Z_{\lambda,N}^\beta(f)(\cdot,s;x),
    \quad \text{as } k \to \infty, \text{ in } \gamma(H,\B),$$
for every $x \in \Omega$ and $s \in \Q \cap [1/N,N]$. Here $\Omega$ and $(n_k)_{k \in \N}$ do not depend on $N$.

Also, there exists an increasing sequence $(k_j)_{j \in \N} \subset \N$ and a subset $W$ of $\Omega$ with $|\Omega |=|W|$, such that
$$Z_{\lambda,N}^\beta(f_{n_{k_j}})(\cdot,s;x) \longrightarrow [\widetilde{Z}_{\lambda,N}^\beta(f)(x)](s),
    \quad \text{as } j \to \infty, \text{ in } \gamma(H,\B),$$
for every $x \in W$ and $s \in \Q \cap [1/N,N]$. Again, $W$ and $(k_j)_{j \in \N}$ do not depend on $N$.

We conclude that
$$Z_{\lambda,N}^\beta(f)(\cdot,s,x)
    = [\widetilde{Z}_{\lambda,N}^\beta(f)(x)](s), \quad x \in W, \ s \in \Q \cap [1/N,N].$$
This equality is understood in $\gamma(H,\B)$.

Hence, we can write
\begin{align*}
    & \left| \left\{ x \in (0,\infty) : P_*^\lambda( G_\B^{\lambda,\beta}(f))(x) > \alpha \right\} \right| \\
    & \qquad = \left| \bigcup_{M \in \N}\left\{ x \in (0,\infty) : \sup_{s \in [1/M,M]} \left\|P_s^\lambda ( G_\B^{\lambda,\beta}(f))(x) \right\|_{\gamma(H,\B)} > \alpha \right\} \right| \\
    & \qquad \leq \lim_{M \to \infty} \left| \left\{ x \in (0,\infty) : \sup_{s \in [1/M,M] \cap \Q} \left\|P_s^\lambda (G_\B^{\lambda,\beta}(f))(x) \right\|_{\gamma(H,\B)} > \alpha \right\} \right| \\
    & \qquad \leq \lim_{M \to \infty} \left| \left\{ x \in W : \sup_{s \in [1/M,M] \cap \Q} \left\|\widetilde{Z}_{\lambda,M}^\beta (f)(\cdot,s,x) \right\|_{\gamma(H,\B)} > \alpha \right\} \right| \\
    & \qquad \leq \frac{C}{\alpha} \|f\|_{L^1((0,\infty),\B)}, \quad \alpha>0.
\end{align*}
Thus, we prove that the operator $Z_\lambda^\beta$ defined by
$$Z_{\lambda}^\beta(f)(t,s;x)
    = P_s^\lambda \left(G_\B^{\lambda,\beta}(f)(t,\cdot )\right)(x), \quad s,t,x \in (0,\infty),$$
is bounded from $L^1((0,\infty),\B)$ into $L^{1,\infty }((0,\infty), L^\infty((0,\infty),\gamma(H,\B)))$.

By proceeding in a similar way, since $L^\infty_c(0,\infty) \otimes \B$ is a dense subset of
$H^1((0,\infty),\B)$, we can see that $Z_{\lambda}^\beta$ defines a bounded operator from
$H^1((0,\infty),\B)$ into $L^1((0,\infty), L^\infty((0,\infty),\gamma(H,\B)))$. Here $L^\infty_c(0,\infty)$ represents the space of bounded measurable functions with compact support in $(0,\infty )$.

Thus, if $a$ is a $2$-atom satisfying $(Aii)$ we get that
$$\left\| P_s^\lambda \left(G_\B^{\lambda,\beta}(a)\right) \right\|_{L^1((0,\infty), L^\infty((0,\infty),\gamma(H,\B)))}
    \leq C,$$
where $C$ does not depend on $a$.

On the other hand, by using \eqref{H2} it follows that
\begin{equation}\label{H19}
    \|K_\lambda^\beta(\cdot,\cdot;x,y)\|_{L^\infty((0,\infty),H)}
        \leq C \frac{(xy)^\lambda}{|x-y|^{2\lambda+1}}, \quad x,y \in (0,\infty), \ x \neq y.
\end{equation}
The, since the operator $Z_{\lambda}^\beta$ is bounded from $L^2((0,\infty),\B)$ into
$L^2((0,\infty), L^\infty((0,\infty),\gamma(H,\B)))$, by proceeding as in the proof of \eqref{H15}
we can deduce that there exists $C>0$ such that, for every $\delta>0$, and $b\in \B$, $\|b\|_\B=1$,
$$\left\| P_s^\lambda \left(G_\B^{\lambda,\beta}(a)\right) \right\|_{L^1((0,\infty), L^\infty((0,\infty),\gamma(H,\B)))}
    \leq C,$$
when $a=b\chi_{(0,\delta)}/\delta$.

We conclude that, for every $f \in H_o^1((0,\infty ),\B)$,
$$\|P_*^\lambda(G_\B^{\lambda,\beta}(f))\|_{L^1(0,\infty )}
    \leq C \|f\|_{H_{\rm o}^1((0,\infty ),\B)}.$$
\end{proof}

%\newpage
\subsection{}\label{subsec:2.2}
We show now that $G_\B^{\lambda,\beta}$ is bounded from $BMO_{\rm o}((0,\infty ),\B)$ into $BMO_{\rm o}((0,\infty ),\gamma(H,\B))$. This requires verifying the corresponding vector-valued conditions ($Bi$) and ($Bii$) which we collected in Lemma \ref{Lemap17} and Lemma \ref{Lemap19}, respectively.

\begin{Lem}\label{Lemap17}
Consider $\B$ a UMD Banach space and $\beta,\lambda >0$. There exists $C>0$ such that, for every $r>0$,
\begin{equation}\label{H22}
\frac{1}{r}\int_0^r\|G_\B^{\lambda ,\beta}(f)(\cdot ,x)\|_{\gamma (H,\B)}dx\leq C\|f\|_{BMO_{\rm o}((0,\infty ),\B)},\quad f\in BMO_{\rm o}((0,\infty ),\B).
\end{equation}
\end{Lem}
\begin{proof}
Assume that $f \in BMO_{\rm o}((0,\infty ),\B)$. According to \eqref{H2}, for every $k \in \N$, we have that
$$\left| \partial_t^k P_t^\lambda(x,y) \right|
    \leq C \frac{(xy)^\lambda}{(t+|x-y|)^{2\lambda+1+k}}
    \leq \frac{C}{t^k} \frac{(xy)^\lambda}{(t+|x-y|)^{2\lambda+1}},\quad t,x,y \in (0,\infty).$$
Then, for every
\begin{align}\label{H19A}
    & \int_0^\infty \left| \partial_t^k P_t^\lambda(x,y) \right| \|f(y)\|_\B dy
        \leq \frac{C}{t^k} \int_0^\infty  \frac{(xy)^\lambda}{(t+|x-y|)^{2\lambda+1}} \|f(y)\|_\B dy \nonumber \\
    & \qquad \leq \frac{C}{t^k} \left\{ \int_0^{2x}  \frac{x^{2\lambda}}{t^{2\lambda+1}} \|f(y)\|_\B dy
                                      + \int_{2x}^\infty \frac{(xy)^\lambda}{y^{2\lambda+1}} \|f(y)\|_\B dy \right\} \nonumber \\
    & \qquad \leq \frac{C}{t^k} \left\{  \left( \frac{x}{t} \right)^{2\lambda+1}\|f\|_{BMO_{\rm o}((0,\infty ),\B)}
                                      + \sum_{j=1}^\infty x^\lambda \int_{2xj^{2/\lambda}}^{2x(j+1)^{2/\lambda}} \frac{1}{y^{\lambda+1}} \|f(y)\|_\B dy \right\} \nonumber \\
    & \qquad \leq \frac{C}{t^k} \left\{  \left( \frac{x}{t} \right)^{2\lambda+1}\|f\|_{BMO_{\rm o}((0,\infty ),\B)}
                                      + \sum_{j=1}^\infty \frac{x^\lambda}{(xj^{2/\lambda})^{\lambda+1}} \int_0^{2x(j+1)^{2/\lambda}} \|f(y)\|_\B dy \right\} \nonumber \\
    & \qquad \leq \frac{C}{t^k} \left\{  \left( \frac{x}{t} \right)^{2\lambda+1}
                                      + \sum_{j=1}^\infty \frac{1}{j^2} \right\} \|f\|_{BMO_o((0,\infty ),\B)}, \quad t,x \in (0,\infty) \text{ and } k \in \N .
\end{align}
We can write
$$\partial_t^k P_t^\lambda(f)(x)
    = \int_0^\infty \partial_t^k P_t^\lambda(x,y) f(y)dy, \quad t,x \in (0,\infty) \text{ and } k \in \N.$$
By using again \eqref{H2}, if $m \in \N$ such that $m-1 \leq \beta < m$, we get
\begin{align*}
    & \int_0^\infty s^{m-\beta-1} \int_0^\infty \left| \partial_t^m P_{t+s}^\lambda(x,y) \right| \|f(y)\|_\B dy ds \\
    & \qquad \leq C \|f\|_{BMO_{\rm o}(\R,\B)} \int_0^\infty \frac{s^{m-\beta-1}}{(t+s)^m}  \left\{ \left( \frac{x}{t+s}\right)^{2\lambda+1} +1 \right\}ds
    < \infty, \quad t,x \in (0,\infty).
\end{align*}
This leads to
$$G_\B^{\lambda,\beta}(f)(t,x)
    = \int_0^\infty K_\lambda^\beta(t;x,y) f(y)dy, \quad t,x \in (0,\infty),$$
where $K_\lambda^\beta(t;x,y)= t^\beta \partial_t^\beta P_t^\lambda(x,y)$, $t,x,y \in (0,\infty)$.

Let $r>0$. We split $G_\B^{\lambda,\beta}(f)(t,x)$ as follows
$$G_\B^{\lambda,\beta}(f)(t,x)
    = G_{\B,1}^{\lambda,\beta}(f)(t,x) + G_{\B,2}^{\lambda,\beta}(f)(t,x), \quad t,x \in (0,\infty),$$
being
$$G_{\B,1}^{\lambda,\beta}(f)(t,x)
    = \int_0^{2r} K_\lambda^\beta(t;x,y) f(y)dy, \quad t,x \in (0,\infty). $$
Since $G_\B^{\lambda,\beta}$ is a bounded operator from $L^2((0,\infty),\B)$ into $L^2((0,\infty),\gamma(H,\B))$
(Theorem \ref{Th:B}), we obtain
\begin{align}\label{H20}
     \frac{1}{r} \int_0^r \| G_{\B,1}^{\lambda,\beta}(f)(\cdot,x) \|_{\gamma(H,\B)} dx
        & \leq \left( \frac{1}{r} \int_0^\infty \| G_{\B}^{\lambda,\beta}(f \chi_{(0,2r)})(\cdot,x) \|_{\gamma(H,\B)}^2 dx \right)^{1/2} \nonumber \\
        & \leq C \left( \frac{1}{r} \int_0^{2r} \| f(y) \|_{\B}^2 dy \right)^{1/2}
          \leq C \|f\|_{BMO_o((0,\infty ),\B)}.
\end{align}
Note that the last inequality follows from John-Niremberg's property.

If $h \in H$, by using \eqref{Kb} and proceeding as in (\ref{H19A}) we have
\begin{align*}
    & \int_0^\infty \int_{2r}^\infty |K_\lambda^\beta(t;x,y)| \|f(y)\|_{\B} dy |h(t)| \frac{dt}{t}
        \leq \|h\|_H \int_{2r}^\infty \|K_\lambda^\beta(\cdot ;x,y)\|_H \|f(y)\|_\B dy \\
    & \qquad \leq C \|h\|_H \int_{2r}^\infty  \frac{(xy)^\lambda}{|x-y|^{2\lambda+1}} \|f(y)\|_\B dy
        \leq C \|h\|_H x^\lambda \int_{2r}^\infty  \frac{1}{y^{\lambda+1}} \|f(y)\|_\B dy \\
    & \qquad \leq C \|h\|_H \|f\|_{BMO_{\rm o}((0,\infty ),\B)}, \quad x \in  (0,r).
\end{align*}
Then, if $(h_j)_{j=1}^n$ is a set of orthonormal functions in $H$, we can write
\begin{align*}
    & \left( \E \left\| \sum_{j=1}^n \gamma_j \int_0^\infty h_j(t) \int_{2r}^\infty K_\lambda^\beta(t;x,y) f(y) dydt\right\|_\B^2 \right)^{1/2} \\
    & \qquad = \left( \E \left\|  \int_{2r}^\infty f(y) \sum_{j=1}^n \gamma_j \int_0^\infty  K_\lambda^\beta(t;x,y) h_j(t)dt dy\right\|_\B^2 \right)^{1/2} \\
    & \qquad \leq \int_{2r}^\infty \|f(y)\|_\B \left( \E \left| \sum_{j=1}^n \gamma_j \int_0^\infty  K_\lambda^\beta(t;x,y) h_j(t)dt \right|^2 \right)^{1/2}dy \\
    & \qquad \leq \int_{2r}^\infty \|f(y)\|_\B \|K_\lambda^\beta(\cdot;x,y)\|_{\gamma(H,\C)}dy.
\end{align*}
Since $\gamma(H,\C)=H$ and again by (\ref{Kb}) we get, for each $x \in (0,r)$,
\begin{align*}
     \| G_{\B,2}^{\lambda,\beta}(f)\|_{\gamma (H,\mathbb{B})}&=\left\| \int_{2r}^\infty K_\lambda^\beta(\cdot;x,y) f(y) dy \right\|_{\gamma(H,\B)}
        \leq C \int_{2r}^\infty \|f(y)\|_{\B}\| K_\lambda^\beta(\cdot;x,y)\|_{H}  dy
        \leq C \|f\|_{BMO_{\rm o}((0,\infty ),\B)} .
\end{align*}
Hence,
\begin{equation}\label{H21}
    \frac{1}{r} \int_0^r \| G_{\B,2}^{\lambda,\beta}(f)(\cdot,x) \|_{\gamma(H,\B)} dx
         \leq C \|f\|_{BMO_{\rm o}((0,\infty ),\B)}.
\end{equation}
From \eqref{H20} and \eqref{H21} we conclude the proof of this Lemma.
\end{proof}

 Note that \eqref{H22} implies that $\| G_{\B}^{\lambda,\beta}(f)(\cdot,x) \|_{\gamma(H,\B)}< \infty$, a.e. $x \in (0,\infty).$

\begin{Lem}\label{Lemap19}
Let $\B$ be a UMD Banach space and $\beta,\lambda >0$. The operator $G_\B^{\lambda,\beta}$ is bounded from $BMO_{\rm o}((0,\infty ),\B)$ into $BMO((0,\infty), \gamma (H,\B))$.
\end{Lem}
\begin{proof}
Let $f\in BMO_{\rm o}((0,\infty ),\B)$. We consider the odd extension function $f_{\rm o}$ of $f$ to $\R$ and
$$G_\B^\beta (f_{\rm o})(t,x)
    = \int_\R K^\beta(t;x,y) f_{\rm o}(y)dy, \quad t \in (0,\infty) \text{ and } x \in \R,$$
where
$$K^\beta(t;x,y)
    = t^\beta \partial_t^\beta P_t(x-y), \quad t \in (0,\infty) \text{ and } x,y \in \R. $$
Here $P_t(z)=t/[\pi(t^2+z^2)]$, $t \in (0,\infty)$ and $z \in \R$, is the classical Poisson semigroup.

We can write
\begin{align*}
    G_\B^\beta (f_{\rm o})(t,x)
        = & \int_0^\infty t^\beta \partial_t^\beta P_t(x-y) f(y)dy - \int_0^\infty t^\beta \partial_t^\beta P_t(x+y) f(y)dy \\
        = & \int_{x/2}^{2x} t^\beta \partial_t^\beta P_t(x-y) f(y)dy - \int_{x/2}^{2x} t^\beta \partial_t^\beta P_t(x+y) f(y)dy \\
          & + \int_0^{x/2} t^\beta \partial_t^\beta \left[P_t(x-y) - P_t(x+y) \right] f(y)dy
            + \int_{2x}^\infty t^\beta \partial_t^\beta \left[P_t(x-y) - P_t(x+y) \right] f(y)dy \\
        = & \sum_{j=1}^4 I_j(f)(t,x), \quad t,x \in (0,\infty).
\end{align*}
In \cite[Lemma 1]{BCCFR1} it was established that, if $m \in \N$ is such that $m-1 \leq \beta < m$,
$$t^\beta \partial_t^\beta P_t(z)
    = \sum_{k\in \N, \,0\leq k\leq (m+1)/2} \frac{c_k}{t} \varphi^k\left( \frac{z}{t}\right), \quad t \in (0,\infty) \text{ and } z \in \R,$$
where, for every $k \in \N$, $0 \leq k \leq (m+1)/2$, $c_k \in \C$ and
$$\varphi^k(z)
    = \int_0^\infty \frac{(1+v)^{m+1-2k}v^{m-\beta-1}}{[(1+v)^2+z^2]^{m-k+1}} dv, \quad z \in \R.$$
Let $k \in \N$ such that $0 \leq k \leq (m+1)/2$. We have that
\begin{align*}
    & \frac{1}{t}\varphi^k\left( \frac{x+y}{t}\right) - \frac{1}{t}\varphi^k\left( \frac{x-y}{t}\right)
        = t^{2(m-k)+1} \int_0^\infty (1+v)^{m+1-2k}v^{m-\beta-1}\\
    & \qquad \qquad \times \left[ \frac{1}{[(1+v)^2t^2+(x+y)^2]^{m-k+1}} - \frac{1}{[(1+v)^2t^2+(x-y)^2]^{m-k+1}} \right] dv \\
    & \qquad =t^{2(m-k)+1}  \sum_{\ell=0}^{m-k+1} {{m-k+1}\choose {\ell}}\int_0^\infty (1+v)^{m+1-2k}v^{m-\beta-1}\\
    & \qquad \qquad \times \frac{[(1+v)t]^{2(m-k+1-\ell)}[(x-y)^{2\ell}-(x+y)^{2\ell}]}{[\left((1+v)^2t^2+(x+y)^2\right)\left((1+v)^2t^2+(x-y)^2\right)]^{m-k+1}} dv, \quad t,x,y \in (0,\infty).
\end{align*}
We get
\begin{align*}
    & \left|\frac{1}{t}\varphi^k\left( \frac{x+y}{t}\right) - \frac{1}{t}\varphi^k\left( \frac{x-y}{t}\right) \right|
        \leq C y t^{2(m-k)+1} \sum_{\ell=1}^{m-k+1} x^{2\ell-1} \int_0^\infty (1+v)^{m+1-2k}v^{m-\beta-1} \\
    & \qquad \qquad \times \frac{[(1+v)t]^{2(m-k+1-\ell)}}{[(1+v)t + x]^{4(m-k+1)}} dv\\
    & \qquad \leq C y t^{2(m-k)+1} \int_0^\infty \frac{(1+v)^{m+1-2k}v^{m-\beta-1}}{[(1+v)t + x]^{2(m-k)+3}} dv, \quad t\in (0,\infty), \ 0<y<x/2.
\end{align*}
Minkowski's inequality leads to
\begin{align*}
    \left\|\frac{1}{t}\varphi^k\left( \frac{x+y}{t}\right) - \frac{1}{t}\varphi^k\left( \frac{x-y}{t}\right) \right\|_H
        \leq & C y  \int_0^\infty (1+v)^{m+1-2k}v^{m-\beta-1} \\
             & \times \left(\int_0^\infty \frac{t^{4(m-k)+1}}{[(1+v)t + x]^{4(m-k)+6}} dt \right)^{1/2} dv \\
        \leq & C \frac{y}{x^2} \int_0^\infty \frac{v^{m-\beta-1}}{(1+v)^{m}} dv \\
        \leq & C \frac{y}{x^2}, \quad 0<y<x/2.
\end{align*}
Hence, we obtain
$$ \| t^\beta \partial_t^\beta [P_t(x-y)-P_t(x+y)]\|_H
    \leq C \frac{y}{x^2}, \quad 0<y<x/2.$$
It follows that
\begin{equation}\label{H23}
    \|I_3(f)(\cdot,x)\|_{\gamma(H,\B))}
        \leq \frac{C}{x^2} \int_0^{x/2} y \|f(y)\|_\B dy
        \leq C \|f\|_{BMO_{\rm o}((0,\infty ),\B)}, \quad x \in (0,\infty).
\end{equation}
By symmetries, we also get
$$ \| t^\beta \partial_t^\beta [P_t(x-y)-P_t(x+y)]\|_H
    \leq C \frac{x}{y^2}, \quad 0<2x<y,$$
and then
\begin{align}\label{H24}
    \|I_4(f)(\cdot,x)\|_{\gamma(H,\B))}
        \leq & C x \int_{2x}^\infty \frac{\|f(y)\|_\B}{y^2} dy
        \leq C x \sum_{k=1}^\infty \int_{2k^2x}^{2(k+1)^2x} \frac{\|f(y)\|_\B}{y^2} dy \nonumber \\
        \leq & C \sum_{k=1}^\infty \frac{1}{xk^4} \int_0^{2(k+1)^2x} \|f(y)\|_\B dy
        \leq C \|f\|_{BMO_{\rm o}((0,\infty ),\B)}, \quad x \in (0,\infty).
\end{align}
Note that to establish
$$\|I_3(f)(\cdot,x)\|_{\gamma(H,\B))}
    \leq C \|f\|_{BMO_{\rm o}((0,\infty ),\B)}, \quad x \in (0,\infty),$$
it is enough to use that
$$\| t^\beta \partial_t^\beta [P_t(x-y)-P_t(x+y)]\|_H
    \leq \| t^\beta \partial_t^\beta P_t(x-y)\|_H + \| t^\beta \partial_t^\beta P_t(x+y)\|_H
    \leq \frac{C}{x}, \quad 0<y<\frac{x}{2}. $$
However, the estimation
$$\| t^\beta \partial_t^\beta [P_t(x-y)-P_t(x+y)]\|_H
    \leq \frac{C}{y}, \quad 0<2x<y, $$
does not allow to show that
$$\|I_4(f)(\cdot,x)\|_{\gamma(H,\B))}
    \leq C \|f\|_{BMO_{\rm o}((0,\infty ),\B)}, \quad x \in (0,\infty).$$

Also, we obtain
\begin{align*}
    \|t^\beta \partial_t^\beta P_t(x+y)\|_H
        \leq & C \sum_{k\in \N, 0\leq k\leq (m+1)/2} \int_0^\infty (1+v)^{m+1-2k}v^{m-\beta-1}
                    \left(\int_0^\infty \frac{t^{4(m-k)+1}}{[(1+v)t + x+y]^{4(m-k)+4}} dt \right)^{1/2} dv \\
        \leq & \frac{C}{x+y}, \quad x,y \in (0,\infty).
\end{align*}
Hence,
$$\|I_2(f)(\cdot,x)\|_{\gamma(H,\B))}
    \leq \frac{C}{x} \int_{x/2}^{2x} \|f(y)\|_\B dy
    \leq C \|f\|_{BMO_{\rm o}((0,\infty ),\B)}, \quad x \in (0,\infty).$$
We conclude that
$$G_\B^\beta(f_{\rm o})(t,x) - \int_{x/2}^{2x} t^\beta \partial_t^\beta P_t(x-y) f(y)dy \in L^\infty \left( (0,\infty), \gamma(H,\B) \right),$$
and we deduce that $G_\B^\beta(f_{\rm o}) \in BMO((0,\infty),\gamma(H,\B))$ if, and only if, $G_{\B,{\rm loc}}^\beta(f) \in BMO((0,\infty),\gamma(H,\B))$,
where
$$G_{\B,{\rm loc}}^\beta(f)(t,x)
    = \int_{x/2}^{2x} t^\beta \partial_t^\beta P_t(x-y) f(y)dy, \quad t,x \in (0,\infty).$$

We are going to prove that $G_\B^\beta(f_{\rm o}) \in BMO((0,\infty),\gamma(H,\B))$. Let $0<r<s<\infty$.
We define $I=(r,s)$, $x_I=(r+s)/2$ and $d_I=(s-r)/2$. We decompose $f_{\rm o}$ as follows:
$$f_{\rm o}
    =(f_{\rm o}-f_I)\chi_{2I} + (f_{\rm o}-f_I)\chi_{(0,\infty )\setminus 2I} + f_I
    = f_1 + f_2 + f_3,$$
where $2I=(x_I-2d_I,x_I+2d_I)$. By \cite[Theorem 4.2]{KaWe} and \cite[proof of Theorem 1.2]{BCR3}, it follows that the operator $G_\B^\beta$ is bounded from $L^2(\R,\B)$ into $L^2(\R, \gamma(H,\B))$. Then,
\begin{align}\label{H25}
    & \frac{1}{|I|} \int_I \| G_\B^\beta(f_1)(\cdot,x)\|_{\gamma(H,\B)} dx
        \leq \left( \frac{1}{|I|} \int_I \| G_\B^\beta(f_1)(\cdot,x)\|_{\gamma(H,\B)}^2 dx \right)^{1/2} \nonumber\\
    & \qquad \leq C \left( \frac{1}{|I|} \int_{2I} \| f_{\rm o}(x)-f_I\|_\B^2 dx \right)^{1/2}
        \leq C \|f\|_{BMO_{\rm o}((0,\infty ),\B)}.
\end{align}
Hence,
$$\| G_\B^\beta(f_1)(\cdot,x)\|_{\gamma(H,\B)}
    < \infty, \quad \text{a.e. } x \in (0,\infty).$$

On the other hand, since
$$\int_\R P_t(x-y) dy = 1, \quad t \in (0,\infty) \text{ and } x \in \R,$$
if $m \in \N$, being $m-1 \leq \beta < m$,
\begin{align}\label{H26}
    \partial_t^\beta \int_\R P_t(x-y) dy
        = & \frac{e^{-i\pi(m-\beta)}}{\G(m-\beta)} \int_0^\infty s^{m-\beta-1} \partial_t^m \int_\R P_{t+s}(x-y) dy ds
        =0, \quad t \in (0,\infty) \text{ and } x \in \R.
\end{align}
It follows that $G_\B^\beta(f_3)=0$.

In \cite[Section 3.1]{BCR3} it was proved that
$$\| K_\lambda^\beta(\cdot;x,y) - K^\beta(\cdot;x,y)\|_H
    \leq \frac{C}{x}, \quad x/2 < y <2x, \ x \in (0,\infty).$$
Then,
\begin{align}\label{H26A}
    & \Big\| \int_{x/2}^{2x} \left( K_\lambda^\beta(\cdot;x,y) - K^\beta(\cdot;x,y) \right) f(y) dy \Big\|_{\gamma(H,\B)}
        \leq  \int_{x/2}^{2x} \Big\| K_\lambda^\beta(\cdot;x,y) - K^\beta(\cdot;x,y) \Big\|_{H} \|f(y)\|_\B dy \nonumber \\
    & \qquad \leq \frac{C}{x} \int_{x/2}^{2x} \|f(y)\|_\B dy
             \leq C \|f\|_{BMO_{\rm o}((0,\infty ),\B)}, \quad x \in (0,\infty).
\end{align}
Moreover, by using \eqref{Kb} and as it was seen in \eqref{H19A},
\begin{equation}\label{H26B}
     \Big\| \int_{2x}^\infty K_\lambda^\beta(\cdot;x,y) f(y) dy \Big\|_{\gamma(H,\B)}
        \leq  C x^\lambda \int_{2x}^\infty \|f(y)\|_\B \frac{dy}{y^{\lambda+1}} \leq  C \|f\|_{BMO_{\rm o}((0,\infty ),\B)}, \quad x \in (0,\infty),
\end{equation}
and
\begin{equation}\label{H26C}
     \Big\| \int_0^{x/2} K_\lambda^\beta(\cdot;x,y) f(y) dy \Big\|_{\gamma(H,\B)}
        \leq  \frac{C}{x} \int_0^{x/2} \|f(y)\|_\B dy  \leq  C \|f\|_{BMO_{\rm o}((0,\infty ),\B)}, \quad x \in (0,\infty).
\end{equation}
Hence,
$$\| G_\B^{\lambda,\beta}(f)(\cdot,x) - G_\B^{\beta}(f_{\rm o})(\cdot,x)\|_{\gamma(H,\B)}
    \leq C \|f\|_{BMO_{\rm o}((0,\infty ),\B)}, \quad x \in (0,\infty).$$
Since, $\| G_\B^{\lambda,\beta}(f)(\cdot,x)\|_{\gamma(H,\B)}<\infty$ and
$\| G_\B^{\beta}(f_1)(\cdot,x)\|_{\gamma(H,\B)}<\infty$, a.e. $x \in (0,\infty)$,
also
$$\| G_\B^{\beta}(f_2)(\cdot,x)\|_{\gamma(H,\B)}<\infty, \quad \text{a.e. } x \in (0,\infty).$$

By proceeding as in the proof of \eqref{H6} we get
\begin{equation}\label{H27}
    \|\partial_x K^\beta(\cdot;x,y)\|_H
        \leq \frac{C}{|x-y|^2}, \quad x,y \in (0,\infty), \ x \neq y.
\end{equation}

We choose $x_0 \in I$ such that $\| G_\B^{\beta}(f_2)(\cdot,x_0)\|_{\gamma(H,\B)}<\infty$.
By using \eqref{H27} we obtain
\begin{align*}
    & \frac{1}{|I|} \int_I \| G_\B^{\beta}(f_2)(\cdot,x) -  G_\B^{\beta}(f_2)(\cdot,x_0)\|_{\gamma(H,\B)} dx \\
    & \qquad \leq \frac{C}{|I|} \int_I \int_{(0,\infty )\setminus 2I} \|f_{\rm o}(y)-f_I\|_\B \| K^\beta(\cdot;x,y) - K^\beta(\cdot;x_0,y) \|_H dy dx \\
    & \qquad \leq \frac{C}{|I|} \int_I \int_{(0,\infty )\setminus 2I} \|f_{\rm o}(y)-f_I\|_\B \left| \int_x^{x_0} \|\partial_u K^\beta(\cdot,u,y)\|_H du \right| dy dx \\
    & \qquad \leq \frac{C}{|I|} \int_I \int_{(0,\infty )\setminus 2I} \|f_{\rm o}(y)-f_I\|_\B \frac{|x-x_0|}{|x-y|^2} dy dx.
\end{align*}
By employing standard arguments (see, for instance, \cite[(9)]{BFHR})
we deduce that
\begin{equation}\label{H27b}
    \frac{1}{|I|} \int_I \| G_\B^{\beta}(f_2)(\cdot,x) -  G_\B^{\beta}(f_2)(\cdot,x_0)\|_{\gamma(H,\B)} dx
        \leq C \|f\|_{BMO_o((0,\infty ),\B)}.
\end{equation}
By putting together \eqref{H25}, \eqref{H26} and \eqref{H27b} we conclude that
$$\frac{1}{|I|} \int_I \| G_\B^{\beta}(f_{\rm o})(\cdot,x) -  G_\B^{\beta}(f_2)(\cdot,x_0)\|_{\gamma(H,\B)} dx
        \leq C \|f\|_{BMO_o((0,\infty ),\B)},$$
where $C>0$ does not depend on $I$. Hence, $G_\B^\beta(f_{\rm o}) \in BMO((0,\infty),\gamma(H,\B))$, and then
$G_{\B,{\rm loc}}^\beta(f) \in BMO((0,\infty),\gamma(H,\B))$. Finally by using \eqref{H26A}, \eqref{H26B} and
\eqref{H26C} we obtain that
$$\| G_\B^{\lambda,\beta}(f)(\cdot,x) -  G_{\B,{\rm loc}}^{\beta}(f)(\cdot,x)\|_{\gamma(H,\B)}
    \leq C \|f\|_{BMO_{\rm o}((0,\infty ),\B)}, \quad x \in (0,\infty),$$
and hence $G_\B^{\lambda,\beta}(f) \in BMO((0,\infty),\gamma(H,\B))$.
\end{proof}

%\newpage
\subsection{}\label{subsec:2.3}

In this section we are going to show that
$$\|f\|_{BMO_{\rm o}((0,\infty ),\B)}
    \leq C \|G_\B^{\lambda,\beta}(f)\|_{BMO_{\rm o}((0,\infty ),\gamma(H,\B))}, \quad f \in BMO_{\rm o}((0,\infty ),\B).$$
To establish this property we need to prove some auxiliary results.

Suppose firstly that $f \in BMO_{\rm o}(0,\infty)$. According to \cite[Theorem 6.1]{BCS} we have that the measure
$$d\mu_f(x,t)
    = \left| t^\beta \partial_t^\beta P_t^\lambda(f)(x) \right|^2 \frac{dx dt}{t}$$
is Carleson on $(0,\infty)^2$. Then, by \cite[Proposition 5.3]{BCS}, for every $a \in L^\infty_{+c}(0,\infty)$,
where $L^\infty_{+c}(0,\infty)$ denotes the space of bounded measurable functions with upper bounded support on $(0,\infty)$,
\begin{equation}\label{H28}
    \int_0^\infty \int_0^\infty t^\beta \partial_t^\beta P_t^\lambda(f)(x) t^\beta \partial_t^\beta P_t^\lambda(a)(x) \frac{dxdt}{t}
        = \frac{e^{2\pi i \beta}\G(2\beta)}{2^{2\beta}} \int_0^\infty f(x)a(x)dx.
\end{equation}
The following is a vector-valued version of \eqref{H28}.

\begin{Prop}\label{polarizacion}
    Let $\lambda,\beta>0$. If $f \in BMO_{\rm o}((0,\infty ),\B)$ and $a \in L^\infty_{+c}(0,\infty) \otimes \B^*$, then
    $$\int_0^\infty \int_0^\infty \langle G_{\B^*}^{\lambda,\beta}(a)(t,x), G_{\B}^{\lambda,\beta}(f)(t,x) \rangle_{\B^*,\B} \frac{dxdt}{t}
        = \frac{e^{2\pi i \beta}\G(2\beta)}{2^{2\beta}} \int_0^\infty \langle a(x),f(x) \rangle_{\B^*,\B}dx.$$
\end{Prop}

\begin{proof}
    Assume that $f \in BMO_{\rm o}((0,\infty ),\B)$ and $a \in L^\infty_{+c}(0,\infty) \otimes \B^*$. If $a=\sum_{j=1}^n a_j b_j$,
    being $a_j \in L^\infty_{+c}(0,\infty)$ and $b_j \in \B^*$, $j=1, \dots, n$, by using \eqref{H28} we can write
    \begin{align*}
        & \int_0^\infty \int_0^\infty \langle G_{\B^*}^{\lambda,\beta}(a)(t,x), G_{\B}^{\lambda,\beta}(f)(t,x) \rangle_{\B^*,\B} \frac{dxdt}{t} \\
        & \qquad = \sum_{j=1}^n \int_0^\infty \int_0^\infty G_{\C}^{\lambda,\beta}(a_j)(t,x) \langle b_j , G_{\B}^{\lambda,\beta}(f)(t,x) \rangle_{\B^*,\B} \frac{dxdt}{t} \\
        & \qquad = \sum_{j=1}^n \int_0^\infty \int_0^\infty G_{\C}^{\lambda,\beta}(a_j)(t,x)  G_{\C}^{\lambda,\beta}(\langle b_j ,f\rangle_{\B^*,\B})(t,x)  \frac{dxdt}{t} \\
        & \qquad = \frac{e^{2\pi i \beta}\G(2\beta)}{2^{2\beta}} \sum_{j=1}^n \int_0^\infty a_j(x) \langle b_j ,f(x)\rangle_{\B^*,\B} dx
                 = \frac{e^{2\pi i \beta}\G(2\beta)}{2^{2\beta}} \int_0^\infty  \langle a(x),f(x)\rangle_{\B^*,\B} dx,
    \end{align*}
    because, for every $j=1, \dots, n$, $\langle b_j, f \rangle_{\B^*,\B} \in BMO_{\rm o}(0,\infty )$.
\end{proof}

Let $f \in BMO_{\rm o}((0,\infty ),\B)$. We have that
\begin{equation}\label{H28A}
    \|f\|_{BMO_{\rm o}((0,\infty ),\B)}
        = \sup_{\substack{ g \in \mathcal{A} \otimes \B^* \\ \|g\|_{H_{\rm o}^1((0,\infty ),\B^*) \leq 1}}} \left| \int_0^\infty \langle g(x), f(x) \rangle_{\B^*,\B}\right|,
\end{equation}
where $\mathcal{A}=\spann \{ a : a \text{ is an } \infty-\text{atom}\}$. Note that, since $\B$ is
a UMD space, $\B^*$ is also UMD, $\B$ is reflexive and $BMO_{\rm o}((0,\infty ),\B)$ is the dual space of $H_{\rm o}^1((0,\infty ),\B^*)$.
Moreover, since $\mathcal{A}$ is a dense subspace of $H_{\rm o}^1((0,\infty ))$, $\mathcal{A} \otimes \B^*$ is dense in $H_{\rm o}^1((0,\infty ),\B^*)$.

According to Proposition~\ref{polarizacion} we get, for every $a \in \mathcal{A} \otimes \B^*$,
$$\int_0^\infty \int_0^\infty \langle G_{\B^*}^{\lambda,\beta}(a)(t,x), G_{\B}^{\lambda,\beta}(f)(t,x) \rangle_{\B^*,\B} \frac{dxdt}{t}
        = \frac{e^{2\pi i \beta}\G(2\beta)}{2^{2\beta}} \int_0^\infty \langle a(x),f(x) \rangle_{\B^*,\B}dx.$$
Also, for every $a \in \mathcal{A} \otimes \B^*$, we can write
\begin{equation}\label{H29}
    \int_0^\infty \langle G_{\B^*}^{\lambda,\beta}(a)(t,x), G_{\B}^{\lambda,\beta}(f)(t,x) \rangle_{\B^*,\B} \frac{dt}{t}
        = \left\langle G_{\B^*}^{\lambda,\beta}(a)(\cdot,x), G_{\B}^{\lambda,\beta}(f)(\cdot,x) \right\rangle_{\gamma(H,\B^*),\gamma(H,\B)},
        \quad x \in (0,\infty).
\end{equation}
Indeed, we take $a=\sum_{j=1}^n a_j b_j$, where $a_j \in \mathcal{A}$ and $b_j \in \B^*$, $j=1, \dots, n$. By taking into account the results in Section~\ref{subsec:2.1}, we have that
$$G_{\B^*}^{\lambda,\beta}(a)(t,x)
    = \sum_{j=1}^n b_j G_{\C}^{\lambda,\beta}(a_j)(t,x) \in H_{\rm o}^1((0,\infty ),\gamma(H,\B^*)).$$
The dual $(\gamma(H,\B))^*$ of $\gamma(H,\B)$ can be identified with $\gamma(H,\B^*)$ by using the trace functional.
If $(h_m)_{m\in \N}$ is an orthonormal basis in $H$ we can write
\begin{align*}
    & \left\langle G_{\B^*}^{\lambda,\beta}(a)(\cdot,x), G_{\B}^{\lambda,\beta}(f)(\cdot,x) \right\rangle_{\gamma(H,\B^*),\gamma(H,\B)}
        = \sum_{j=1}^n \left\langle b_j G_{\C}^{\lambda,\beta}(a_j)(\cdot,x), G_{\B}^{\lambda,\beta}(f)(\cdot,x) \right\rangle_{\gamma(H,\B^*),\gamma(H,\B)} \\
    & \qquad = \sum_{j=1}^n \sum_{m\in \N}\int_0^\infty h_m(t) \int_0^\infty
                    \left\langle b_j G_{\C}^{\lambda,\beta}(a_j)(u,x), G_{\B}^{\lambda,\beta}(f)(t,x) \right\rangle_{\B^*,\B} h_m(u) \frac{du}{u} \frac{dt}{t} \\
    & \qquad = \sum_{j=1}^n \sum_{m\in \N}\int_0^\infty h_m(t) \int_0^\infty G_{\C}^{\lambda,\beta}(a_j)(u,x)
                     G_{\C}^{\lambda,\beta}(\left\langle b_j ,f\right\rangle_{\B^*,\B})(t,x)  h_m(u) \frac{du}{u} \frac{dt}{t} \\
    & \qquad = \sum_{j=1}^n \langle  G_{\C}^{\lambda,\beta}(a_j)(\cdot ,x),
                     G_{\C}^{\lambda,\beta}(\langle b_j ,f\rangle_{\B^*,\B})(\cdot,x) \rangle _{\gamma(H,\C),\gamma(H,\C)}\\
    & \qquad = \sum_{j=1}^n \int_0^\infty G_{\C}^{\lambda,\beta}(a_j)(u,x)
                     G_{\C}^{\lambda,\beta}(\left\langle b_j ,f\right\rangle_{\B^*,\B})(u,x)   \frac{du}{u} \\
     & \qquad = \int_0^\infty
                     \left\langle G_{\B^*}^{\lambda,\beta}(a)(u,x), G_{\B}^{\lambda,\beta}(f)(u,x) \right\rangle_{\B^*,\B}  \frac{du}{u}, \quad x \in (0,\infty).
\end{align*}
We have used that $\gamma(H,\C)=H$.

By  \eqref{H29} we obtain
\begin{align*}
    & \int_0^\infty \left| \left\langle G_{\B^*}^{\lambda,\beta}(a)(\cdot,x), G_{\B}^{\lambda,\beta}(f)(\cdot,x) \right\rangle_{\gamma(H,\B^*),\gamma(H,\B)} \right| dx \\
    & \qquad \leq \sum_{j=1}^n \int_0^\infty \int_0^\infty
            \left| G_{\C}^{\lambda,\beta}(a_j)(t,x) \right| \left|G_{\C}^{\lambda,\beta}(\langle b_j ,f \rangle_{\B^*,\B})(t,x) \right|  \frac{dt}{t} dx.
\end{align*}
Since, $\langle b_j ,f \rangle_{\B^*,\B} \in BMO_{\rm o}(0,\infty )$, for every $j=1, \dots, n$, by \cite[Theorem 6.1]{BCS},
$|G_{\C}^{\lambda,\beta}(\langle b_j ,f \rangle_{\B^*,\B})(t,x)|^2 dtdx/t$ is a Carleson measure on $(0,\infty)^2$,
for every $j=1, \dots, n$. Then, according to \cite[Propositions 5.1 and 5.2]{BCS},
$$\int_0^\infty \left| \left\langle G_{\B^*}^{\lambda,\beta}(a)(\cdot,x), G_{\B}^{\lambda,\beta}(f)(\cdot,x) \right\rangle_{\gamma(H,\B^*),\gamma(H,\B)} \right| dx
    < \infty.$$
From \cite[Proposition 2.5]{BCCFR2} (adapted to this Bessel context), Proposition~\ref{polarizacion}, \eqref{H29},
and by taking into account, as it has already been proved, that $G_{\B^*}^{\lambda,\beta}(a) \in H_{\rm o}^1((0,\infty ),\gamma(H,\B^*))$ (Section \ref{subsec:2.1})
and $G_{\B}^{\lambda,\beta}(f) \in BMO_{\rm o}((0,\infty ),\gamma(H,\B))$ (Section \ref{subsec:2.2}), we conclude that
\begin{align*}
    \left| \int_0^\infty \langle a(x) , f(x) \rangle_{\B^*,\B}dx \right|
        \leq & C \|G_{\B^*}^{\lambda,\beta}(a)\|_{H_{\rm o}^1((0,\infty ),\gamma(H,\B^*))} \|G_{\B}^{\lambda,\beta}(f)\|_{BMO_{\rm o}((0,\infty ),\gamma(H,\B))} \\
        \leq & C \|a\|_{H_{\rm o}^1((0,\infty ),\B^*)} \|G_{\B}^{\lambda,\beta}(f)\|_{BMO_{\rm o}((0,\infty ),\gamma(H,\B))}.
\end{align*}
From \eqref{H28A} it follows that
$$ \|f\|_{BMO_{\rm o}((0,\infty ),\B)}
    \leq C \|G_{\B}^{\lambda,\beta}(f)\|_{BMO_{\rm o}((0,\infty ),\gamma(H,\B))}.$$

%\newpage
\subsection{}\label{subsec:2.4}

Our objective is to prove that, for every $a \in H_{\rm o}^1((0,\infty ),\B)$,
\begin{equation}\label{H30}
    \|a\|_{H_{\rm o}^1((0,\infty ),\B)}
        \leq C \|G_{\B}^{\lambda,\beta}(a)\|_{H_{\rm o}^1((0,\infty ),\gamma(H,\B)}.
\end{equation}
We have that
$$\|a\|_{H_{\rm o}^1((0,\infty ),\B)}
    = \sup_{\substack{ f \in BMO_{\rm o}((0,\infty ),\B^*) \\ \|f\|_{BMO_{\rm o}((0,\infty ),\B^*)} \leq 1}} \left| \int_0^\infty \langle f(x), a(x) \rangle_{\B^*,\B}dx\right|,
    \quad a\in H_{\rm o}^1((0,\infty ),\B).$$
Suppose that $a \in \mathcal{A} \otimes \B$, where $\mathcal{A}$ is defined as in Section~\ref{subsec:2.3}. By proceeding
as in Section~\ref{subsec:2.3} we get
\begin{align*}
    \left| \int_0^\infty \langle f(x) , a(x) \rangle_{\B^*,\B} \right|
        \leq & C \|G_{\B^*}^{\lambda,\beta}(f)\|_{BMO_{\rm o}((0,\infty ),\gamma(H,\B^*))} \|G_{\B}^{\lambda,\beta}(a)\|_{H_{\rm o}^1((0,\infty ),\gamma(H,\B))} \\
        \leq & C \|f\|_{BMO_{\rm o}((0,\infty ),\B^*)} \|G_{\B}^{\lambda,\beta}(a)\|_{H_{\rm o}^1((0,\infty ),\gamma(H,\B))}.
\end{align*}
Hence, \eqref{H30} holds.

In order to see that \eqref{H30} holds for every $a \in H_{\rm o}^1((0,\infty ),\B)$ it is enough to take into account that $\mathcal{A} \otimes \B$
is a dense subset of $H_{\rm o}^1((0,\infty ),\B)$ and that the operator $G_\B^{\lambda,\beta}$ is bounded from $H_{\rm o}^1((0,\infty ),\B)$
into $H_{\rm o}^1((0,\infty ),\gamma(H,\B))$ (Section \ref{subsec:2.1}).

%\newpage
%%%%%%%%%%%%%%%%%%%%%%%%%%%%%%%%%%%%%%%%%%%%%%%%%%%%%%%%%%%%%%%%%%%%%%%%%%%%%%%%%%%%%%%%%%%%%%%%%%%%%%%%%%%%%%%%%%%%
\section{Proof of Theorem~\ref{Th:2}}\label{sec:Proof2}
%%%%%%%%%%%%%%%%%%%%%%%%%%%%%%%%%%%%%%%%%%%%%%%%%%%%%%%%%%%%%%%%%%%%%%%%%%%%%%%%%%%%%%%%%%%%%%%%%%%%%%%%%%%%%%%%%%%%

In this section we prove that
$$
\mathcal{G}_\mathbb{B}^\lambda (f)(t,x)=tD_{\lambda ,x}^*P_t^{\lambda +1}(f)(x),\quad t,x\in (0,\infty ),
$$
is a bounded operator from $H^1_{\rm o}((0,\infty ),\mathbb{B})$ into $H^1_{\rm o}((0,\infty ),\gamma (H,\mathbb{B}))$ and from $BMO_{\rm o}((0,\infty ),\mathbb{B})$ into $BMO_{\rm o}((0,\infty ),\gamma (H,\mathbb{B}))$. Here $D_{\lambda,x}^*=-x^{-\lambda }\frac{d}{dx}x^\lambda $.

\subsection{}
 We establish now the behavior of $\mathcal{G}^\lambda _{\mathbb{B}}$ between Hardy spaces. In \cite[Theorem 1.3]{BCR3} we proved that the operator $\mathcal{G}^\lambda _{\mathbb{B}}$ is bounded from $L^p((0,\infty ),\mathbb{B})$ into $L^p((0,\infty ),\gamma (H,\mathbb{B}))$, for every $1<p<\infty$.

We show that $\mathcal{G}^\lambda _{\mathbb{B}}$ is a $(\mathbb{B},\gamma (H,\mathbb{B}))$-Calder\'on-Zygmund operator.

\begin{Lem}\label{Lemp25} Let $B$ be a UMD Banach space and $\lambda \geq 1$. The operator $\mathcal{G}^\lambda _{\mathbb{B}}$ is a $(\mathbb{B},\gamma (H,\mathbb{B}))$-Calder\'on-Zygmund operator.
\end{Lem}
\begin{proof}
We consider the function
\begin{equation}\label{Mlambda}
M^\lambda (t;x,y)=tD_{\lambda ,x}^*P_t^{\lambda +1}(x,y),\quad t,x,y\in (0,\infty ).
\end{equation}
$M^\lambda $ defines, as it will be specified, a standard $(\mathbb{B},\gamma (H,\mathbb{B}))$-Calder\'on-Zygmund kernel. Indeed, we have that
\begin{eqnarray*}
M^\lambda (t;x,y)&=&-tx^{-\lambda }\partial _x\left[\frac{2(\lambda +1)t}{\pi }x^{2\lambda +1}y^{\lambda +1}\int_0^\pi \frac{(\sin \theta )^{2\lambda +1}}{((x-y)^2+t^2+2xy(1-\cos \theta ))^{\lambda +2}}d\theta \right]\\
&=&-\frac{2(\lambda +1)}{\pi }t^2y^{\lambda +1}\left[(2\lambda +1)x^\lambda \int_0^\pi  \frac{(\sin \theta )^{2\lambda +1}}{((x-y)^2+t^2+2xy(1-\cos \theta ))^{\lambda +2}}d\theta \right.\\
&&-\left. 2(\lambda +2)x^{\lambda +1}\int_0^\pi  \frac{(\sin \theta )^{2\lambda +1}((x-y)+y(1-\cos \theta ))}{((x-y)^2+t^2+2xy(1-\cos \theta ))^{\lambda +3}}d\theta \right],\quad t,x,y\in (0,\infty ).
\end{eqnarray*}
Then,
\begin{eqnarray}\label{H30A}
|M^\lambda (t;x,y)|&\leq&Ct^2\left[(xy)^{\lambda+1}\int_0^\pi  \frac{(\sin \theta )^{2\lambda +1}}{(|x-y|+t+\sqrt{2xy(1-\cos \theta )})^{2\lambda +5}}d\theta \right.\nonumber\\
&&+\left. y^{\lambda +1}x^\lambda \int_0^\pi  \frac{(\sin \theta )^{2\lambda +1}}{(|x-y|+t+\sqrt{2xy(1-\cos \theta )})^{2\lambda +4}}d\theta\right]\nonumber\\
&=&C[\mathcal{I}_1(t;x,y)+\mathcal{I}_2(t;x,y)],\quad t,x,y\in (0,\infty ).
\end{eqnarray}
We observe that $\mathcal{I}_j(t;x,y)=t^2I_j(t/2,t/2;x,y)$, $t,x,y\in (0,\infty )$, $j=1,2$, where $I_j$, $j=1,2$, are the functions appearing in \eqref{I1I2} for $m=1$ and $\lambda +1$  instead of $\lambda$.

Then, by \eqref{10.1} and \eqref{I2}, for $j=1,2$, we have that
\begin{equation}\label{Ical}
\mathcal{I}_j(t;x,y)\leq C\frac{t^2}{(t+|x-y|)^3},\quad t,x,y\in (0,\infty ),
\end{equation}
and
$$
\|\mathcal{I}_j(t;x,y)\|_H\leq C\left(\int_0^\infty \frac{t^3}{(t+|x-y|)^6}dt\right)^{1/2}=\frac{C}{|x-y|},\quad x,y\in (0,\infty ), x\not=y.
$$
Hence,
\begin{equation}\label{H31}
\|M^\lambda (\cdot;x,y)\|_H\leq \frac{C}{|x-y|},\quad x,y\in (0,\infty ),\;x\not=y.
\end{equation}
We consider, for every $x,y\in (0,\infty )$, $x\not=y$, the operator
$$
\hspace{-3.5cm}
\begin{array}{c}
M^\lambda (x,y):\mathbb{B}\longrightarrow \gamma (H,\mathbb{B})\\
\hspace{4cm}b\longrightarrow M^\lambda (x,y)(b):H\longrightarrow \mathbb{B}\\
\hspace{10.5cm}\displaystyle h\longrightarrow \int_0^\infty M^\lambda (t;x,y)h(t)\frac{dt}{t}b.
\end{array}
$$

We have that, for every $b\in \B$,
$$
\|M^\lambda (x,y)(b)\|_{\gamma (H,\mathbb{B})}\leq \|M^\lambda (\cdot ;x,y)\|_H\|b\|_\mathbb{B}\leq C\frac{\|b\|_\mathbb{B}}{|x-y|}, \quad x,y\in (0,\infty ),\;x\not=y.
$$
Then,
\begin{equation}\label{H32}
\|M^\lambda (x,y)\|_{\mathbb{B}\rightarrow \gamma (H,\mathbb{B})}\leq \frac{C}{|x-y|},\quad x,y\in (0,\infty ),\;x\not=y.
\end{equation}
We can write, for every $t,x,y\in (0,\infty )$,
\begin{eqnarray*}
\partial _x M^\lambda (t;x,y)&=&-\frac{2(\lambda +1)}{\pi }t^2y^{\lambda +1}\left[(2\lambda +1)\lambda x^{\lambda -1}\int_0^\pi  \frac{(\sin \theta )^{2\lambda +1}}{((x-y)^2+t^2+2xy(1-\cos \theta ))^{\lambda +2}}d\theta \right.\\
&&-2(2\lambda +1)(\lambda +2) x^\lambda \int_0^\pi  \frac{(\sin \theta )^{2\lambda +1}((x-y)+y(1-\cos \theta ))}{((x-y)^2+t^2+2xy(1-\cos \theta ))^{\lambda +3}}d\theta \\
&&-2(\lambda +1)(\lambda +2) x^\lambda \int_0^\pi  \frac{(\sin \theta )^{2\lambda +1}((x-y)+y(1-\cos \theta ))}{((x-y)^2+t^2+2xy(1-\cos \theta ))^{\lambda +3}}d\theta\\
&&+4(\lambda +2)(\lambda +3) x^{\lambda +1}\int_0^\pi  \frac{(\sin \theta )^{2\lambda +1}((x-y)+y(1-\cos \theta ))^2}{((x-y)^2+t^2+2xy(1-\cos \theta ))^{\lambda +4}}d\theta\\
&&-\left.2(\lambda +2) x^{\lambda +1}\int_0^\pi  \frac{(\sin \theta )^{2\lambda +1}}{((x-y)^2+t^2+2xy(1-\cos \theta ))^{\lambda +3}}d\theta \right].
\end{eqnarray*}

It follows that
\begin{eqnarray}\label{H32A}
\left|\partial _xM^\lambda (t;x,y)\right|&\leq &Ct^2\left[y^{\lambda +1}x^{\lambda -1}\int_0^\pi  \frac{(\sin \theta )^{2\lambda +1}}{(|x-y|+t+\sqrt{2xy(1-\cos \theta )})^{2\lambda +4}}d\theta \right.\nonumber\\
&&+y^{\lambda +1}x^\lambda \int_0^\pi  \frac{(\sin \theta )^{2\lambda +1}}{(|x-y|+t+\sqrt{2xy(1-\cos \theta )})^{2\lambda +5}}d\theta\nonumber\\
&&+\left.(yx)^{\lambda +1}\int_0^\pi  \frac{(\sin \theta )^{2\lambda +1}}{(|x-y|+t+\sqrt{2xy(1-\cos \theta )})^{2\lambda +6}}d\theta\right]\nonumber\\
&=&C[J_1(t;x,y)+J_2(t;x,y)+J_3(t;x,y)],\quad t,x,y\in (0,\infty ).
\end{eqnarray}
We note that $J_2(t;x,y)=t^2I_2(t/2,t/2;x,y)$ and $J_3(t;x,y)=t^2I_1(t/2,t/2;x,y)$, $t,x,y\in (0,\infty )$, being $I_j$, $j=1,2$, the functions in \eqref{I1I2} for $m=2$ and $\lambda +1$ instead of $\lambda$. Hence, by (\ref{10.1}) and (\ref{I2}), for $\ell =2,3$,
\begin{equation}\label{26.1}
J_\ell (t;x,y)\leq C\frac{t^2}{(t+|x-y|)^4},\quad t,x,y\in (0,\infty ).
\end{equation}
On the other hand, since $\lambda \geq 1$, by proceeding as in the estimation \eqref{I2}, we obtain
\begin{equation}\label{J1}
J_1(t;x,y)\leq C\frac{t^2}{(t+|x-y|)^4},\quad t,x,y\in (0,\infty ).
\end{equation}
Thus, for $\ell =1,2,3$,
\begin{equation}\label{Jell}
\|J_\ell (\cdot ;,x,y)\|_H\leq C\left(\int_0^\infty \frac{t^3}{(t+|x-y|)^8}dt\right)^{1/2}\leq \frac{C}{|x-y|^2},\quad x,y\in (0,\infty ),x\not=y.
\end{equation}
Hence,
\begin{equation}\label{H33}
\left\|\partial _xM^\lambda (\cdot;x,y)\right\|_H\leq \frac{C}{|x-y|^2},\quad x,y\in (0,\infty ),\;x\not=y.
\end{equation}
We have also that
\begin{equation}\label{H34}
\left\|\partial _yM^\lambda (\cdot;x,y)\right\|_H\leq \frac{C}{|x-y|^2},\quad x,y\in (0,\infty ),\;x\not=y.
\end{equation}
%From (\ref{H34}) we deduce that
%\begin{eqnarray}\label{H35}
%\|M^\lambda (x,y)-M^\lambda (x,z)\|_{\mathbb{B}\rightarrow \gamma (H,\mathbb{B})}&\leq &\|M^\lambda (\cdot; x,y)-M^\lambda (\cdot ;x,z)\|_H\nonumber\\
%&=&\left\| \int _z^y\partial _vM^\lambda (\cdot ; x,v)dv\right\|_H\leq \int _z^y\left\|\partial _vM^\lambda (\cdot ; x,v)\right\|_Hdv\nonumber\\
%&\leq&C\int _z^y\frac{1}{(x-v)^2}dv\leq C\frac{|y-z|}{|x-y|^2},\quad |x-y|\geq 2|y-z|,\quad x,y,z\in (0,\infty ).
%\end{eqnarray}
%By (\ref{H33}) it follows that
%\begin{equation}\label{H36}
%\|M^\lambda (y,x)-M^ \lambda (z,x)\|_{\mathbb{B}\rightarrow \gamma (H,\mathbb{B})}\leq C\frac{|y-z|}{|x-y|^2},\quad |x-y|\geq 2|y-z|,\quad x,y,z\in (0,\infty %).
%\end{equation}
Estimations (\ref{H32}), (\ref{H33}) and (\ref{H34}) show that $M^\lambda $ is a standard $(\mathbb{B},\gamma (H,\mathbb{B}))$-Calder\'on-Zygmund kernel.

Let $f\in S_\lambda (0,\infty)\otimes \mathbb{B}$. According to (\ref{H32}) we have that
$$
\int_0^\infty \|M^\lambda (x,y)\|_{\mathbb{B}\rightarrow \gamma (H,\mathbb{B})}\|f(y)\|_\mathbb{B}dy<\infty ,\quad x\not\in \mbox{supp } f.
$$
We define
\begin{equation}\label{27.1}
Q^\lambda (f)(x)=\int_0^\infty M^\lambda (x,y)f(y)dy,\quad x\not\in \mbox{supp }f,
\end{equation}
where the integral is understood in the $\gamma (H,\mathbb{B})$-Bochner sense.

We can differentiate under the integral sign to get
$$
\mathcal{G}_\mathbb{B}^\lambda (f)(t,x)=\int_0^\infty M^\lambda (t;x,y)f(y)dy,\quad t,x\in (0,\infty ),
$$
and by using Minkowski's inequality and (\ref{H32}) we obtain
$$
\|\mathcal{G}_\mathbb{B}^\lambda (f)(\cdot ,x)\|_{L^2((0,\infty ),dt/t,\mathbb{B})}\leq C\int_0^\infty \frac{\|f(y)\|_\mathbb{B}}{|x-y|}dy,\quad x\not\in \mbox{supp }f.
$$
We consider, for every $x\not \in {\rm supp }f$, the operator
$$
\hspace{-4cm}\begin{array}{c}
\mathbb{G}^\lambda _\mathbb{B} (f)(x):H\longrightarrow \mathbb{B}\\
\hspace{7cm}h\longrightarrow \displaystyle [\mathbb{G}^\lambda _\mathbb{B} (f)(x)](h)=\int_0^\infty \mathcal{G}^\lambda _\mathbb{B}(f)(t,x)h(t)\frac{dt}{t}.
\end{array}
$$
We can write
\begin{eqnarray*}
[Q^\lambda (f)(x)](h)&=&\int_0^\infty [M^\lambda (x,y)f(y)](h)dy\\
&=&\int_0^\infty \int_0^\infty M^\lambda (t;x,y)h(t)\frac{dt}{t}f(y)dy\\
&=&\int_0^\infty \mathcal{G}_\mathbb{B}^\lambda (f)(t,x)h(t)\frac{dt}{t},\quad h\in H\mbox{ and }x\not \in \mbox{supp }f.
\end{eqnarray*}
The interchange of the order of integration is justified because
$$
\int_0^\infty \int_0^\infty |M^\lambda (t;x,y)||h(t)|\frac{dt}{t}\|f(y)\|_\mathbb{B}dy\leq C\|h\|_H\int_0^\infty \frac{\|f(y)\|_\mathbb{B}}{|x-y|}dy<\infty ,\quad h\in H\mbox{ and }x\not \in \mbox{supp }f.
$$
Hence, we deduce that
$$
Q^\lambda (f)(x)=\mathcal{G}_\mathbb{B}^\lambda (f)(\cdot ,x),\quad x\not\in \mbox{supp }f,
$$
as elements of $\gamma (H,\mathbb{B})$.
\end{proof}

According to \cite[Theorem 1.3]{BCR3}, by using vector-valued Calder\'on-Zygmund theory (\cite{RRT}) we infer that the operator defined by \eqref{27.1} can  be extended from $S^\lambda (0,\infty )\otimes \mathbb{B}$ to $L^1((0,\infty ),\mathbb{B})$ as a bounded operator from $L^1((0,\infty ),\mathbb{B})$ into $L^{1,\infty }((0,\infty ),\gamma (H,\mathbb{B}))$ and from $H^1((0,\infty ),\mathbb{B})$
into $L^1((0,\infty ),\gamma (H,\mathbb{B}))$. By proceeding as in the proof of Lemma \ref{Lemap8} we can conclude that the operator $\mathcal{G}_\mathbb{B}^\lambda $ is bounded from $L^1((0,\infty ),\mathbb{B})$ into
$L^{1,\infty }((0,\infty ),\gamma (H,\mathbb{B}))$ and from $H^1((0,\infty ),\mathbb{B})$ into $L^1((0,\infty ),\gamma (H,\mathbb{B}))$.

By (\ref{H30A}) we have that
$$
|M^\lambda (t; x,y)|\leq C[\mathcal{I}_1(t;x,y)+\mathcal{I}_2(t;x,y)],\quad t,x,y\in (0,\infty ),
$$
where
$$
\mathcal{I}_1(t;x,y)=t^2(xy)^{\lambda +1}\int_0^\pi  \frac{(\sin \theta )^{2\lambda +1}}{(|x-y|+t+\sqrt{2xy(1-\cos \theta )})^{2\lambda +5}}d\theta ,\quad t,x,y\in (0,\infty ),
$$
$$
\mathcal{I}_2(t;x,y)=t^2y^{\lambda +1}x^\lambda \int_0^\pi  \frac{(\sin \theta )^{2\lambda +1}}{(|x-y|+t+\sqrt{2xy(1-\cos \theta )})^{2\lambda +4}}d\theta ,\quad t,x,y\in (0,\infty ),
$$
and $M^\lambda (t;x,y)$, $t,x,y\in (0,\infty )$, is the kernel introduced in \eqref{Mlambda}.

We obtain (see the proof of (\ref{H2}) with $\lambda +1$ instead of $\lambda$ and $k=2$)
\begin{equation}\label{I1b}
|\mathcal{I}_1(t;x,y)|\leq C\frac{t^2(xy)^{\lambda +1}}{(t+|x-y|)^{2\lambda +5}},\quad t,x,y\in (0,\infty ),
\end{equation}
and
\begin{equation}\label{I2b}
|\mathcal{I}_2(t;x,y)|\leq C\frac{t^2y^{\lambda +1}x^\lambda}{(t+|x-y|)^{2\lambda +4}},\quad t,x,y\in (0,\infty ),
\end{equation}
Then,
\begin{equation}\label{H37}
\|\mathcal{I}_1(\cdot ; x,y)\|_H\leq C\frac{(xy)^{\lambda +1}}{|x-y|^{2\lambda +3}},\quad x,y\in (0,\infty ),\;x\not=y,
\end{equation}
and
\begin{equation}\label{H38}
\|\mathcal{I}_2(\cdot ; x,y)\|_H\leq C\frac{y^{\lambda +1}x^\lambda}{|x-y|^{2\lambda +2}},\quad x,y\in (0,\infty ),\;x\not=y.
\end{equation}
Let $\delta >0$ and $b\in \mathbb{B}$, $\|b\|_\B=1$. We define $a=\frac{b}{\delta}\chi _{(0,\delta)}$. Since $\mathcal{G}_\mathbb{B}^\lambda$ is bounded from $L^2((0,\infty ),\mathbb{B})$ into $L^2((0,\infty ),\gamma (H,\mathbb{B}))$ (\cite[Theorem 1.3]{BCR3}), we deduce that
$$
\int_0^{2\delta }\|\mathcal{G}_\mathbb{B}^\lambda (a)(\cdot ,x)\|_{\gamma (H,\mathbb{B})}dx\leq C\delta ^{1/2}\|\mathcal{G}_\mathbb{B}^\lambda (a)\|_{L^2((0,\infty ),\gamma (H,\mathbb{B}))}
\leq C\delta ^{1/2}\|a\|_{L^2((0,\infty ),\mathbb{B})}\leq C,
$$
where $C>0$ does not depend on  $\delta$ or $b$. Also, by (\ref{H37}) and (\ref{H38}), we get
\begin{eqnarray*}
\int_{2\delta }^\infty \|\mathcal{G}_\mathbb{B}^\lambda (a)(\cdot ,x)\|_{\gamma (H,\mathbb{B})}dx&\leq&\frac{C}{\delta}\int_{2\delta }^\infty \int_0^\delta
\left(\frac{y^{\lambda +1}x^\lambda}{|x-y|^{2\lambda +2}}+\frac{(xy)^{\lambda +1}}{|x-y|^{2\lambda +3}}\right)dydx\\
&\leq&\frac{C}{\delta}\int_{2\delta }^\infty \frac{dx}{x^{\lambda +2}}\int_0^\delta y^{\lambda +1}dy\leq C,
\end{eqnarray*}
where $C>0$ does not depend on $\delta $ or $b$.
We conclude that $\|\mathcal{G}_\mathbb{B}^\lambda (a)\|_{L^1((0,\infty ),\gamma (H,\mathbb{B}))}\leq C$, where $C>0$ does not depend on $\delta$ or $b$.

Also, since $\mathcal{G}_\mathbb{B}^\lambda$ is bounded from $H^1((0,\infty ),\mathbb{B})$ into $L^1((0,\infty ),\gamma (H,\mathbb{B}))$, there exists $C>0$ such that, for every 2-atom $a$ satisfying $(Aii)$,
$$
\|\mathcal{G}_\mathbb{B}^\lambda (a)\|_{L^1((0,\infty ),\gamma (H,\mathbb{B}))}\leq C.
$$
Then, by taking into account that $\mathcal{G}_\mathbb{B}^\lambda $ is bounded from $L^1((0,\infty ),\mathbb{B})$ into $L^{1,\infty }((0,\infty ),\gamma (H,\mathbb{B}))$, we conclude that
$\mathcal{G}_\mathbb{B}^\lambda $ is bounded from $H^1_{\rm o}((0, \infty ),\mathbb{B})$ into $L^1((0,\infty ),\gamma (H,\mathbb{B}))$.

Since the Hardy space $H^1_{\rm o}((0,\infty ),\mathbb{B})$ does not depend on $\lambda$, in order to see that $\mathcal{G}_\mathbb{B}^\lambda $ is bounded from $H^1_{\rm o}((0,\infty ),\mathbb{B})$ into $H^1_{\rm o}((0,\infty ),\gamma (H,\mathbb{B}))$ it is enough to show that $P_*^{\lambda +1}\circ\mathcal{G}_\mathbb{B}^\lambda $ is bounded from $H^1_{\rm o}((0,\infty ),\mathbb{B})$ into $L^1(0,\infty )$, where
$$
P_*^{\lambda +1}(f)=\sup_{s>0}\|P_s^{\lambda +1}(f)\|_{\gamma (H,\mathbb{B})},\quad f\in L^p((0,\infty ),\gamma (H,\mathbb{B})),\;\,1\leq p<\infty .
$$
Since the maximal operator $P_*^{\lambda +1}$ is bounded from $L^1((0,\infty ),\gamma (H,\mathbb{B}))$ into $L^{1,\infty }(0,\infty )$, the operator $P_*^{\lambda +1}\circ \mathcal{G}_\mathbb{B}^\lambda$ is bounded from $H^1_{\rm o}((0,\infty ),\mathbb{B})$ into $L^{1,\infty }(0,\infty )$. By using vector-valued Calder\'on-Zygmund theory as in the proof of Lemma \ref{Lemap14} we can see that $P_*^{\lambda +1}\circ \mathcal{G}_\mathbb{B}^\lambda$ is bounded from $H^1((0,\infty ),\mathbb{B})$ into $L^1(0,\infty )$. Moreover, as above, we can prove that there exists $C>0$ such that, for every $\delta >0$ and $b\in \mathbb{B}$, $\|b\|_\B=1$,
$$
\|P_*^{\lambda +1}(\mathcal{G}_\mathbb{B}^\lambda (a))\|_{L^1(0,\infty )}\leq C,
$$
being $a=\frac{b}{\delta }\chi _{(0,\delta )}$. Then, we conclude that $P_*^{\lambda +1}\circ \mathcal{G}_\mathbb{B}^\lambda$ is bounded from $H^1_{\rm o}((0,\infty ),\mathbb{B})$ into $L^1(0,\infty )$.

Thus the proof of Theorem \ref{Th:2} related to Hardy spaces is complete.
\subsection{} Our objective is to show that $\mathcal{G}_\mathbb{B}^\lambda $ is bounded from $BMO_{\rm o}((0,\infty ),\mathbb{B})$ into $BMO_{\rm o}((0,\infty ),\gamma (H,\mathbb{B}))$.
As in Section \ref{subsec:2.2} this is naturally divided into the following two lemmas.
\begin{Lem}\label{Lemap29}
Assume that $\B$ is a UMD Banach space and $\lambda >0$. We can find $C>0$ such that, for every $r>0$,
\begin{equation}\label{61}
\frac{1}{r}\int_0^r\|\mathcal{G}_\B^\lambda (f)(\cdot ,x)\|_{\gamma (H,\B)}dx\leq C\|f\|_{BMO_{\rm o}((0,\infty ),\B)},\quad f\in BMO_{\rm o}((0,\infty ), \B).
\end{equation}
\end{Lem}
\begin{proof}
Let $f\in BMO_{\rm o}((0,\infty ),\mathbb{B})$. According to (\ref{H30A}), (\ref{I1b}) and (\ref{I2b}) we get
$$
|D_{\lambda ,x}^*P_t^{\lambda +1}(x,y)|\leq Ct\left(\frac{(xy)^{\lambda +1}}{(|x-y|+t)^{2\lambda +5}}+\frac{y^{\lambda +1}x^\lambda}{(|x-y|+t)^{2\lambda +4}}\right),\quad t,x,y\in (0,\infty ).
$$
Since $\int_0^\infty (1+y^2)^{-1}\|f(y)\|_\mathbb{B}dy<\infty$, we can write
$$
\mathcal{G}_\mathbb{B}^\lambda (f)(t,x)=\int_0^\infty tD_{\lambda ,x}^*P_t^{\lambda +1}(x,y)f(y)dy,\quad t,x\in (0,\infty ).
$$
Let $r>0$. Since $\mathcal{G}_\mathbb{B}^\lambda $ is bounded from $L^2((0,\infty ),\mathbb{B})$ into
$L^2((0,\infty ),\gamma (H,\mathbb{B}))$ (\cite[Theorem 1.3]{BCR3}) we get
\begin{eqnarray*}
\frac{1}{r}\int_0^r\|\mathcal{G}_\mathbb{B}^\lambda (f\chi _{(0,2r)})(\cdot ,x)\|_{\gamma (H,\mathbb{B})}dx&\leq &\frac{1}{\sqrt{r}}\|\mathcal{G}_\mathbb{B}^\lambda (f\chi _{(0,2r)})\|_{L^2((0,\infty ),\gamma (H,\mathbb{B}))}\\
&\leq &\frac{C}{\sqrt{r}}\|f\chi _{(0,2r)}\|_{L^2((0,\infty ),\mathbb{B})}\leq C\|f\|_{BMO_{\rm o}((0,\infty ),\mathbb{B})}.
\end{eqnarray*}
Moreover, by using (\ref{H37}) and (\ref{H38}), as in (\ref{H19A}) we obtain
\begin{eqnarray*}
\frac{1}{r}\int_0^r\|\mathcal{G}_\mathbb{B}^\lambda (f\chi _{(2r,\infty )})(\cdot ,x)\|_{\gamma (H,\mathbb{B})}dx&\leq&\frac{1}{r}\int_0^r\int_{2r}^\infty \|M^\lambda (\cdot ;x,y)\|_H\|f(y)\|_\mathbb{B}dydx\\
&\hspace{-6cm}\leq&\hspace{-3cm}\frac{1}{r}\int_0^r\int_{2r}^\infty \left(\frac{(xy)^{\lambda +1}}{|x-y|^{2\lambda +3}}+\frac{y^{\lambda +1}x^\lambda}{|x-y|^{2\lambda +2}}\right)\|f(y)\|_\mathbb{B}dydx\\
&\hspace{-6cm}\leq&\hspace{-3cm}\frac{1}{r}\int_0^r\int_{2x}^\infty \frac{x^\lambda }{y^{\lambda +1}}\|f(y)\|_\mathbb{B}dydx\leq C\|f\|_{BMO_{\rm o}((0,\infty ),\mathbb{B})}.
\end{eqnarray*}
Hence, (\ref{61}) holds with $C$ independent of $f$ and $r$, and then, $\mathcal{G}_\mathbb{B}^\lambda (f)(\cdot ,x)\in \gamma (H,\mathbb{B})$, a.e. $x\in (0,\infty )$.
\end{proof}

\begin{Lem}\label{Lep30}
Let $\B$ be a UMD Banach space and $\lambda >0$. $\mathcal{G}_\B^\lambda$ is a bounded operator from $BMO_{\rm o}((0,\infty ),\B)$ into $BMO((0,\infty ),\gamma (H,\B))$.
\end{Lem}
\begin{proof}
Let $f\in BMO_{\rm o}((0,\infty ),\B)$. We consider the even extension $f_{\rm e}$ of the function $f$ to $\R$. Note that $f_{\rm e}\in BMO(\mathbb{R},\mathbb{B})$.

We define
$$
\mathcal{G}_\mathbb{B}(f_{\rm e})(t,x)=-\int_\R t\partial _xP_t(x-y)f_{\rm e}(y)dy,\quad t\in (0,\infty )\mbox{ and }x\in \mathbb{R},
$$
where $P_t(z)=\frac{1}{\pi }\frac{t}{t^2+z^2}$, $t\in (0,\infty )$ and $z\in \mathbb{R}$.

We have that
$$
\mathcal{G}_\mathbb{B}(f_{\rm e})(t,x)=-\int_0^\infty t\partial _x[P_t(x-y)+P_t(x+y)]f(y)dy,\quad t\in (0,\infty )\mbox{ and } x\in \R .
$$
By using mean value theorem we obtain
\begin{eqnarray}\label{H40}
\left|\partial _x[P_t(x-y)+P_t(x+y)]\right|&=&\frac{2t}{\pi }\left|\frac{y-x}{(t^2+(x-y)^2)^2}-\frac{x+y}{(t^2+(x+y)^2)^2}\right|\nonumber\\
&\leq&C\frac{tx}{(t+|x-y|)^4},\quad t,x,y\in (0,\infty ).
\end{eqnarray}
We split $\mathcal{G}_\mathbb{B} (f_{\rm e})$ as follows:
\begin{eqnarray}\label{30.1}
\mathcal{G}_\mathbb{B}(f_{\rm e})(t,x)&=&-\int_0^{x/2}t\partial _x[P_t(x-y)+P_t(x+y)]f(y)dy\nonumber\\
&&-\int_{2x}^\infty t\partial _x[P_t(x-y)+P_t(x+y)]f(y)dy\nonumber\\
&&-\int_{x/2}^{2x}t\partial _xP_t(x+y)f(y)dy-\int_{x/2}^{2x}t\partial _xP_t(x-y)f(y)dy\nonumber\\
&=&\sum_{j=1}^4I_j(f)(t,x),\quad t,x\in (0,\infty ).
\end{eqnarray}
According to (\ref{H40}) we get
\begin{eqnarray*}
\|I_1(f)(\cdot ,x)\|_{\gamma (H,\mathbb{B})}&\leq&Cx\int_0^{x/2}\|f(y)\|_\mathbb{B}\left(\int_0^\infty \frac{t^3}{(t+x)^8}dt\right)^{1/2}dy\\
&\leq&\frac{C}{x}\int_0^{x/2}\|f(y)\|_\mathbb{B}dy\leq C\|f\|_{BMO_{\rm o}((0,\infty ),\mathbb{B})},\quad x\in (0,\infty ),
\end{eqnarray*}
and, as in (\ref{H24})
\begin{eqnarray*}
\|I_2(f)(\cdot ,x)\|_{\gamma (H,\mathbb{B})}&\leq&Cx\int_{2x}^\infty \|f(y)\|_\mathbb{B}\left(\int_0^\infty \frac{t^3}{(t+y)^8}dt\right)^{1/2}dy\\
&\leq&Cx\int_{2x}^\infty \frac{\|f(y)\|_\mathbb{B}}{y^2}dy\leq C\|f\|_{BMO_{\rm o}((0,\infty ),\mathbb{B})},\quad x\in (0,\infty ).
\end{eqnarray*}
We also have that
\begin{eqnarray*}
\|I_3(f)(\cdot ,x)\|_{\gamma (H,\mathbb{B})}&\leq&C\int_{x/2}^{2x}\left\|t\partial _xP_t(x+y)\right\|_H\|f(y)\|_\mathbb{B}dy\\
&\leq&C\int_{x/2}^{2x}\|f(y)\|_\mathbb{B}\left(\int_0^\infty \frac{t^3}{(t+x+y)^6}dt\right)^{1/2}dy\\
&\leq&\frac{C}{x}\int_{x/2}^{2x}\|f(y)\|_\mathbb{B}dy\leq C\|f\|_{BMO_{\rm o}((0,\infty ),\mathbb{B})},\quad x\in (0,\infty ).
\end{eqnarray*}

We are going to show that $\mathcal{G}_\mathbb{B} (f_{\rm e})\in BMO((0,\infty ),\gamma (H,\mathbb{B}))$. In order to do this we proceed as in Section 2.2. Let $0<r<s<\infty$. We define $I=(r,s)$ and $2I=(x_I-2d_I,x_I+2d_I)$, where $x_I=(r+s)/2$ and $d_I=(s-r)/2$, and we decompose $f_{\rm e}$ as follows:
$$
f_{\rm e}=(f_{\rm e}-f_I)\chi _{2I}+(f_{\rm e}-f_I)\chi _{(0,\infty )\setminus 2I}+f_I=f_1+f_2+f_3.
$$
Since $\mathcal{G}_\mathbb{B}$ is a bounded operator from $L^2(\mathbb{R},\mathbb{B})$ into $L^2(\mathbb{R},\gamma (H,\mathbb{B}))$ (\cite[Theorem 4.2]{KaWe}) we deduce that
\begin{equation}\label{G1}
\frac{1}{|I|}\int_I\|\mathcal{G}_\mathbb{B}(f_1)(\cdot ,x)\|_{\gamma (H,\mathbb{B})}dx\leq C\|f\|_{BMO_{\rm o}((0,\infty ),\mathbb{B})}
\end{equation}
and then,  $\mathcal{G}_\mathbb{B}(f_1)(\cdot ,x)\in \gamma (H,\mathbb{B})$, a.e. $x\in I$. Moreover, since $\int_\mathbb{R}P_t(x-y)dy=1$, $t\in (0,\infty )$ and $x\in \mathbb{R}$, $\mathcal{G}_\mathbb{B}(f_3)=0$. Finally, we study $\mathcal{G}_\mathbb{B}(f_2)$. We have that
$$
\mathcal{G}_\mathbb{B}^\lambda (f)(t,x)-\mathcal{G}_\mathbb{B} (f_{\rm e})(t,x)=\sum_{j=1}^3J_j^\lambda (f)(t,x)-\sum_{j=1}^3I_j(f)(t,x),\quad t,x\in (0,\infty ),
$$
where
$$
J_1^\lambda (f)(t,x)=\int_0^{x/2}M^\lambda (t;x,y)f(y)dy,\quad t,x\in (0,\infty ),
$$
$$
J_2^\lambda (f)(t,x)=\int_{2x}^\infty M^\lambda (t;x,y)f(y)dy,\quad t,x\in (0,\infty ),
$$
and
$$
J_3^\lambda (f)(t,x)=\int_{x/2}^{2x}[M^\lambda (t;x,y)+t\partial _xP_t(x-y)]f(y)dy,\quad t,x\in (0,\infty ).
$$
Here, $I_j$, $j=1,2,3$, are defined as in \eqref{30.1} and $M^\lambda$ is the kernel in (\ref{Mlambda}).

According to (\ref{H37}) and (\ref{H38}) we get
$$
\|M^\lambda (\cdot ;x,y)\|_H\leq C\left(\frac{(xy)^{\lambda +1}}{|x-y|^{2\lambda +3}}+\frac{y^{\lambda +1}x^\lambda }{|x-y|^{2\lambda +2}}\right),\quad x,y\in (0,\infty ),\;x\not =y.
$$
It follows that, for every $x\in (0,\infty )$
$$
\|J_1^\lambda (f)(\cdot,x)\|_{\gamma (H,\mathbb{B})}\leq \int_0^{x/2}\|M^\lambda (\cdot ;x,y)\|_H\|f(y)\|_\mathbb{B}dy
\leq \frac{C}{x}\int_0^{x/2}\|f(y)\|_\mathbb{B}dy\leq C\|f\|_{BMO_{\rm o}((0,\infty ),\mathbb{B})},
$$
and,  as in (\ref{H19A}),
$$
\|J_2^\lambda (f)(\cdot,x)\|_{\gamma (H,\mathbb{B})}\leq Cx^\lambda \int_{2x}^\infty \frac{\|f(y)\|_\mathbb{B}}{y^{\lambda +1}}dy\leq C\|f\|_{BMO_{\rm o}((0,\infty ),\mathbb{B})},\quad x\in (0,\infty ).
$$
We write
\begin{eqnarray*}
M^\lambda (t;x,y)&=&-\frac{2(\lambda +1)}\pi t^2y^{\lambda +1}\left[(2\lambda +1)x^\lambda \int_0^\pi \frac{(\sin \theta )^{2\lambda +1}}{((x-y)^2+t^2+2xy(1-\cos \theta ))^{\lambda +2}}d\theta \right.\\
&&-\left.2(\lambda +2) x^{\lambda +1}\int_0^\pi  \frac{(\sin \theta )^{2\lambda +1}((x-y)+y(1-\cos \theta ))}{((x-y)^2+t^2+2xy(1-\cos \theta ))^{\lambda +3}}d\theta \right]\\
&=&M_1^\lambda (t;x,y)+M_2^\lambda (t;x,y),\quad t,x,y\in (0,\infty ),
\end{eqnarray*}
where
\begin{eqnarray*}
M_1^\lambda (t;x,y)&=&-\frac{2(\lambda +1)}\pi t^2y^{\lambda +1}\left[(2\lambda +1)x^\lambda \int_0^{\pi /2}\frac{(\sin \theta )^{2\lambda +1}}{((x-y)^2+t^2+2xy(1-\cos \theta ))^{\lambda +2}}d\theta \right.\\
&&-\left.2(\lambda +2) x^{\lambda +1}\int_0^{\pi /2}\frac{(\sin \theta )^{2\lambda +1}((x-y)+y(1-\cos \theta ))}{((x-y)^2+t^2+2xy(1-\cos \theta ))^{\lambda +3}}d\theta \right],\quad t,x,y\in (0,\infty ).
\end{eqnarray*}

We have that
\begin{eqnarray}\label{M1}
\|M_2^\lambda (\cdot ;x,y)\|_H&\leq &C\left[x^{2\lambda +1}\left(\int_0^\infty \frac{t^3}{(|x-y|+t+\sqrt{xy})^{4(\lambda +2)}}dt\right)^{1/2}\right.\nonumber\\
&&+\left. x^{2\lambda +2}\left(\int_0^\infty \frac{t^3}{(|x-y|+t+\sqrt{xy})^{4(\lambda +5/2)}}dt\right)^{1/2}\right]\leq \frac{C}{x},\quad 0<\frac{x}{2}<y<2x.
\end{eqnarray}
We now write
\begin{eqnarray*}
M_1^\lambda (t;x,y)&=&-\frac{2(\lambda +1)}{\pi }t^2y^{\lambda +1}\left[(2\lambda +1)x^\lambda \int_0^{\pi /2} \frac{(\sin \theta )^{2\lambda +1}}{((x-y)^2+t^2+2xy(1-\cos \theta ))^{\lambda +2}}d\theta \right.\\
&&-2(\lambda +2)x^{\lambda +1}\int_0^{\pi /2} \frac{(\sin \theta )^{2\lambda +1}y(1-\cos \theta )}{((x-y)^2+t^2+2xy(1-\cos \theta ))^{\lambda +3}}d\theta\\
&&-\left. 2(\lambda +2)x^{\lambda +1}\int_0^{\pi /2}  \frac{(x-y)(\sin \theta )^{2\lambda +1}}{((x-y)^2+t^2+2xy(1-\cos \theta ))^{\lambda +3}}d\theta\right]\\
&=&\sum_{j=1}^3M_{1,j}^\lambda (t;x,y),\quad t,x,y\in (0,\infty ),
\end{eqnarray*}
We obtain that
$$
|M_{1,2}^\lambda (t;x,y)|\leq C|M_{1,1}^\lambda (t;x,y)|\leq Ct^2x^\lambda y^{\lambda +1}\int_0^{\pi /2} \frac{\theta ^{2\lambda +1}}{(|x-y|+t+\sqrt{xy}\theta )^{2\lambda +4}}d\theta ,\quad t,x,y\in (0,\infty ).
$$
Then,
\begin{eqnarray}\label{M2}
\|M_{1,2}^\lambda (\cdot ;x,y)\|_H&\leq&C\|M_{1,1}^\lambda (t,x,y)\|_H\leq Cx^\lambda y^{\lambda +1}\int_0^{\pi /2}\theta ^{2\lambda +1}\left(\int_0^\infty \frac{t^3}{(|x-y|+t+\sqrt{xy}\theta )^{4\lambda +8}}dt\right)^{1/2}d\theta\nonumber\\
&\leq&Cx^\lambda y^{\lambda +1}\int_0^{\pi /2}\frac{\theta ^{2\lambda +1}}{(|x-y|+\sqrt{xy}\theta )^{2\lambda +2}}d\theta \nonumber\\
&\leq&C\sqrt{\frac{y}{x}}\int_0^{\pi /2}\frac{d\theta}{|x-y|+\sqrt{xy}\theta}\leq \frac{C}{x}\log \left(1+\frac{\sqrt{xy}}{|x-y|}\right),\quad x,y\in (0,\infty ).
\end{eqnarray}

We put
$$
J_3^\lambda (f)=M_{1,1}^\lambda(f)+M_{1,2}^\lambda(f)+M_2^\lambda (f)+H^\lambda (f),
$$
being
\begin{eqnarray*}
M_{1,j}^\lambda (f)(t,x)=\int_{x/2}^{2x}M_{1,j}^\lambda (t;x,y)f(y)dy,\quad t,x\in (0,\infty ),\;j=1,2,\\
M_2^\lambda (f)(t,x)=\int_{x/2}^{2x}M_2^\lambda (t;x,y)f(y)dy,\quad t,x\in (0,\infty ),
\end{eqnarray*}
and
$$
H^\lambda (f)(t,x)=\int_{x/2}^{2x}[M_{1,3}^\lambda (t;x,y)+t\partial_xP_t(x-y)]f(y)dy, \quad t,x\in (0,\infty ).
$$

According to (\ref{M1}) we obtain that
\begin{eqnarray*}
\|M_2^\lambda (f)(x)\|_{\gamma (H,\mathbb{B})}&\leq&C\int_{x/2}^{2x}\|M_2^\lambda (\cdot  ;x,y)\|_H\|f(y)\|_\mathbb{B}dy\\
&\leq&\frac{C}{x}\int_{x/2}^{2x}\|f(y)\|_\mathbb{B}dy\leq C\|f\|_{BMO_{\rm o}((0,\infty ),\mathbb{B})},\quad x\in (0,\infty ).
\end{eqnarray*}

Also, by taking into account (\ref{M2}) we can write
\begin{eqnarray*}
\|M_{1,j}^\lambda (f)(\cdot ,x)\|_{\gamma (H,\mathbb{B})}&\leq&\frac{C}{x}\int_{x/2}^{2x}\log \left(1+\frac{\sqrt{xy}}{|x-y|}\right)\|f(y)\|_\mathbb{B}dy\\
&\leq&C\left(\frac{1}{x}\int_{x/2}^{2x}\left(\log \left(1+\frac{\sqrt{xy}}{|x-y|}\right)\right)^2dy\right)^{1/2}\left(\frac{1}{x}\int_{x/2}^{2x}\|f(y)\|_\mathbb{B}^2dy\right)^{1/2}\\
&\leq&C\left(\int_{1/2}^{2}\left(\log \left(1+\frac{\sqrt{u}}{|1-u|}\right)\right)^2du\right)^{1/2}\|f\|_{BMO_{\rm o}((0,\infty ),\mathbb{B})}\\
&\leq&C\|f\|_{BMO_{\rm o}((0,\infty ),\mathbb{B})},\quad x\in (0,\infty )\mbox{ and }j=1,2.
\end{eqnarray*}
We now consider
$$
H^\lambda (t;x,y)=M_{1,3}^\lambda (t;x,y)+t\partial _xP_t(x-y),\quad t,x,y\in (0,\infty ).
$$
We have that
\begin{eqnarray*}
H^\lambda (t;x,y)&=&-\frac{2t^2}{\pi}\frac{x-y}{(t^2+(x-y)^2)^2}+\frac{4(\lambda +1)(\lambda +2)}{\pi}t^2(xy)^{\lambda +1}\int_0^{\pi /2} \frac{(x-y)(\sin \theta )^{2\lambda +1}}{((x-y)^2+t^2+2xy(1-\cos \theta ))^{\lambda +3}}d\theta\\
&\hspace{-4cm}=&\hspace{-2cm}-\frac{2t^2}{\pi}(x-y)\left[\frac{1}{(t^2+(x-y)^2)^2}-2(\lambda +1)(\lambda +2)(xy)^{\lambda +1}\int_0^{\pi /2} \frac{\theta ^{2\lambda +1}}{((x-y)^2+t^2+xy\theta ^2)^{\lambda +3}}d\theta\right.\\
&\hspace{-4cm}&\hspace{-2cm}+2(\lambda +1)(\lambda +2)(xy)^{\lambda +1}\int_0^{\pi /2} \theta ^{2\lambda +1}\left(\frac{1}{((x-y)^2+t^2+xy\theta ^2)^{\lambda +3}}-\frac{1}{((x-y)^2+t^2+2xy(1-\cos \theta ))^{\lambda +3}}\right)d\theta\\
&\hspace{-4cm}&\hspace{-2cm}+\left.2(\lambda +1)(\lambda +2)(xy)^{\lambda +1}\int_0^{\pi /2}\frac{\theta ^{2\lambda +1}-(\sin \theta)^{2\lambda +1}}{((x-y)^2+t^2+2xy(1-\cos \theta ))^{\lambda +3}}d\theta \right]\\
&\hspace{-4cm}=&\hspace{-2cm}\sum_{j=1}^3H_j^\lambda (t,x,y),\quad t,x,y\in (0,\infty ).
\end{eqnarray*}
By using mean value theorem we get
\begin{eqnarray*}
|H_2^\lambda (t;x,y)|&\leq &Ct^2|x-y|(xy)^{\lambda +1}\int_0^{\pi /2} \frac{\theta ^{2\lambda +1}xy|1-\cos \theta -\theta ^2/2|}{((x-y)^2+t^2+xy\theta ^2)^{\lambda +4}}d\theta\\
&\leq&Ct^2|x-y|(xy)^{\lambda +2}\int_0^{\pi /2} \frac{\theta ^{2\lambda +5}}{(|x-y|+t+\sqrt{xy}\theta )^{2\lambda +8}}d\theta\\
&\leq&Ct^2|x-y|(xy)^{\lambda +1}\int_0^{\pi /2} \frac{\theta ^{2\lambda +3}}{(|x-y|+t+\sqrt{xy}\theta )^{2\lambda +6}}d\theta ,\quad t,x,y\in (0,\infty  ),
\end{eqnarray*}
and
$$
|H_3^\lambda (t;x,y)|\leq Ct^2|x-y|(xy)^{\lambda +1}\int_0^{\pi /2} \frac{\theta ^{2\lambda +3}}{(|x-y|+t+\sqrt{xy}\theta )^{2\lambda +6}}d\theta ,\quad t,x,y\in (0,\infty  ).
$$
Then
\begin{eqnarray*}
\|H_j^\lambda (t;x,y)\|_H&\leq &C(xy)^{\lambda +1}\int_0^\infty \theta ^{2\lambda +3}\left(\int_0^\infty \frac{t^3}{(|x-y|+t+\sqrt{xy}\theta )^{4\lambda +10}}dt\right)^{1/2}d\theta\\
&\leq&C(xy)^{\lambda +1}\int_0^{\pi /2}\frac{\theta ^{2\lambda +3}}{(|x-y|+\sqrt{xy}\theta )^{2\lambda +3}}d\theta\\
&\leq&C\frac{(xy)^{\lambda +1}}{(\sqrt{xy})^{2\lambda +3}}\leq \frac{C}{x},\quad 0<\frac{x}{2}<y<2x,\;j=2,3.
\end{eqnarray*}
On the other hand, after straightforward change of variables we get
\begin{eqnarray*}
H_1^\lambda (t;x,y)&=&-\frac{2t^2}{\pi}(x-y)\left[\frac{1}{(t^2+(x-y)^2)^2}-2(\lambda +1)(\lambda +2)(xy)^{\lambda +1}\right.\\
&&\left.\times\left(\int_0^\infty -\int_{\pi /2}^\infty \right) \frac{\theta ^{2\lambda +1}}{((x-y)^2+t^2+xy\theta ^2)^{\lambda +3}}d\theta\right]\\
&=&-\frac{2t^2}{\pi}(x-y)\left[\frac{1}{(t^2+(x-y)^2)^2}\left(1-2(\lambda +1)(\lambda +2)\int_0^\infty \frac{u^{2\lambda +
1}}{(1+u^2)^{\lambda +3}}du\right)\right.\\
&&\left. +2(\lambda +1)(\lambda +2)(xy)^{\lambda +1}\int_{\pi /2}^\infty \frac{\theta ^{2\lambda +1}}{((x-y)^2+t^2+xy\theta ^2)^{\lambda +3}}d\theta\right]\\
&=&-\frac{4(\lambda +1)(\lambda +2)}{\pi}t^2(x-y)(xy)^{\lambda +1}\int_{\pi /2}^\infty \frac{\theta ^{2\lambda +1}}{((x-y)^2+t^2+xy\theta ^2)^{\lambda +3}}d\theta,\quad t,x,y\in (0,\infty ).
\end{eqnarray*}
It follows that
\begin{eqnarray*}
\|H_1 ^\lambda (t;x,y)\|_H&\leq &C(xy)^{\lambda +1}\int_{\pi /2}^\infty \theta ^{2\lambda +1}\left(\int_0^\infty \frac{t^3}{(|x-y|+t+\sqrt{xy}\theta )^{4\lambda +10}}dt\right)^{1/2}d\theta\\
&\leq&C(xy)^{\lambda +1}\int_{\pi /2}^\infty \frac{\theta ^{2\lambda +1}}{(|x-y|+\sqrt{xy}\theta )^{2\lambda +3}}d\theta\\
&\leq&C\frac{(xy)^{\lambda +1}}{(\sqrt{xy})^{2\lambda +3}}\int_{\pi /2}^\infty \frac{d\theta }{\theta ^2}\leq \frac{C}{x},\quad 0<\frac{x}{2}<y<2x<\infty .
\end{eqnarray*}
We conclude that
$$
\|H^\lambda (\cdot ;x,y)\|_H\leq \frac{C}{x},\quad 0<\frac{x}{2}<y<2x<\infty ,
$$
and hence, for each $x\in (0,\infty )$,
$$
\|H^\lambda (f)(\cdot ,x)\|_{\gamma (H,\mathbb{B})}\leq C\int_{x/2}^{2x}\|H^\lambda (\cdot ;x,y)\|_H\|f(y)\|_\mathbb{B}dy
\leq \frac{C}{x}\int_0^{2x}\|f(y)\|_\mathbb{B}dy\leq C\|f\|_{BMO_{\rm o}((0,\infty ),\mathbb{B})}.
$$
By putting together the above estimates we obtain that
\begin{equation}\label{H41}
\|\mathcal{G}_\mathbb{B}(f_{\rm e})(\cdot ,x)-\mathcal{G}_\mathbb{B}^\lambda (f)(\cdot ,x)\|_{\gamma (H,\mathbb{B})}\leq C\|f\|_{BMO_{\rm o}((0,\infty ),\mathbb{B})},\quad x\in (0,\infty  ).
\end{equation}
Since $\mathcal{G}_\mathbb{B}(f_1)(\cdot ,x)\in \gamma (H,\mathbb{B})$, a.e. $x\in (0,\infty )$, $\mathcal{G}_\mathbb{B}^\lambda (f)(\cdot ,x)\in \gamma (H,\mathbb{B})$, a.e. $x\in (0,\infty )$, and
$\mathcal{G}_\mathbb{B}(f_3)(t ,x)=0$, $t,x\in (0,\infty )$, we deduce that $\mathcal{G}_\mathbb{B}(f_2)(\cdot ,x)\in \gamma (H,\mathbb{B})$, a.e. $x\in (0,\infty )$.

On the other hand,
$$
\left\|\partial _x\left[t\partial _xP_t(x-y)\right]\right\|_H\leq C\left(\int_0^\infty \frac{t^3}{(t+|x-y|)^8}dt\right)^{1/2}\leq \frac{C}{|x-y|^2},\quad x,y\in (0,\infty ),\;x\not=y.
$$
We have established all the properties needed in order to show, by using the arguments developed in Section 2.2, that
\begin{equation}\label{G2}
\frac{1}{|I|}\int_I\|\mathcal{G}_\mathbb{B}(f_2)(\cdot ,x)-\mathcal{G}_\mathbb{B}(f_2)(\cdot ,x_0)\|_\mathbb{B}dx\leq C\|f\|_{BMO_{\rm o}((0,\infty ),\mathbb{B})},
\end{equation}
where $x_0\in I$ is chosen such that $\mathcal{G}_\mathbb{B}(f_2)(\cdot ,x_0)\in \gamma (H,\mathbb{B})$ and the constant $C>0$ does not depend on $I$.

From (\ref{G1}), \eqref{H41} and \eqref{G2} we deduce that there exists $C>0$ such that, for every interval $I\subset (0,\infty )$, we can find $\alpha _I\in \gamma (H,\mathbb{B})$ such that
\begin{equation}\label{H42}
\frac{1}{|I|}\int_I\|\mathcal{G}_\mathbb{B}^\lambda (f)(\cdot ,x)-\alpha _I\|_{\gamma (H,\mathbb{B})}dx\leq C\|f\|_{BMO_{\rm o}((0,\infty ),\mathbb{B})}.
\end{equation}
\end{proof}

%\newpage
%%%%%%%%%%%%%%%%%%%%%%%%%%%%%%%%%%%%%%%%%%%%%%%%%%%%%%%%%%%%%%%%%%%%%%%%%%%%%%%%%%%%%%%%%%%%%%%%%%%%%%%%%%%%%%%%%%%%
\section{Proof of Theorem~\ref{Th:3}}\label{sec:Proof3}
%%%%%%%%%%%%%%%%%%%%%%%%%%%%%%%%%%%%%%%%%%%%%%%%%%%%%%%%%%%%%%%%%%%%%%%%%%%%%%%%%%%%%%%%%%%%%%%%%%%%%%%%%%%%%%%%%%%%

Properties $(i)\Longrightarrow (ii)$ and $(i)\Longrightarrow (iii)$ have been established in Theorem \ref{Th:1} and Theorem \ref{Th:2}.

\subsection {} We now prove that $(ii) \Longrightarrow (i)$. In order to do this we use Riesz transfoms in the Bessel setting.
The Riesz transform $R_\lambda$ is defined by
$$
R_\lambda (f)(x)=\lim_{\varepsilon \rightarrow 0^+}\int_{0,|x-y|>\varepsilon}^\infty R_\lambda (x,y)f(y)dy,\mbox{ a.e. }x\in (0,\infty ),
$$
for every $f\in L^p(0,\infty )$, $1\leq p<\infty$. Here the kernel $R_\lambda$ is given by
$$
R_\lambda (x,y)=\int_0^\infty D_{\lambda ,x}P_t^\lambda (x,y)dt,\quad x,y\in (0,\infty ),\;x\not =y.
$$
$R_\lambda $ is a Calder\'on-Zygmund operator that is bounded from $L^p(0,\infty )$ into itself and from $L^1(0,\infty )$ into $L^{1,\infty }(0,\infty )$ (\cite[Theorem 4.2]{BBFMT}).

We define the Riesz transform on $L^p(0,\infty )\otimes \mathbb{B}$, $1\leq p<\infty$, in the natural way, and the new operator is also denoted by $R_\lambda$. In \cite[Theorem 2.1]{BFMT} it was proved that the Banach space $\mathbb{B}$ is UMD if, and only if, $R_\lambda$ can be extended from $L^p(0,\infty )\otimes \mathbb{B}$ to $L^p((0,\infty ),\mathbb{B})$ as a bounded operator from $L^p((0,\infty ),\mathbb{B})$ into itself, for some (equivalently, for any) $1<p<\infty$.

We consider the operator $R_\lambda ^*$ defined by
$$
R_\lambda ^*(f)(x)=\lim_{\varepsilon \rightarrow 0^+}\int_{0,|x-y|>\varepsilon}^\infty R_\lambda (y,x)f(y)dy,\mbox{ a.e. }x\in (0,\infty ),
$$
for every $f\in L^p(0,\infty )$, $1\leq p<\infty$. This operator has the same $L^p$-boundedness properties of $R_\lambda$. Also, the Banach space $\mathbb{B}$ is UMD if and only if $R_\lambda ^*$ can be extended from $L^p(0,\infty )\otimes \mathbb{B}$ to $L^p((0,\infty ),\mathbb{B})$ as a bounded operator from $L^p((0,\infty ),\mathbb{B})$ into itself, for some (equivalently, for any) $1<p<\infty$.

In \cite[(16.6)]{MS} Cauchy-Riemann type equations in the Bessel setting were given. Motivated by these equations and using Hankel transform we can see that
$$
\partial _tP_t^\lambda (R_\lambda ^*(f))=D_\lambda ^*P_t^{\lambda +1}(f),\quad f\in S_\lambda (0,\infty ).
$$
In other words, we have that
$$
G_\mathbb{C}^{\lambda ,1}(R_\lambda ^*(f))=\mathcal{G}_\mathbb{C}^\lambda (f),\quad f\in S_\lambda (0,\infty ).
$$
Since $S_ \lambda (0,\infty )$ is a dense subspace of $L^p(0,\infty )$, $1<p<\infty$, by using \cite[Theorems 1.2 and 1.3]{BCR3} we obtain
\begin{equation}\label{H42A}
G_\mathbb{C}^{\lambda ,1}(R_\lambda ^*(f))=\mathcal{G}_\mathbb{C}^\lambda (f),\quad f\in L^p(0,\infty ),\;1<p<\infty .
\end{equation}
Now, we need to know the behaviour of  $R_\lambda ^*$ on $H^1_{\rm o}(\mathbb{R})$ and on $BMO_{\rm o}(\mathbb{R})$.

\begin{Prop}\label{Riesz}
Let $\lambda >0$. The Riesz transform $R_\lambda ^*$ is a bounded operator from $E(0,\infty )$ into itself, where $E$ denotes $H^1_{\rm o}$ or $BMO_{\rm o}$.
\end{Prop}

\begin{proof}[Proof of Proposition \ref{Riesz}: the case of $E=H^1_{\rm o}$.]

Since $R_\lambda ^*$ is a Calder\'on-Zygmund operator, $R_\lambda ^*$ is bounded from $H^1(0,\infty )$ into $L^1(0,\infty )$, and there exists $C>0$ such that, for every $\infty$-atom $a$ satisfying the $(Aii)$ property,
$$
\|R_\lambda ^*(a)\|_{L^1(0,\infty )}\leq C.
$$
According to \cite[(3.16)]{BBFMT},
\begin{equation}\label{H43}
|R_\lambda (y,x)|\leq Cy^{\lambda +1}x^{-\lambda -2},\quad 0<y<\frac{x}{2}<\infty .
\end{equation}
Let $\delta >0$. If $a=\frac{1}{\delta }\chi _{(0,\delta)}$, being $R_\lambda ^*$ bounded from $L^2(0,\infty )$ into itself, by (\ref{H43}) we have that
\begin{eqnarray*}
\|R_\lambda ^*(a)\|_{L^1(0,\infty )}&=&\int_0^{2\delta }|R_\lambda ^*(a)(x)|dx+\int_{2\delta }^\infty \int_0^\delta |R_\lambda (y,x)|\frac{dy}{\delta }dx\\
&\leq &\left(\delta \int_0^\infty |R_\lambda ^*(a)(x)|^2dx\right)^{1/2}+\frac{C}{\delta }\int_{2\delta }^\infty \frac{1}{x^{\lambda +2}}\int_0^\delta y^{\lambda +1}dydx\\
&\leq&C\left(1+\left(\delta \int_0^\delta |a(x)|^2dx\right)^{1/2}\right)\leq C.
\end{eqnarray*}
Here $C>0$ does not depend on $\delta$.

Moreover, since $R_\lambda ^*$ is a bounded operator from $L^1(0,\infty )$ into $L^{1,\infty }(0,\infty )$ we obtain that $R_\lambda ^*$ is bounded from $H^1_{\rm o}(0,\infty )$ into $L^1(0,\infty )$.

On the other hand, it can be seen that $R_\lambda (f)=h_{\lambda +1}(h_\lambda(f)) $, $f\in L^2(0,\infty )$ (\cite[\S 16]{MS}). Then,
$$
R_\lambda ^*(f)=h_\lambda (h_{\lambda +1}(f)),\quad f\in L^2(0,\infty ).
$$
Then, by using Hankel transforms we get, for every $f\in L^2(0,\infty )$,
\begin{equation}\label{H44}
P_s^\lambda (R_\lambda ^*(f))=h_\lambda (e^{-sy}h_{\lambda +1}(f))=h_\lambda h_{\lambda +1}(P_s^{\lambda +1}(f))=R_\lambda ^*(P_s^{\lambda +1}(f)),\quad s\in (0,\infty ).
\end{equation}
We consider the operators
$$
\mathcal{H}_*^\lambda (f)(x)=\sup_{s>0}|P_s^\lambda (R_\lambda ^*(f))(x)|,\quad x\in (0,\infty ),
$$
and, for every $N\in \mathbb{N}$,
$$
\mathcal{H}_{N,*}^\lambda (f)(x)=\sup_{s\in [1/N,N]}|P_s^\lambda (R_\lambda ^*(f))(x)|,\quad x\in (0,\infty ),
$$
where $f\in L^p(0,\infty )$, $1<p<\infty$. $\mathcal{H}_*^\lambda $ and $\mathcal{H}_{N,*}^\lambda $, $N\in \mathbb{N}$, are bounded operators from $L^p(0,\infty )$ into itself, for each $1<p<\infty$.

Since $D_{\lambda ,y}P_t^\lambda (x,y)=-D_{\lambda ,x}^*P_t^{\lambda +1}(x,y)$, $t,x,y\in (0,\infty )$, (\ref{H44}) leads to
$$
P_s^\lambda (R_\lambda ^*(f))(x)=\int_0^\infty \mathcal{H}^\lambda (s;x,y)f(y)dy,\quad \mbox{ a.e. }x\not \in \mbox{supp }f,
$$
for every $f\in L^2(0,\infty )$, where
$$
\mathcal{H}^\lambda (s;x,y)=-\int_0^\infty D_{\lambda ,x}^*P_{t+s}^{\lambda +1}(x,y)dt,\quad s,x,y\in (0,\infty ).
$$
By (\ref{H30A}) and (\ref{Ical}) we have that
\begin{eqnarray}\label{64}
\int_0^\infty |D_{\lambda ,x}^*P_{t+s}^{\lambda +1}(x,y)|dt&\leq&C\int_0^\infty \frac{t+s}{(t+s+|x-y|)^3}dt\leq C\int_0^\infty \frac{dt}{(t+s+|x-y|)^2}\nonumber\\
&\leq &\frac{C}{s+|x-y|},\quad s,x,y\in (0,\infty ).
\end{eqnarray}
Then
\begin{equation}\label{H45}
\|\mathcal{H}^\lambda (\cdot ;x,y)\|_{L^\infty (0,\infty )}\leq \frac{C}{|x-y|},\quad x,y\in (0,\infty ),\;x\not=y.
\end{equation}
On the other hand, from (\ref{H32A}), (\ref{26.1}) and (\ref{J1}) it follows that
$$
\int_0^\infty |\partial _xD_{\lambda ,x}^*P_{t+s}^{\lambda +1}(x,y)|\leq C\int_0^\infty \frac{t+s}{(t+s+|x-y|)^4}dt\leq \frac{C}{(s+|x-y|)^2},\quad s,x,y\in (0,\infty ).
$$
Hence,
\begin{equation}\label{H46}
\left\|\partial _x\mathcal{H}^\lambda (\cdot ;x,y)\right\|_{L^\infty (0,\infty )}\leq\frac{C}{|x-y|^2},\quad x,y\in (0,\infty ),\,x\not=y.
\end{equation}
The same arguments lead to
\begin{equation}\label{H47}
\left\|\partial _y\mathcal{H}^\lambda (\cdot ;x,y)\right\|_{L^\infty (0,\infty )}\leq\frac{C}{|x-y|^2},\quad x,y\in (0,\infty ),\,x\not=y.
\end{equation}
Let $N\in \mathbb{N}$. We consider the operator
$$
\mathcal{H}_N^\lambda (f)(s,x)=P_s^\lambda (R_\lambda ^*(f))(x),\quad x\in (0,\infty )\mbox{ and }s\in \Big[\frac{1}{N},N\Big].
$$
For every $f\in S_\lambda (0,\infty )$, we have that
\begin{equation}\label{H48}
\mathcal{H}_N^\lambda (f)(\cdot, x)=\int _0^\infty \mathcal{H}_\lambda (\cdot ;x,y)f(y)dy,\quad \mbox{ a.e. }x\not \in \mbox{supp }f,
\end{equation}
where the integral is understood in the $C[1/N,N]$-Bochner sense, the space of continuous functions in $[1/N,N]$. Indeed, let $f\in S_\lambda (0,\infty )$.
According to (\ref{H45}), the integral in (\ref{H48}) is convergent in the $C[1/N,N]$-Bochner sense, for almost all $x\not\in \mbox{supp }f$.
Assume that $\mu \in \mathcal{M}([1/N,N])$, the space of complex measures supported on $[1/N,N]$. Since $\mathcal{M}([1/N,N])=(C[1/N,N])^*$, we can write
\begin{eqnarray*}
\int_0^\infty  \left(\int_0^\infty \mathcal{H}^\lambda (\cdot ;x,y)f(y)dy\right)d\mu (s)&=&\int_0^\infty \int_0^\infty \mathcal{H}^\lambda (s;x,y)d\mu (s)f(y)dy\\
&=&\int_0^\infty  \left(\int_0^\infty \mathcal{H}^\lambda (s;x,y)f(y)dy\right)d\mu (s),\quad x\not \in \mbox{supp }f.
\end{eqnarray*}
Note that the interchange of the order of integration is justified by (\ref{H45}).

According to the vector-valued Calder\'on-Zygmund theory, since $\mathcal{H}_*^\lambda $ is bounded from $L^2(0,\infty )$ into itself, estimations (\ref{H45}), (\ref{H46}) and \eqref{H47} imply that the operator $\mathcal{H}_N^\lambda $ can be extended to $L^1(0,\infty )$ as a bounded ope\-rator from $L^1(0,\infty )$ into $L^{1,\infty }((0,\infty ),C[1/N,N])$ and from $H^1(0,\infty )$ into $L^1((0,\infty ),C[1/N,N])$. Moreover, if we denote by $\widetilde{\mathcal{H}}_N^\lambda $ to this extension we have that
$$
\sup_{N\in \mathbb{N}}\|\widetilde{\mathcal{H}}_N^\lambda \|_{L^1(0,\infty )\rightarrow L^{1,\infty }((0,\infty ),C[1/N,N])}<\infty ,
$$
and
$$
\sup_{N\in \mathbb{N}}\|\widetilde{\mathcal{H}}_N^\lambda \|_{H^1(0,\infty )\rightarrow L^1((0,\infty ),C[1/N,N])}<\infty .
$$
We now show that $\widetilde{\mathcal{H}}_N^\lambda (f)=\mathcal{H}_N^\lambda (f)$, $f\in L^1(0,\infty )$.  Indeed, let $f\in L^1(0,\infty)$. We choose a sequence $(f_n)_{n\in \N}\subset S_\lambda (0,\infty )$ such that $f_n\longrightarrow f$, as $n\rightarrow \infty $, in $L^1(0,\infty )$. Then
$$
\mathcal{H}_N^\lambda (f_n)\longrightarrow \widetilde{\mathcal{H}}_N^\lambda (f),\mbox{ as }n\rightarrow \infty ,
$$
in $L^{1,\infty }((0,\infty ),C[1/N,N])$. There exists an increasing sequence $(n_k)_{k\in \mathbb{N}}$ such that
$$
\|\mathcal{H}_N^\lambda (f_{n_k})(\cdot ,x)-\widetilde{\mathcal{H}}_N^\lambda (f)(x)\|_{C[1/N,N]}\longrightarrow 0,\mbox{ as }k\rightarrow \infty ,
$$
for almost all $x\in (0,\infty )$. Moreover by (\ref{64}),
$$
\mathcal{H}_N^\lambda (f_{n_k})(s,x)\longrightarrow \mathcal{H}_N^\lambda (f)(s,x),\mbox{ as }k\rightarrow \infty ,\quad s,x\in (0,\infty ).
$$
Hence, $\mathcal{H}_N^\lambda (f)(s,x)=[\widetilde{\mathcal{H}}_N(f)(x)](s)$, $s\in [1/N,N]$ and almost all $x\in (0,\infty )$.

We conclude that $\mathcal{H}_N^\lambda$ is bounded from $L^1(0,\infty )$ into $L^{1,\infty }((0,\infty ),C[1/N,N])$ and from $H^1(0,\infty )$ into $L^1((0,\infty ),C[1/N,N])$. Moreover
$$
\sup_{N\in \mathbb{N}}\|\mathcal{H}_N^\lambda \|_{L^1(0,\infty )\rightarrow L^{1,\infty }((0,\infty ),C[1/N,N])}<\infty ,
$$
and
$$
\sup_{N\in \mathbb{N}}\|\mathcal{H}_N^\lambda \|_{H^1(0,\infty )\rightarrow L^1((0,\infty ),C[1/N,N])}<\infty .
$$

Then, $\mathcal{H}_*^\lambda $ is bounded from $L^1(0,\infty )$ into $L^{1,\infty }(0,\infty )$ and from $H^1(0,\infty )$ into $L^1(0,\infty )$.

In order to see that $\mathcal{H}_ *^\lambda$ is bounded from $H_{\rm o}^1(0,\infty )$ into $L^1(0,\infty )$ it is enough to show that there exists $C>0$ such that, for every $\infty$-atom $a$,
\begin{equation}\label{H48a}
\|\mathcal{H}_*^\lambda (a)\|_{L^1(0,\infty )}\leq C.
\end{equation}
Since  $\mathcal{H}_ *^\lambda$ is bounded from $H^1(0,\infty )$ into $L^1(0,\infty )$, for a certain $C>0$, (\ref{H48a}) holds provided that $a$ is a an $\infty$-atom satisfying $(Aii)$. Also, by using (\ref{H30A}), (\ref{I1b}) and (\ref{I2b}), and by proceeding as in (\ref{64}) we deduce that
\begin{equation}\label{H49}
\|\mathcal{H}^\lambda (\cdot; x,y)\|_{L^\infty (0,\infty )}\leq C(xy)^\lambda \left(\frac{xy}{|x-y|^{2\lambda +3}}+\frac{y}{|x-y|^{2\lambda +2}}\right)\leq C\frac{y^{\lambda +1}}{x^{\lambda +2}},\quad 0<y<\frac{x}{2}<\infty .
\end{equation}
Also $\mathcal{H}_ *^\lambda$ is bounded from $L^2(0,\infty )$ into itself. Then, for every $a=\frac{1}{\delta }\chi _{(0,\delta )}$, with $\delta >0$, by (\ref{H49}) we get
\begin{eqnarray}\label{H50}
\|\mathcal{H}_*^\lambda (a)\|_{L^1(0,\infty )}&\leq&\int_0^{2\delta }|\mathcal{H}_*^\lambda (a)(x)|dx+\int_{2\delta }^\infty \int_0^\delta \|\mathcal{H}^\lambda (\cdot ;x,y)\|_{L^\infty (0,\infty )}|a(y)|dydx\nonumber\\
&\leq&C\left(\delta ^{1/2}\|a\|_{L^2(0,\infty )}+\frac{1}{\delta }\int_{2\delta }^\infty \frac{1}{x^{\lambda +2}}\int_0^\delta y^{\lambda +1}dy\right)\leq C,
\end{eqnarray}
where $C>0$ does not depend on $\delta$.

Hence (\ref{H48a}) holds and we obtain that $\mathcal{H}_ *^\lambda$ is bounded from $H_{\rm o}^1(0,\infty )$ into $L^1(0,\infty )$. We conclude that $R_\lambda ^*$ is bounded from $H_{\rm o}^1(0,\infty )$ into itself.
\end{proof}

\begin{proof}[Proof of Proposition \ref{Riesz}: the case of $E=BMO_{\rm o}(0,\infty )$]
Suppose that $f\in BMO_{\rm o}(0,\infty )$. Let $r>0$. Since $R_\lambda ^*$ is bounded from $L^2(0,\infty )$ into itself and that (see \cite[(3.15)]{BBFMT})
\begin{equation}\label{H51}
|R_\lambda (y,x)|\leq C\frac{x^\lambda }{y^{\lambda +1}}, \quad 0<2x<y<\infty,
\end{equation}
we obtain
\begin{eqnarray*}
\frac{1}{r}\int_0^r|R_\lambda ^*(f)(x)|dx&\leq&\frac{1}{r}\left(\int_0^r|R_\lambda ^*(f\chi _{(0, 2r)})(x)|dx+\int_0^r|R_\lambda ^*(f\chi _{(2r,\infty )})(x)|dx\right)\\
&\leq&C\left[\left(\frac{1}{r}\int_0^r|R_\lambda ^*(f\chi _{(0, 2r)})(x)|^2dx\right)^{1/2}+\frac{1}{r}\int_0^r\int_{2r}^\infty |R_\lambda (y,x)||f(y)|dydx\right]\\
&\leq&C\left[\left(\frac{1}{r}\int_0^{2r}|f(y)|^2dy\right)^{1/2}+r^\lambda \int_{2r}^\infty \frac{|f(y)|}{y^{\lambda +1}}dy\right]\\
&\leq&C\|f\|_{BMO_{\rm o}(0,\infty )},
\end{eqnarray*}
where $C>0$ does not depend on $r$.

The next step is to show that $R_\lambda ^*(f)\in BMO(0,\infty )$.

We consider the even extension $f_{\rm e}$ of the function $f$ to $\R$.
Then, $f_{\rm e}\in BMO (\mathbb{R})$. According to \cite[p. 294]{Tor} the Hilbert transform $\mathfrak{H}(f_{\rm e})$ given by
$$
\mathfrak{H}(f_e)(x)=\lim_{\varepsilon \rightarrow 0^+}\frac{1}{\pi }\int_{|x-y|>\varepsilon }\left(\frac{1}{x-y}-\frac{\chi _{\mathbb{R}\setminus (-1,1)}(y)}{y}\right)f_e(y)dy,\quad \mbox{ a.e. }x\in \mathbb{R},
$$
is in $BMO(\mathbb{R})$.

In \cite[(26) and (27)]{BCFR2} it was seen that
$$
\mathfrak{H}(f_e)(x)-\lim_{\varepsilon \rightarrow 0^+}\frac{1}{\pi }\int_{x/2,|x-y|>\varepsilon }^{2x}\left(\frac{1}{x-y}+\frac{\chi _{(1,\infty )}(y)}{y}\right)f(y)dy \;\;\in\;\; L^\infty (0,\infty ),
$$
and
\begin{equation}\label{H52}
\left\|\mathfrak{H}(f_e)(x)-\lim_{\varepsilon \rightarrow 0^+}\frac{1}{\pi }\int_{x/2,|x-y|>\varepsilon }^{2x}\left(\frac{1}{x-y}+\frac{\chi _{(1,\infty )}(y)}{y}\right)f(y)dy\right\|_{L^\infty (0,\infty )}\leq C\|f\|_{BMO_{\rm o}(0,\infty )}.
\end{equation}
We now show that
\begin{equation}\label{H53}
\left\|R_\lambda ^*(f)(x)+\lim_{\varepsilon \rightarrow 0^+}\frac{1}{\pi }\int_{x/2,|x-y|>\varepsilon }^{2x}\left(\frac{1}{x-y}+\frac{\chi _{(1,\infty )}(y)}{y}\right)f(y)dy
\right\|_{L^\infty (0,\infty )}\leq C\|f\|_{BMO_{\rm o}(0,\infty )}.
\end{equation}
By using \cite[(5)]{BCFR2}, (\ref{H43}) and (\ref{H51}), as in \cite[pp. 319 and 320]{BCFR2} we deduce that
\begin{eqnarray*}
\left|R_\lambda ^*(f)(x)+\lim_{\varepsilon \rightarrow 0^+}\frac{1}{\pi }\int_{x/2,|x-y|>\varepsilon }^{2x}\left[\frac{1}{x-y}+\frac{\chi _{(1,\infty )}(y)}{y}\right)f(y)dy\right|&&\\
&\hspace{-12cm}\leq&\hspace{-6cm}C\left[\int_{x/2}^{2x}\frac{1}{y}\left(1+\log _+\frac{\sqrt{xy}}{|x-y|}\right)|f(y)|dy+\int_{x/2}^{2x}\frac{\chi _{(1,\infty )}(y)}{y}|f(y)|dy\right.\\
&\hspace{-12cm}&\hspace{-6cm}\left.+\frac{1}{x^{\lambda +2}}\int_0^{x/2}y^{\lambda +1}|f(y)|dy+x^\lambda \int_{2x}^\infty \frac{|f(y)|}{y^{\lambda +1}}dy\right]\\
&\hspace{-12cm}\leq&\hspace{-6cm}C\left[\left(\int_{x/2}^ {2x}\frac{1}{y}\left(1+\log _+\frac{\sqrt{xy}}{|x-y|}\right)^2dy\right)^{1/2}\left(\frac{1}{x}\int_0^{2x}|f(y)|^2dy\right)^{1/2}\right.\\
&\hspace{-12cm}&\hspace{-6cm}+\frac{1}{x}\int_0^{2x}|f(y)|dy\left. +x^\lambda \int_{2x}^\infty \frac{|f(y)|}{y^{\lambda +1}}dy\right]\leq C\|f\|_{BMO_{\rm o}(0,\infty )},\quad x\in (0,\infty ),
\end{eqnarray*}
where $\log_+z=\max\{0,\log z\}$, $z\in (0,\infty )$. Thus, (\ref{H53}) is established.

From (\ref{H52}) and (\ref{H53}), since $\mathfrak{H}(f_{\rm e})\in BMO(\mathbb{R})$, we deduce that $R_\lambda ^*(f)\in BMO(0,\infty )$.
\end{proof}

Let $f\in L^\infty _c(0,\infty )\otimes \mathbb{B}$. Since $L^\infty _c(0,\infty )\otimes \mathbb{B}\subset H_{\rm o}^1(0,\infty )\otimes \mathbb{B}$, by Proposition \ref{Riesz}, $R_\lambda ^* (f)\in H_{\rm o}^1(0,\infty )\otimes\mathbb{B}$.

Assume that (ii) holds for $H_{\rm o}^1(\mathbb{R},\mathbb{B})$. By (\ref{H42A}) it follows that
$$
\|R_\lambda ^*(f)\|_{H^1_{\rm o}((0,\infty ),\mathbb{B})}\leq C\|G_\mathbb{B}^{\lambda ,1}(R_\lambda ^*(f))\|_{H^1_{\rm o}((0,\infty ),\gamma (H,\mathbb{B}))}\leq C\|\mathcal{G}_\mathbb{B}^\lambda (f)\|_{H^1_{\rm o}((0,\infty ),\gamma (H,\mathbb{B}))}\leq C\|f\|_{H^1_{\rm o}((0,\infty ),\mathbb{B})}.
$$

Since $H^1((0,\infty ),\B)\subset H^1_{\rm o}((0,\infty ),\B)\subset L^1((0,\infty ),\B)$, according to \cite[Theorem 4.1]{MTX}, $R_\lambda ^*$ can be extended to $L^p((0,\infty),\mathbb{B})$ as a bounded operator from $L^p((0,\infty),\mathbb{B})$  into itself. Then, we conclude that $\mathbb{B}$ is UMD.

When it is assumed that $(ii)$ holds for $BMO_{\rm o}((0,\infty ),\mathbb{B})$, by proceeding in a similar way, we can prove that $\mathbb{B}$ is UMD.

Thus, the proof of $(ii)\Longrightarrow (i)$ is complete.

\subsection{} We prove in this section that $(iii)\Longrightarrow (i)$. In order to see this, we use that the Banach space $\mathbb{B}$ is UMD if, and only if, for every $\gamma \in \mathbb{R}\setminus\{0\}$, the imaginary power $\Delta _\lambda ^{i\gamma }$ can be extended to $L^p((0,\infty),\mathbb{B})$ as a bounded operator from $L^p((0,\infty),\mathbb{B})$  into itself, for some (equivalently, for every) $1<p<\infty$ (see \cite[Proposition 5.1]{BCR3}).

Let $\gamma \in \mathbb{R}\setminus\{0\}$. We recall that the imaginary power $\Delta _\lambda ^{i\gamma }$ of $\Delta _\lambda $ is defined by
$$
\Delta _\lambda ^{i\gamma }(f)=h_\lambda (y^{2i\gamma }h_\lambda (f)),\quad f\in L^2(0,\infty ).
$$
Since $h_\lambda $ is an isometry in $L^2(0,\infty)$, $\Delta _\lambda ^{i\gamma}$ is bounded from $L^2(0,\infty )$ into itself. Also, $\Delta _\lambda ^{i\gamma}$ is a Laplace transform type multiplier for the Bessel operator $\Delta _\lambda $, because
$$
y^{2i\gamma }=y^2\int_0^\infty e^{-y^2u}\frac{u^{-i\gamma}}{\Gamma (1-i\gamma )}du,\quad y\in (0,\infty ).
$$
Laplace transform type Hankel multipliers have been studied in \cite{BCC1} and \cite{BMR}. According to the results established in \cite{BCC1} and \cite{BMR} the imaginary power $\Delta_\lambda ^{i\gamma }$ is a Calder\'on-Zygmund operator defined by the kernel
\begin{equation}\label{40.1}
K_\lambda ^\gamma (x,y)=-\frac{1}{\Gamma (1-i\gamma )}\int_0^\infty t^{-i\gamma }\partial _tW_t^\lambda (x,y)dt,\quad x,y\in (0,\infty ),\;x\not =y,
\end{equation}
where $W_t^\lambda (x,y)$ is the heat kernel for the Bessel operator $\Delta _\lambda$, and it is given by
$$
W_t^\lambda (x,y)=\frac{(xy)^{1/2}}{2t}I_{\lambda -1/2}\left(\frac{xy}{2t}\right)e^{-(x^2+y^2)/(4t)},\quad t,x,y\in (0,\infty ).
$$
Here $I_\nu $ represents the modified Bessel function of the first class and order $\nu$.

Moreover, the operator $\Delta _\lambda ^{i\gamma }$ can be extended to $L^p(0,\infty )$ as a bounded operator from $L^p(0,\infty )$ into itself, for every $1<p<\infty$, from $L^1(0,\infty )$ into $L^{1,\infty }(0,\infty )$, and from $H^1(0,\infty )$ into $L^1(0,\infty )$. This extension, that we continue denoting by $\Delta _\lambda ^{i\gamma }$, has the following integral representation, for every $f\in L^p(0,\infty )$, $1\leq p<\infty$,
\begin{equation}\label{H53A}
\Delta _\lambda ^{i\gamma }(f)(x)=\lim_{\varepsilon \rightarrow 0^+}\left(\alpha (\varepsilon  )f(x)+\int_{0,|x-y|>\varepsilon }^\infty K_\lambda ^\gamma (x,y) f(y)dy\right),\quad \mbox{ a.e. }x\in (0,\infty ),
\end{equation}
where $\alpha $ is a measurable bounded function.

The operator $\Delta _\lambda ^{i\gamma}$ is defined on $L^p(0,\infty )\otimes \mathbb{B}$ in the natural way.

Let $\gamma \in \mathbb{R}\setminus\{0\}$ and $\beta >0$. We define the operator $T_{\gamma ,\beta}$ by
$$
T_{\gamma ,\beta}(h)(t)=\frac{1}{t^\beta}\int_0^t(t-s)^{\beta -1}h(t-s)\psi _\gamma (s)ds,\quad t>0,
$$
for every $h\in H$, where $\psi _\gamma (s)=s^{-2\gamma i}/\Gamma (1-2\gamma i)$, $s\in(0,\infty )$. $T_{\gamma ,\beta}$ is a bounded operator from $H$ into itself.

If $f\in S_\lambda (0,\infty )$ we proved in \cite[(55)]{BCR3} that
\begin{equation}\label{H54}
G_\mathbb{C}^{\lambda ,\beta}(\Delta _\lambda ^{i\gamma }(f))(\cdot ,x)=-G_\mathbb{C}^{\lambda ,\beta +1}(f)(\cdot ,x)\circ T_{\gamma ,\beta},\quad \mbox{ a.e. }x\in (0,\infty ),
\end{equation}
as elements of $\gamma (H,\mathbb{C})=H$. Moreover, the two sides of (\ref{H54}) define bounded operators from $L^p(0,\infty )$ into $L^p((0,\infty ),H)$ (see Theorem \ref{Th:1} and \cite[Theorem 6.2]{Nee}), for every $1<p<\infty$. Since $S_\lambda (0,\infty )$ is dense in $L^p(0,\infty )$, we deduce that (\ref{H54}) holds, for every $f\in L^p(0,\infty )$, $1<p<\infty$.

In the following proposition we establish the behaviour of $\Delta _\lambda ^{i\gamma}$ on $H^1_{\rm o}(0,\infty )$ and $BMO_{\rm o}(0,\infty )$.

\begin{Prop}\label{imaginarypower}
Let $\gamma \in  \mathbb{R}\setminus\{0\}$ and $\lambda >0$. The operator $\Delta _\lambda ^{i\gamma }$ is bounded from $E(0,\infty )$ into itself, where $E$ denotes $H^1_{\rm o}$ or $BMO_{\rm o}$.
\end{Prop}
\begin{proof}[Proof of Proposition \ref{imaginarypower}: the case of $E=H^1_{\rm o}$.]
In order to show that $\Delta _\lambda ^{i\gamma}$ is bounded from $H^1_{\rm o}(0,\infty )$ into itself we consider the function $\varphi _\gamma (x)=x^{2i\gamma}$, $x\in(0,\infty )$. It is not hard to see that
$$
\sup _{t>0}\|\eta (\cdot )\varphi _\gamma (t\cdot )\|_{L_1^2(0,\infty )}<\infty ,
$$
where $\eta\in C_c^\infty (0,\infty )$, the space of smooth functions with compact support on $(0,\infty )$, and $\|\cdot \|_{L_1^2(0,\infty )}$  denotes the Sobolev norm of order 1 and exponent 2. According to \cite[Theorem 4.11]{BDT} we deduce that $\Delta _\lambda ^{i\gamma}$ can be extended from $\mathcal{A}=\{\mbox{$a$ $\infty$-atom for $H_{\rm o}^1(0,\infty )$}\}$ to $H_{\rm o}^1(0,\infty )$ as a bounded operator from $H_{\rm o}^1(0,\infty )$ into itself. We denote by $\widetilde{\Delta}_\lambda ^{i\gamma }$ to this extension.

Suppose that $f=\sum_{j\in \N}\lambda _ja_j$, where, for every $j\in \mathbb{N}$, $a_j$ is a $\infty$-atom
and $\lambda _j\in \mathbb{C}$ such that $\sum_{j\in \N}|\lambda _j|<\infty $. We have that
$$
\widetilde{\Delta}_\lambda ^{i\gamma }(f)=\sum_{j\in \N}\lambda _j\Delta _\lambda ^{i\gamma }(a_j), \mbox{ in }H^1_{\rm o}(0,\infty ).
$$
Since $H_{\rm o}^1(0,\infty )$ is continuously contained in $L^1(0,\infty )$, the series $\sum_{j=0}^\infty\lambda _j\Delta _\lambda ^{i\gamma} (a_j)$ converges also in $L^1(0,\infty )$, and then in $L^{1,\infty }(0,\infty )$. Also, the operator $\Delta_\lambda ^{i\gamma} $ is bounded from $L^1(0,\infty )$ into $L^{1,\infty }(0,\infty )$. These facts lead to $\widetilde{\Delta }_\gamma ^{i\gamma}(f)=\Delta _\lambda ^{i\gamma}(f)$.

We conclude that $\Delta _\lambda ^{i\gamma }$ is bounded from $H^1_{\rm o}(0,\infty )$ into itself.
\end{proof}

\begin{proof}[Proof of Theorem \ref{imaginarypower}: the case of $E=BMO_{\rm o}$]
In a first step we prove that
\begin{equation}\label{H55}
|K_\lambda ^\gamma (x,y)|\leq C\frac{(xy)^\lambda }{|x-y|^{2\lambda +1}},\quad x,y\in (0,\infty ),\;x\not =y,
\end{equation}
being $K_\lambda ^\gamma $ the kernel in (\ref{40.1}).

We will use the following properties of Bessel functions $I_\nu $, $\nu >-1$ (see \cite[pp. 108, 110 and 123]{Leb}):
\begin{equation}\label{H56}
\frac{d}{dz}(z^{-\nu }I_\nu (z))=z^{-\nu}I_{\nu +1},\quad z\in (0,\infty ),
\end{equation}
\begin{equation}\label{H57}
I_\nu (z)\sim \frac{1}{2^\nu\Gamma (\nu +1)}z^\nu ,\mbox{ as }z\rightarrow 0^+,
\end{equation}
and
\begin{equation}\label{H58}
e^{-z}\sqrt{z}I_\nu (z)=\mathcal{O}\Big(\frac{1}{z}\Big).
\end{equation}

According to (\ref{H55}) we have that
\begin{eqnarray*}
\partial _tW_t^\lambda (x,y)&=&\partial _t\left[\frac{(xy)^\lambda }{(2t)^{\lambda +1/2}}\left(\frac{xy}{2t}\right)^{-\lambda +1/2}I_{\lambda -1/2}\left(\frac{xy}{2t}\right)e^{-(x^2+y^2)/(4t)}\right]\\
&=&e^{-(x^2+y^2)/(4t)}\left[-(2\lambda +1)(2t)^{-\lambda -3/2}(xy)^\lambda \left(\frac{xy}{2t}\right)^{-\lambda +1/2}I_{\lambda -1/2}\left(\frac{xy}{2t}\right)\right.\\
&&-\frac{(xy)^\lambda }{(2t)^{\lambda +1/2}}\frac{xy}{2t^2}\left(\frac{xy}{2t}\right)^{-\lambda +1/2}I_{\lambda +1/2}\left(\frac{xy}{2t}\right)\\
&&\left. +\frac{(xy)^\lambda }{(2t)^{\lambda +1/2}}\left(\frac{xy}{2t}\right)^{-\lambda +1/2}I_{\lambda -1/2}\left(\frac{xy}{2t}\right)\frac{x^2+y^2}{4t^2}\right],\quad t,x,y\in (0,\infty ).
\end{eqnarray*}
From (\ref{H57}) we deduce that
\begin{eqnarray*}
\left|\partial _tW_t^\lambda (x,y)\right|&\leq&Ce^{-(x^2+y^2)/(4t)}\left(\frac{(xy)^\lambda }{t^{\lambda +3/2}}+\frac{(xy)^{\lambda +2}}{t^{\lambda +7/2}}+\frac{(xy)^\lambda }{t^{\lambda +5/2}}(x^2+y^2)\right)\\
&\leq&Ce^{-c(x^2+y^2)/t}\frac{(xy)^\lambda }{t^{\lambda +3/2}},\quad t,x,y\in (0,\infty )\mbox{ and }xy\leq t,
\end{eqnarray*}
and by (\ref{H58}) we get
\begin{eqnarray*}
\left|\partial _tW_t^\lambda (x,y)\right|&\leq&Ce^{-|x-y|^2/(4t)}\left(\frac{1}{t^{3/2}}+\frac{|x-y|^2}{t^{5/2}}\right)\\
&\leq&C\frac{e^{-c|x-y|^2/t}}{t^{3/2}},\quad t,x,y\in (0,\infty )\mbox{ and }xy\geq t.
\end{eqnarray*}
Hence,
\begin{equation}\label{H59}
\left|\partial _tW_t^\lambda (x,y)\right|\leq Ce^{-c|x-y|^2/t}\frac{(xy)^\lambda }{t^{\lambda +3/2}},\quad t,x,y\in (0,\infty ),
\end{equation}
and then,
\begin{equation}\label{H60}
|K_\lambda ^\gamma (x,y)|\leq C(xy)^\lambda\int_0^\infty \frac{e^{-c|x-y|^2/t}}{t^{\lambda +3/2}}dt\leq C\frac{(xy)^\lambda}{|x-y|^{2\lambda +1}},\quad x,y\in (0,\infty ),\;x\not=y.
\end{equation}
We define the operator $\Delta _\lambda ^{i\gamma}$ on $BMO_{\rm o}(0,\infty )$ as follows: if $f\in BMO_{\rm o}(0,\infty )$ and $\delta >0$,
$$
\Delta _\lambda ^{i\gamma}(f)(x)=\Delta _\lambda ^{i\gamma}(f\chi _{(0,2\delta )})(x)+\int_{2\delta} ^\infty K_\lambda ^\gamma (x,y)f(y)dy,\mbox{ a.e. }x\in (0,\delta ).
$$
This definition does not depend on $\delta$ (in the suitable sense). By (\ref{H60}) we have that, for every $\delta >0$,
\begin{eqnarray*}
\frac{1}{\delta}\int_0^\delta |\Delta _\lambda ^{i\gamma}(f)(x)|dx&\leq&\left(\frac{1}{\delta}\int_0^\delta |\Delta _\lambda ^{i\gamma}(f\chi _{(0,2\delta )})(x)|^2dx\right)^{1/2}+\frac{C}{\delta}\int_0^\delta \int_{2\delta}^\infty \frac{(xy)^\lambda}{|x-y|^{2\lambda +1}}|f(y)|dydx \\
&\leq&\left(\frac{1}{\delta}\int_0^{2\delta }|f(y)|^2dy\right)^{1/2}+\frac{C}{\delta}\int_0^\delta x^\lambda \int_{2\delta }^\infty \frac{|f(y)|}{y^{\lambda+1}}dydx\\
&\leq&C\|f\|_{BMO_{\rm o}(0,\infty )},     \quad f\in BMO_{\rm o}(0,\infty ).
\end{eqnarray*}
Hence, $|\Delta _\lambda ^{i\gamma}(f)(x)|<\infty$, a.e. $x\in (0,\infty )$, for every $f\in BMO _{\rm o}(0,\infty )$.

Let $f\in BMO_{\rm o}(0,\infty )$. We are going to show that $\Delta _\lambda ^{i\gamma}(f)\in BMO(0,\infty )$. It is wellknown that $BMO(0,\infty )=(H^1(0,\infty ))^*$. We denote by $\mathcal{A}(0,\infty )=\mbox{span }\{a\mbox{ is $\infty$-atom satisfying $(Aii)$}\}$. Let $g\in \mathcal{A}$. We choose $\delta >0$ such that $\supp g\subset (0,\delta )$. Then,
$$
\Delta _\lambda ^{i\gamma}(f)(x)=\Delta _\lambda ^{i\gamma}(f\chi _{(0,2\delta )})(x)+\int_{2\delta} ^\infty K_\lambda ^\gamma (x,y)f(y)dy,\mbox{ a.e. }x\in (0,\delta ).
$$
Since $\Delta _\lambda ^{i\gamma}$ is bounded from $L^2(0,\infty )$ into itself, (\ref{H60}) leads to
\begin{eqnarray*}
\int_0^\infty |\Delta _\lambda ^{i\gamma}(f)(x)||g(x)|dx&\leq&\int_0^\delta |\Delta _\lambda ^{i\gamma}(f\chi _{(0,2\delta )}(x)||g(x)|dx\\
&&+\int_0^\delta \int_{2\delta }^\infty |K_\lambda ^\gamma (x,y)||f(y)|dy|g(x)|dx\\
&\leq&\left(\int_0^\delta |\Delta _\lambda ^{i\gamma}(f\chi _{(0,2\delta )}(x)|^2dx\right)^{1/2}\left(\int_0^\delta |g(x)|^2dx\right)^{1/2}\\
&&+C\int_0^\delta \int_{2\delta }^\infty \frac{(xy)^\lambda }{|x-y|^{2\lambda +1}}|f(y)|dy|g(x)|dx\\
&\leq&C\left[\left(\frac{1}{\delta }\int_0^{2\delta}|f(x)|^2dx\right)^{1/2}+\int_0^\delta x^\lambda \int_{2\delta }^\infty \frac{|f(y)|}{y^{\lambda +1}}dydx\right]\\
&\leq&C\|f\|_{BMO_{\rm o}(0,\infty )},
\end{eqnarray*}
because $g\in L^\infty (0,\infty )$.

We define the functional $\mathcal{L}$ on $\mathcal{A}$ by
$$
\mathcal{L}(g)=\int_0^\infty \Delta_\lambda ^{i\gamma }(f)(x)g(x)dx,\quad g\in \mathcal{A}.
$$
Let $g\in \mathcal{A}$ such that $\mbox{supp }g\subset (0,\delta )$ with $\delta >0$. We can write
$$
\mathcal{L}(g)=\int_0^\infty \Delta _\lambda ^{i\gamma}(f\chi _{(0,2\delta )})(x)g(x)dx+\int_0^\infty g(x)\int_{2\delta} ^\infty K_\lambda ^\gamma (x,y)f(y)dydx.
$$
Since, as it was seen,
$$
\int_0^\infty |g(x)|\int_{2\delta }^\infty |K_\lambda ^\gamma (x,y)||f(y)|dydx<\infty,
$$
according to (\ref{H53A}) we get
$$
\int_0^\infty g(x)\int_{2\delta}^\infty K_\lambda ^\gamma (x,y)f(y)dydx=\int_{2\delta }^\infty f(y)\int_0^\delta K_\lambda ^\gamma (x,y)g(x)dxdy=\int_{2\delta }^\infty f(y)\Delta _\lambda ^{i\gamma }(g)(y)dy.
$$
By using Plancherel's equality for Hankel transforms it follows that
$$
\int_0^\infty g(x)\Delta _\lambda ^{i\gamma}(f\chi _{(0,2\delta )})(x)dx=\int_0^{2\delta }f(y)\Delta _\lambda ^{i\gamma }(g)(y)dy.
$$
Hence,
$$
\mathcal{L}(g)=\int_0^\infty f(y)\Delta_\lambda ^{i\gamma }(g)(y)dy.
$$
We  also have that
\begin{eqnarray*}
\int_0^\infty |f(y)\Delta_\lambda ^{i\gamma }(g)(y)|dy&\leq&\int_0^{2\delta}|f(y)||\Delta_\lambda ^{i\gamma }(g)(y)|dy+\int_{2\delta }^\infty |f(y)|\int_0^\delta |K_\lambda ^\gamma (x,y)||g(x)|dxdy\\
&\leq&\left(\int_0^{2\delta }|f(y)|^2dy\right)^{1/2}\|\Delta _\lambda ^{i\gamma}(g)\|_{L^2(0,\infty )}+C\int_{2\delta }^\infty \frac{\delta ^{\lambda +1}}{y^{\lambda +1}}|f(y)|dy\\
&\leq&C(\sqrt{\delta }\|g\|_{L^2(0,\infty )}+1)\|f\|_{BMO_{\rm o}(0,\infty )}\leq C\|f\|_{BMO_{\rm o}(0,\infty )}.
\end{eqnarray*}

According to \cite[Proposition 2.5]{BCCFR2} adapted to this Bessel setting (see also \cite{Bonami}), since $\Delta _\lambda ^{i\gamma}$ is bounded from $H^1_{\rm o}(0,\infty )$ into itself, we deduce
$$
|\mathcal{L}(g)|\leq C\|f\|_{BMO_{\rm o}(0,\infty )}\|\Delta _\lambda ^{i\gamma }(g)\|_{H^1_{\rm o}(0,\infty )}\leq C\|f\|_{BMO_{\rm o}(0,\infty )}\|g\|_{H^1_{\rm o}(0,\infty )}.
$$
Here $C>0$ does not depend on $g$.

We conclude that $\Delta _\lambda ^{i\gamma}(f)\in BMO_{\rm o}(0,\infty )$ and $\|\Delta _\lambda ^{i\gamma}(f)\|_{BMO_{\rm o}(0,\infty )}\leq C\|f\|_{BMO_{\rm o}(0,\infty )}$.
\end{proof}
Suppose now that $(iii)$ holds in the Hardy setting. Let $f\in H_{\rm o}^1(0,\infty )\otimes \mathbb{B}$. According to Proposition \ref{imaginarypower}, $\Delta _\lambda ^{i\gamma}(f)\in H_{\rm o}^1(0,\infty )\otimes \mathbb{B}$ . Then, by using (\ref{H54}) and by taking into account \cite[Theorem 6.2]{Nee} we get
$$
\|\Delta _\lambda ^{i\gamma}(f)\|_{H^1_{\rm o}((0,\infty ),\mathbb{B})}\leq C\|G_\mathbb{B}^{\lambda ,\beta}(\Delta _\lambda ^{i\gamma}(f))\|_{H^1_{\rm o}((0,\infty ),\mathbb{B})}\leq
\|G_\mathbb{B}^{\lambda ,\beta +1}(f)\|_{H^1_{\rm o}((0,\infty ),\mathbb{B})}\leq C\|f\|_{H^1_{\rm o}((0,\infty ),\mathbb{B})}.
$$
According to \cite[Theorem 4.1]{MTX} we deduce that $\Delta _\lambda ^{i\gamma}$ can be extended to $L^p((0,\infty ),\mathbb{B})$ into itself, for every $1<p<\infty$.

Hence, by \cite[Proposition 5.1]{BCR3} we conclude that $\mathbb{B}$ is a UMD space.

By proceeding in a similar way we can prove that $\mathbb{B}$ is a UMD space provided that $(iii)$ holds in the BMO situation.

Thus, the proof of this theorem is finished.

%\newpage
%%%%%%%%%%%%%%%%%%%%%%%%%%%%%%%%%%%%%%%%%%%%%%%%%%%%%%%%%%%%%%%%%%%%%%%%%%%%%%%%%%%%%%%%%%%%%%%%%%%%%%%%%%%%%%%%%%%%
%       REFERENCIAS
%%%%%%%%%%%%%%%%%%%%%%%%%%%%%%%%%%%%%%%%%%%%%%%%%%%%%%%%%%%%%%%%%%%%%%%%%%%%%%%%%%%%%%%%%%%%%%%%%%%%%%%%%%%%%%%%%%%%

%\bibliographystyle{siam}
%\bibliography{references}

\begin{thebibliography}{10}

\bibitem{AST}
{\sc I.~Abu-Falahah, P.~R. Stinga, and J.~L. Torrea}, {\em Square functions
  associated to {S}chr\"odinger operators}, Studia Math., 203 (2011),
  pp.~171--194.

\bibitem{BBFMT}
{\sc J.~J. Betancor, D.~Buraczewski, J.~C. Fari{\~n}a, T.~Mart{\'{\i}}nez, and
  J.~L. Torrea}, {\em Riesz transforms related to {B}essel operators}, Proc.
  Roy. Soc. Edinburgh Sect. A, 137 (2007), pp.~701--725.

\bibitem{BCC1}
{\sc J.~J. Betancor, A.~J. Castro, and J.~Curbelo}, {\em Spectral multipliers
  for multidimensional {B}essel operators}, J. Fourier Anal. Appl., 17 (2011),
  pp.~932--975.

\bibitem{BCCFR2}
{\sc J.~J. Betancor, A.~J. Castro, J.~Curbelo, J.~C. Fari{\~n}a, and
  L.~Rodr{\'i}guez-Mesa}, {\em {$\gamma$}-radonifying operators and
  {UMD}-valued {L}ittlewood-{P}aley-{S}tein functions in the {H}ermite setting
  on {BMO} and {H}ardy spaces}, J. Funct. Anal., 263 (2012), pp.~3804--3856.

\bibitem{BCCFR1}
\leavevmode\vrule height 2pt depth -1.6pt width 23pt, {\em Square functions in
  the {H}ermite setting for functions with values in {UMD}
  spaces}, to appear in Annali di Matematica Pura ed Applicata
  (DOI: 10.1007/s10231-013-0335-9).

\bibitem{BCR3}
{\sc J.~J. Betancor, A.~J. Castro, and L.~Rodr{\'i}guez-Mesa}, {\em Square
  functions and spectral multipliers for {B}essel operators in {UMD}
  spaces}, preprint 2013
 \href{http://arxiv.org/abs/1303.3159}{\texttt{(arXiv:1303.3159v1)}}.

\bibitem{BCS}
{\sc J.~J. Betancor, A.~J. Castro, and P.~Stinga}, {\em The fractional Bessel equation in H\"older spaces},
\newblock preprint 2013.

\bibitem{BCFR2}
{\sc J.~J. Betancor, A.~Chicco~Ruiz, J.~C. Fari{\~n}a, and
  L.~Rodr{\'{\i}}guez-Mesa}, {\em Maximal operators, {R}iesz transforms and
  {L}ittlewood-{P}aley functions associated with {B}essel operators on {BMO}},
  J. Math. Anal. Appl., 363 (2010), pp.~310--326.

\bibitem{BCFR1}
\leavevmode\vrule height 2pt depth -1.6pt width 23pt, {\em Odd {${\rm BMO}(\Bbb
  R)$} functions and {C}arleson measures in the {B}essel setting}, Integral
  Equations Operator Theory, 66 (2010), pp.~463--494.

\bibitem{BCFR}
{\sc J.~J. Betancor, R.~Crescimbeni, J.~C. Fari{\~n}a, and
  L.~Rodr{\'i}guez-Mesa}, {\em Multipliers and imaginary powers of the
  {S}chr\"odinger operators characterizing {UMD} {B}anach spaces},
  Ann. Acad. Sci. Fenn., 38 (2013), 209--227.

\bibitem{BDT}
{\sc J.~J. Betancor, J.~Dziuba{\'n}ski, and J.~L. Torrea}, {\em On {H}ardy
  spaces associated with {B}essel operators}, J. Anal. Math., 107 (2009),
  pp.~195--219.

\bibitem{BFHR}
{\sc J.~J. Betancor, J.~C. Fari{\~n}a, E. Harboure, and L. Rodr\'{\i}guez-Mesa},
{\em Variation operators for semigroups and Riesz transforms on BMO in the Schr\"odinger setting},
 to appear in Potential Analysis (DOI: 10.1007/ s11118-012-9294-9).

\bibitem{BFMT}
{\sc J.~J. Betancor, J.~C. Fari{\~n}a, T.~Mart{\'{\i}}nez, and J.~L. Torrea},
  {\em Riesz transform and {$g$}-function associated with {B}essel operators
  and their appropriate {B}anach spaces}, Israel J. Math., 157 (2007),
  pp.~259--282.

\bibitem{BFS}
{\sc J.~J. Betancor, J.~C. Fari{\~n}a, and A.~Sanabria}, {\em On
  {L}ittlewood-{P}aley functions associated with {B}essel operators}, Glasg.
  Math. J., 51 (2009), pp.~55--70.

\bibitem{BMR}
{\sc J.~J. Betancor, T.~Mart{\'{\i}}nez, and L.~Rodr{\'{\i}}guez-Mesa}, {\em
  Laplace transform type multipliers for {H}ankel transforms}, Canad. Math.
  Bull., 51 (2008), pp.~487--496.

\bibitem{BS2}
{\sc J.~J. Betancor and K.~Stempak}, {\em On {H}ankel conjugate functions},
  Studia Sci. Math. Hungar., 41 (2004), pp.~59--91.

\bibitem{Bonami} {\sc A. Bonami, T. Iwaniec, P. Jones, and M. Zinsmeister}, {\em On the product of functions in $BMO$ and $H^1$},
  Ann. Inst. Fourier, Grenoble, 57 (2007), pp.~1405--1439.

\bibitem{Bou}
{\sc J.~Bourgain}, {\em Some remarks on {B}anach spaces in which martingale
  difference sequences are unconditional}, Ark. Mat., 21 (1983), pp.~163--168.

\bibitem{Buk}
{\sc A.~Bukhalov}, {\em Sobolev spaces of vector-valued functions}, J. Math.
  Sciences, 71 (1994), pp.~2173--2179.

\bibitem{Bu1}
{\sc D.~L. Burkholder}, {\em A geometrical characterization of {B}anach spaces
  in which martingale difference sequences are unconditional}, Ann. Probab., 9
  (1981), pp.~997--1011.

\bibitem{CL}
{\sc T.~Coulhon and D.~Lamberton}, {\em R\'egularit\'e {$L^p$} pour les
  \'equations d'\'evolution}, in S\'eminaire d'{A}nalyse {F}onctionelle
  1984/1985, vol.~26 of Publ. Math. Univ. Paris VII, Paris, 1986, pp.~155--165.

\bibitem{Erdelyi} {\sc A. Erd\'elyi et al.}, {\em Tables of integral transforms}, Volume II, McGraw-Hill, Inc., New York, 1954.

\bibitem{Fri}
{\sc S.~Fridli}, {\em Hardy spaces generated by an integrability condition}, J.
  Approx. Theory, 113 (2001), pp.~91--109.

\bibitem{GLLNU}
{\sc P.~Graczyk, J.~J. Loeb, I.~A. L{\'o}pez~P., A.~Nowak, and W.~O.
  Urbina~R.}, {\em Higher order {R}iesz transforms, fractional derivatives, and
  {S}obolev spaces for {L}aguerre expansions}, J. Math. Pures Appl. (9), 84
  (2005), pp.~375--405.

\bibitem{Gue}
{\sc S.~Guerre-Delabri{\`e}re}, {\em Some remarks on complex powers of
  {$(-\Delta)$} and {UMD} spaces}, Illinois J. Math., 35 (1991), pp.~401--407.

\bibitem{Hy}
{\sc T.~Hyt{\"o}nen}, {\em Littlewood-{P}aley-{S}tein theory for semigroups in
  {UMD} spaces}, Rev. Mat. Iberoam., 23 (2007), pp.~973--1009.

\bibitem{Ka}
{\sc C.~Kaiser}, {\em Wavelet transforms for functions with values in
  {L}ebesgue spaces}, in Wavelets XI, vol.~5914 of Proc. of SPIE, Bellingham,
  WA, 2005.

\bibitem{KaWe}
{\sc C.~Kaiser and L.~Weis}, {\em Wavelet transform for functions with values
  in {UMD} spaces}, Studia Math., 186 (2008), pp.~101--126.

\bibitem{Kw}
{\sc S.~Kwapie{\'n}}, {\em Isomorphic characterizations of inner product spaces
  by orthogonal series with vector valued coefficients}, Studia Math., 44
  (1972), pp.~583--595.

\bibitem{Leb}
{\sc N.~N. Lebedev}, {\em Special functions and their applications}, Dover
  Publications Inc., New York, 1972.

\bibitem{MTX}
{\sc T.~Mart{\'i}nez, J.~L. Torrea, and Q.~Xu}, {\em Vector-valued
  {L}ittlewood-{P}aley-{S}tein theory for semigroups}, Adv. Math., 203 (2006),
  pp.~430--475.

\bibitem{Me1}
{\sc S.~Meda}, {\em A general multiplier theorem}, Proc. Amer. Math. Soc., 110
  (1990), pp.~639--647.

\bibitem{MS}
{\sc B.~Muckenhoupt and E.~M. Stein}, {\em Classical expansions and their
  relation to conjugate harmonic functions}, Trans. Amer. Math. Soc., 118
  (1965), pp.~17--92.

\bibitem{RRT}
{\sc J.~L. Rubio~de Francia, F.~J. Ruiz, and J.~L. Torrea}, {\em
  Calder\'on-{Z}ygmund theory for operator-valued kernels}, Adv. in Math., 62
  (1986), pp.~7--48.

\bibitem{SchSi}
{\sc H.~J. Schmei{\ss}er and W.~Sickel}, {\em Vector-valued {S}obolev spaces
  and {G}agliardo-{N}irenberg inequalities}, in Nonlinear elliptic and
  parabolic problems, vol.~64 of Progr. Nonlinear Differential Equations Appl.,
  Birkh\"auser, Basel, 2005, pp.~463--472.

\bibitem{SW}
{\sc C.~Segovia and R.~L. Wheeden}, {\em On certain fractional area integrals},
  J. Math. Mech., 19 (1969/1970), pp.~247--262.

\bibitem{Ste1}
{\sc E.~M. Stein}, {\em Topics in harmonic analysis related to the
  {L}ittlewood-{P}aley theory}, Annals of Mathematics Studies, No. 63,
  Princeton University Press, Princeton, N.J., 1970.

\bibitem{Ste2}
\leavevmode\vrule height 2pt depth -1.6pt width 23pt, {\em Harmonic analysis:
  real-variable methods, orthogonality, and oscillatory integrals}, vol.~43 of
  Princeton Mathematical Series, Princeton University Press, Princeton, NJ,
  1993.

\bibitem{StemPhD}
{\sc K.~Stempak}, {\em The Littlewood-Paley theory for the Fourier-Bessel
  transform}, PhD thesis, Mathematical Institute University of Wroclaw, Poland,
  1985.

\bibitem{Tor}
{\sc A.~Torchinsky}, {\em Real-variable methods in harmonic analysis}, vol.~123
  of Pure and Applied Mathematics, Academic Press Inc., Orlando, FL, 1986.

\bibitem{TZ}
{\sc J.~L. Torrea and C.~Zhang}, {\em Fractional vector-valued
  {L}ittlewood-{P}aley-{S}tein theory for semigroups}.
\newblock Preprint 2011
  \href{http://arxiv.org/abs/1105.6022v3}{\texttt{(arXiv:1105.6022v3)}}.

\bibitem{Nee}
{\sc J.~van Neerven}, {\em {$\gamma$}-radonifying operators---a survey}, in The
  {AMSI}-{ANU} {W}orkshop on {S}pectral {T}heory and {H}armonic {A}nalysis,
  vol.~44 of Proc. Centre Math. Appl. Austral. Nat. Univ., Canberra, 2010,
  pp.~1--61.

\bibitem{NeVeWe}
{\sc J.~van Neerven, M.C. Veraar, and L. Weis},
{\em Stochastic evolution equations in UMD Banach spaces}, J. Funct. Anal., 255 (2008), pp.~940--993.

\bibitem{Xu}
{\sc Q.~Xu}, {\em Littlewood-{P}aley theory for functions with values in
  uniformly convex spaces}, J. Reine Angew. Math., 504 (1998), pp.~195--226.

\bibitem{Ze}
{\sc A.~H. Zemanian}, {\em A distributional {H}ankel transformation}, SIAM J.
  Appl. Math., 14 (1966), pp.~561--576.

\end{thebibliography}

\end{document}